\documentclass[10pt,letterpaper,twoside]{amsart}
\usepackage{amsthm}
\usepackage{thmtools} % Required for restatable theorems
\usepackage{thm-restate}
\usepackage{graphicx}
\usepackage{xcolor}
\usepackage{marvosym}
\usepackage[margin=0.75in]{geometry}
\usepackage{wasysym}
\usepackage{bbm}
\usepackage{todonotes}
\usepackage{float}
\usepackage{fancyhdr}
\usepackage{enumerate}
\usepackage{enumitem}
\usepackage[utf8]{inputenc}
\usepackage{lastpage}
\usepackage[yyyymmdd,hhmmss]{datetime}
\usepackage{units}
\usepackage{esint}
\usepackage[colorlinks]{hyperref}
\usepackage{upref}
\usepackage{xifthen}
\usepackage{mathtools}
\usepackage{amssymb}
\usepackage[normalem]{ulem}

\newtheoremstyle{note}% name
  {3pt}%      Space above
  {3pt}%      Space below
  {}%         Body font
  {}%         Indent amount (empty = no indent, \parindent = para indent)
  {\itshape}% Thm head font
  {:}%        Punctuation after thm head
  {.5em}%     Space after thm head: " " = normal interword space;
  {}%         Thm head spec (can be left empty, meaning `normal')

\newtheorem{lemma}{Lemma}[section]
\newtheorem{proposition}[lemma]{Proposition}

\newtheorem{remark}[lemma]{Remark}

\theoremstyle{definition}{\newtheorem{definition}[lemma]{Definition}}

\theoremstyle{note}{
\newtheorem*{claim*}{Claim}}

\allowdisplaybreaks
\numberwithin{equation}{section}

\newcommand{\dom}{D}
\newcommand{\sol}{\mathcal{U}}
\makeatletter
\newsavebox{\@brx}
\newcommand{\llangle}[1][]{\savebox{\@brx}{\(\m@th{#1\langle}\)}%
  \mathopen{\copy\@brx\kern-0.5\wd\@brx\usebox{\@brx}}}
\newcommand{\rrangle}[1][]{\savebox{\@brx}{\(\m@th{#1\rangle}\)}%
  \mathclose{\copy\@brx\kern-0.5\wd\@brx\usebox{\@brx}}}
\makeatother

\newcommand{\norm}[1]{\left\| #1 \right\|}
\newcommand{\weakstar}{\overset{\star}\rightharpoonup}

\newcommand{\eps}{\varepsilon}

\renewcommand{\i}{\ifmmode\mathit{\mathchar"7010 }\else\char"10 \fi}
\renewcommand{\j}{\ifmmode\mathit{\mathchar"7011 }\else\char"11 \fi}
\newcommand{\R}{\mathbb{R}}
\newcommand{\N}{\mathbb{N}}

\newcommand{\Grad}{\nabla}
\newcommand{\Div}{\operatorname{div}}
\newcommand{\Curl}{\operatorname{curl}}
\newcommand{\weak}{\rightharpoonup}

\newcommand{\hdiv}{L^2_{\Div}(\dom)}
\newcommand{\V}{H^1_{\Div}(\dom)}
\newcommand{\ths}{\mathbb{D}} % theta function space
\newcommand{\bu}{\mathbf{u}}

\newcommand{\bm}{\mathbf{m}}
\newcommand{\Bu}{\mathbf{U}}

\newcommand{\Bm}{\mathbf{M}}
\newcommand{\Bh}{\mathbf{H}}

\newcommand{\bma}{\mathbf{m}^{N_u,N_\theta,N_m}}
\newcommand{\bha}{\mathbf{h}^{N_u,N_\theta,N_m}}
\newcommand{\bua}{\mathbf{u}^{N_u,N_\theta,N_m}}
\newcommand{\tha}{\theta^{N_u,N_\theta,N_m}}
\newcommand{\pha}{\varphi^{N_u,N_\theta,N_m}}

\newcommand{\bmb}{\mathbf{m}^{N_u,N_\theta}}
\newcommand{\bhb}{\mathbf{h}^{N_u,N_\theta}}
\newcommand{\bub}{\mathbf{u}^{N_u,N_\theta}}
\newcommand{\thb}{\theta^{N_u,N_\theta}}
\newcommand{\phb}{\varphi^{N_u,N_\theta}}
\newcommand{\psib}{\psi^{N_u,N_\theta}}

\newcommand{\bmc}{\mathbf{m}^{N_u}}
\newcommand{\bhc}{\mathbf{h}^{N_u}}
\newcommand{\buc}{\mathbf{u}^{N_u}}
\newcommand{\thc}{\theta^{N_u}}
\newcommand{\phc}{\varphi^{N_u}}
\newcommand{\psic}{\psi^{N_u}}

\newcommand{\Fu}{\mathbf{F}_u}
\newcommand{\Fm}{\mathbf{F}_m}
\newcommand{\Jtheta}{J_\theta}

\newcommand{\bh}{\mathbf{h}}

\newcommand{\bud}{\mathbf{u}^\delta}
\newcommand{\bmd}{\mathbf{m}^\delta} 
 
\newcommand{\bhd}{\mathbf{h}^\delta}

\newcommand{\td}{\theta^\delta}

\newcommand{\normal}{\mathbf{n}}
\newcommand{\lint}{\int_0^T\!\!\!\int_{\dom}\!}

\newcommand{\bmt}{\mathbf{m}^{\tau}}
\newcommand{\bht}{\mathbf{h}^{\tau}}
\newcommand{\but}{\mathbf{u}^{\tau}}
\newcommand{\tht}{\theta^{\tau}}
\newcommand{\pht}{\varphi^{\tau}}
\newcommand{\psit}{\psi^{\tau}}

{%

\begin{enumerate}}%
{\end{enumerate}}

{%

\begin{enumerate}}%
{\end{enumerate}}

\title[Multiphase flow ferrofluids]{Global existence of weak solutions to a two-phase diffuse interface model of  ferrofluids dynamics}
\author[S. Lanthaler]{Samuel Lanthaler}
\address[Samuel Lanthaler]{\newline Faculty of Mathematics \newline University of Vienna \newline Kolingasse 14 - 16,
1090 Vienna, Austria.}
\email[]{samuel.lanthaler@univie.ac.at}
\author[F. Weber]{Franziska Weber}
\address[Franziska Weber]{\newline Department of Mathematics \newline UC Berkeley \newline Berkeley, CA, 94720, USA.}
\email[]{fweber@berkeley.edu}
\thanks{F.W. was supported in part by NSF grant DMS-2438083.}
\date{\today}

\begin{document}

\maketitle
\begin{abstract}
   Ferrofluids are a class of materials that exhibit both fluid and magnetic properties. We consider a two-phase diffuse interface model for the dynamics of ferrofluids on a bounded domain. One phase is assumed to be magnetic, the other phase can be magnetic or non-magnetic. We derive a coupled system of partial differential equations consisting of the incompressible Navier-Stokes equations, an evolution equation for the magnetization, the magnetostatics equations for the magnetic field and the Cahn-Hilliard equations for the evolution of the phase field variable, which are all coupled through various source terms and parameters. In contrast to similar models in the literature, the system in this work formally satisfies an energy balance which remains meaningful even in singular limits such as a limit of zero relaxation time. However, the formal derivation of this balance requires a delicate cancellation of several highly non-linear terms, making it challenging to ensure similar cancellations for approximating systems. Our first main result is to prove the existence of global weak solutions for our ferrofluid system based on a carefully constructed sequence of approximation steps. Additionally, we also study the relaxation towards the quasi-equilibrium, in which case the magnetization equation degenerates to a linear relation between the magnetic field and the magnetization. As our second main result, we prove the rigorous convergence to this limiting system.
  
\end{abstract}

\section{Introduction}
 Ferrofluids, also known as magnetic liquids, are a unique class of materials that exhibit both fluid and magnetic properties. They are colloidal suspensions of nanoscale ferromagnetic particles, typically iron oxides like magnetite ($Fe_3O_4$) or maghemite ($\gamma-Fe_2 O_3$), coated with a surfactant and dispersed in a liquid carrier such as water or oil. This composition prevents the aggregation of the magnetic nanoparticles and allows the fluid to become strongly magnetized in the presence of an external magnetic field~\cite{Oehlsen2022,Odenbach2002}. This ability to be remotely manipulated by magnetic fields has led to a diverse range of technological and biomedical applications, from enhancing the performance of loudspeakers and acting as liquid seals in hard drives to targeted drug delivery and cancer hyperthermia~\cite{AmiriRoodan2020,Raj1995,Scherer2005,Pankhurst2003}.
 
 The fascinating and complex behavior of ferrofluids, particularly the dynamics of their free surface under the influence of a magnetic field, has motivated the development of sophisticated mathematical models to describe and predict their motion. The first such mathematical  models for single-phase flows of ferrofluids
 are due to Rosensweig~\cite{Rosensweig1985,Rosensweig1985a,NeuringerRosensweig1964} and Shliomis~\cite{Shliomis2002}.
 These models are crucial for optimizing existing applications and for the design of new technologies. In this paper, we are concerned with two-phase flows of ferrofluids or a ferrofluid and a non-magnetic fluid. Among the various theoretical frameworks to approach multi-phase flows, diffuse interface models have emerged as a powerful tool. Unlike sharp interface models that treat the boundary between the ferrofluid and its surrounding medium as a distinct surface, diffuse interface models consider a thin transition layer where the properties of the two phases change smoothly. This approach is particularly well-suited for capturing topological changes of the interface, such as droplet coalescence or breakup, which are common in many ferrofluid applications.
 
 This paper focuses on the formulation and analysis of a two-phase flow model of ferrofluids based on a simplified version of the Rosensweig and Shliomis mathematical model for ferrofluids  within a diffuse interface framework. The model couples the Navier-Stokes equations, which govern the fluid flow, with a Cahn-Hilliard type equation to describe the evolution of the phase-field variable that distinguishes the ferrofluid from the ambient fluid. Furthermore, an equation for the evolution of the magnetization within the ferrofluid is included to account for the magnetic body forces. 
 This results in the following  diffuse interface ferrofluids model:
\begin{subequations}\label{eq:CHF}
\begin{align}
    \bu_t+(\bu\cdot\Grad)\bu+\Grad p&=\Div(\nu_\theta T(\bu))+\mu_0(\bm\cdot\Grad)\bh-\eps\lambda\Delta\theta\Grad\theta,\label{eq:nse1}\\
    \Div\bu&=0,\\
    \bm_t+(\bu\cdot\Grad)\bm&=-\frac{1}{\tau}(\bm-\chi_\theta\bh),\label{eq:eqm}\\
    \Curl\bh &=0,\quad \Div(\bm+\bh)=0,\label{eq:maxwell1}\\
    \theta_t+(\bu\cdot\Grad)\theta&=\kappa\Delta\psi,\label{eq:phasefield1}\\
    \psi&= -\eps\lambda\Delta\theta+\frac{\lambda}{\eps}F'(\theta)-\frac{\mu_0}{2\chi_\theta^2}\chi'_\theta|\bm|^2.\label{eq:psi}
\end{align}
\end{subequations}
Here $\bu$ is the velocity of the fluid, $\bm$ is the magnetization, $\bh$ is the magnetic field and $\theta$ is the phase field variable which is equal to $1$ when in the magnetic fluid and equal to $-1$ when in the non-magnetic fluid. $\psi$ is called the chemical potential in this setting. A formal derivation of the system~\eqref{eq:CHF}, based on Onsager's variational principle~\cite{Onsager1931,Onsager1931b}, is included in Appendix \ref{sec:derivation}. 

We continue to explain our notation in further detail. First,
\begin{equation}
	\label{eq:symgradient}
	T(\bu):= \frac{\Grad \bu + (\Grad \bu)^\top}{2},
\end{equation}
is the symmetric part of the gradient. 
The function $F$ is a truncated double well potential, we choose
\begin{equation*}
	F(\theta)= \begin{cases}
		(\theta+1)^2,\quad \theta\in (-\infty,-1),\\
		\frac{1}{4}(1-\theta^2)^2,\quad \theta\in [-1,1],\\
		(\theta-1)^2,\quad \theta\in (1,\infty).
	\end{cases}
\end{equation*}
Clearly, $F$ satisfies
\begin{equation}\label{eq:Fass}
	|F'(\theta)|\leq 2|\theta|+1,\quad |F''(\theta)|\leq 2.
\end{equation}
The functions $\nu_\theta$ and $\chi_\theta$ are the viscosity and susceptibility depending on the phase field variable $\theta$. We assume they are Lipschitz continuous and satisfy,
\begin{equation}\label{eq:viscosity}
	0<\min\{\nu_f,\nu_w\}\leq \nu_\theta\leq \max\{\nu_f,\nu_w\},\quad 0\leq \chi_\theta\leq \chi_0,
\end{equation}
where $\nu_f$ is the viscosity of the ferrofluid and $\nu_w$ the viscosity of the non-magnetic phase and $\chi_0$ the susceptibility of the ferrofluid. We set the susceptiblity of the non-magnetic phase to be zero. They can be chosen in various ways, for example as 
\begin{equation}\label{eq:chidef}
	\nu_\theta= \nu_w +(\nu_f-\nu_w)\mathcal{H}(\theta/\eps),\quad \chi_\theta = \chi_0\mathcal{H}(\theta/\eps),
\end{equation}
 where $\mathcal{H}$ is a regularization of the Heavyside function. To be definite, we will focus on the following choice:
\begin{equation}\label{eq:heavyside}
\mathcal{H}(x) = \frac{1}{\pi} \arctan(x) + \frac 12.
\end{equation}
We point out that our analysis only uses certain decay properties of the tails of $\mathcal{H}(x)$ and $\mathcal{H}'(x)$ as $x\to \pm\infty$. Various other choices for $\mathcal{H}$ are also possible.

The parameters
 $\mu_0,\tau,\varepsilon,\lambda,\kappa >0$ are fixed coefficients: $\tau$ is the relaxation time, $\lambda$ the surface tension, $\varepsilon$ is the interface thickness, $\kappa$ is the mobility, and $\mu_0$ is the magnetic constant.
We assume that the domain $\dom$ is bounded and has a smooth boundary. Additionally, we assume that the domain is simply-connected, in which case we obtain from~\eqref{eq:maxwell1} that $\bh$ is the gradient of a potential, i.e.,  $\bh = \Grad\varphi$ and so~\eqref{eq:maxwell1} becomes
\begin{equation}
	\label{eq:elliptic}
	\bh = \Grad \varphi,\quad -\Delta \varphi = \Div\bm.
\end{equation}
As in~\cite{Nochetto2019,Nochetto2016b,Nochetto2016}, we assume that $\bh$ is composed of a demagnetizing field $\bh_d$ and an applied field $\bh_a$,
\begin{equation}
	\label{eq:decomph}
	\bh = \bh_d+\bh_a,
\end{equation}
where $\bh_a$ is a given smooth harmonic field (not perturbed through $\bh_d$ and $\bm$) and that $\bm=0$ in $\R^d\backslash \dom$ and $\bh_d\equiv 0$ in $\R^d\backslash\dom$. Assuming additionally that the wall of the domain is impermeable, this results in the following boundary conditions:
\begin{subequations}
    \label{eq:boundaryconditionsCahnHilliard}
    \begin{align}
        &\bu=0,\quad x\in\partial\dom,\, t\geq 0,\\
       % \begin{equation}
        &	\frac{\partial\varphi}{\partial n} = (\bh_a - \bm)\cdot \mathbf{n},\quad \int_{\dom} \varphi(t,x) dx =0,\quad x\in\partial\dom,\, t\geq 0, \label{eq:bcelliptic}\\
        %\end{equation}
        & \Grad\theta\cdot\normal =0,\quad x\in\partial\dom,\, t\geq 0,\\
        & \Grad\psi\cdot \normal =0,\quad x\in\partial\dom,\, t\geq 0.
    \end{align}
\end{subequations}
The second boundary condition~\eqref{eq:bcelliptic} is equivalent to
\begin{equation*}
  (\bm+\bh)\cdot\normal = \bh_a\cdot\normal,\quad x\in \partial\dom, \, t\geq 0.
\end{equation*}
This is the model derived and studied in~\cite{Nochetto2016b} with the exception of the last term on the right hand side in the equation for $\psi$. We will show in the following sketch of a derivation of the system in Appendix~\ref{sec:derivation} how it emerges. In contrast to~\cite{Nochetto2016b}, we chose to include it here, because we believe it is necessary for the study of the zero relaxation limit $\tau\to 0$. The zero relaxation limit has practical relevance since in applications $\tau$ is often small, for example~\cite{Rinaldi2002,Shliomis2002} states that it is of order $10^{-9} - 10^{-4}s$. Furthermore, the zero relaxation system, the so-called quasi-equilibrium, is simpler and therefore easier to simulate numerically. Without the additional term, the formal energy balance for the system~\eqref{eq:CHF}
	\begin{equation}\label{eq:ourenergybalance} 
	\int_{\dom} E(\sol)(t) dx +\int_0^t\int_{\dom} \mathcal{D}(\sol)(s)\,dx ds= \int_{\dom} E(\sol)(0) dx +\mu_0 \int_0^t\int_{\partial\dom}\partial_s\bh_a\cdot\varphi \,dSds.
\end{equation}
where the energy $E$ is defined by
\begin{equation}\label{eq:ourenergydensity} 
	E(\sol) = \frac{1}{2}\left(|\bu|^2 +\varepsilon\lambda|\Grad\theta|^2+ \frac{\mu_0}{\chi_\theta}|\bm|^2+\mu_0|\bh|^2+\frac{2\lambda}{\varepsilon}F(\theta)\right),
\end{equation}
and the dissipation functional $\mathcal{D}$ is defined by
\begin{equation*}
	\mathcal{D}(\sol) = \left(\nu_\theta |T(\bu)|^2 +\kappa|\Grad\psi|^2+\frac{\mu_0}{\tau \chi_\theta}|\bm-\chi_\theta\bh|^2 \right),
\end{equation*}
does not include the relaxation term $\int \frac{\mu_0}{\tau \chi_\theta}|\bm-\chi_\theta\bh|^2  dx$ which is needed to send $\tau\to 0$ in a stable manner and conclude that $\bm = \chi_\theta\bh$, i.e., the magnetization becomes parallel to the magnetic field when $\tau=0$, c.f. Section~\ref{sec:relaxation}. In contrast, the only previously known energy balance from~\cite[Proposition 3.1]{Nochetto2016b} is of the form
		\begin{equation}\label{eq:nochettoenergy} 
		\int_{\dom} \widetilde{E}(\sol)(t) dx +\int_0^t\int_{\dom} \widetilde{\mathcal{D}}(\sol)(s)\,dx ds\leq  \int_{\dom} \widetilde{E}(\sol)(0) dx+\mu_0 \int_0^t \left(\tau \norm{\partial_t \bh_a(s)}_{L^2}^2 + \frac{1}{\tau}\norm{\bh_a(s)_{L^2}^2}\right)ds,
	\end{equation}
 where $\widetilde{E}$ and $\widetilde{\mathcal{D}}$ are given by
 \begin{equation*} 
 	\widetilde{E}(\sol) = \frac{1}{2}\left(|\bu|^2 +\lambda|\Grad\theta|^2+ {\mu_0}|\bm|^2+\mu_0|\bh|^2+\frac{2\lambda}{\varepsilon^2}F(\theta)\right),
 \end{equation*}
and
 \begin{equation*}
 	\widetilde{\mathcal{D}}(\sol) =  \nu_\theta |T(\bu)|^2 +\frac{\kappa\lambda}{\varepsilon}|\Grad\psi|^2+\frac{\mu_0}{\tau }\left(1-\frac{\chi_0}{4}\right)|\bm|^2+\frac{\mu_0}{2\tau}|\bh|^2,
 \end{equation*}
which would require additional assumptions on $\bh_a$ in order to send $\tau\to 0$ due to the different source term on the right hand side of~\eqref{eq:nochettoenergy}.
 We also note that the division by $\chi_\theta$ in the magnetization energy density in~\eqref{eq:ourenergydensity} heuristically forces $\bm$ to be small when $\chi_{\theta}$ is small as it is the case if the fluid is mostly consisting of the non-magnetic phase. However, the additional term
\begin{equation*}
	\frac{\mu_0}{2\chi_\theta^2}\chi'_\theta|\bm|^2
\end{equation*}
causes substantial challenges in the analysis of the system~\eqref{eq:CHF} due to its nonlinearlity. The main purpose of this work is to overcome these technical challenges. As a result, we prove the existence of global weak solutions of the system~\eqref{eq:CHF} using a Galerkin approximation scheme. We first approximate $\chi_\theta$ by a nonvanishing function $\chi_\theta^\delta$. Then we introduce three Galerkin truncations for $\bu$, $\theta$ and $\bm$ respectively that we send to infinity at different points in the approximation scheme, which is necessary to handle the nonlinearities of the system and obtain an approximate energy balance similar to the one expected for the system,~\eqref{eq:ourenergybalance}. Finally, we send the approximation $\delta>0$ in the susceptibility, $\chi_\theta$ to zero to obtain the one in~\eqref{eq:chidef}, \eqref{eq:heavyside}. In order to deal with the weak a priori regularity for the variables $\bm$ and $\bh$ expected from the energy balance, i.e., $\bm,\bh\in L^\infty(0,T;L^2(\dom))$ and derive compactness of the approximating sequences, we adapt an argument from~\cite{Nochetto2019} using renormalized solutions to our setting to obtain compactness of those approximating sequences in $L^2([0,T]\times\dom)$. We use a similar argument to prove convergence of solutions of~\eqref{eq:CHF} to solutions of the zero relaxation system, which results when setting $\tau=0$ formally in~\eqref{eq:CHF}.

\begin{remark}
	One could also consider the Allen-Cahn version of the model~\eqref{eq:CHF} above:
	\begin{subequations}\label{eq:ACF}
		\begin{align}
			\bu_t+(\bu\cdot\Grad)\bu+\Grad p&=\Div(\nu_\theta T(\bu))+\mu_0(\bm\cdot\Grad)\bh-\eps\lambda\Delta\theta\Grad\theta,\label{eq:nse2}\\
			\Div\bu&=0,\\
			\bm_t+(\bu\cdot\Grad)\bm&=-\frac{1}{\tau}(\bm-\chi_\theta\bh),\\
			\Curl\bh &=0,\quad \Div(\bm+\bh)=0,\label{eq:maxwell2}\\
			\theta_t+(\bu\cdot\Grad)\theta&=-\kappa\psi,\label{eq:phasefield2}\\
			\psi&= -\lambda\eps\Delta\theta+\frac{\lambda}{\eps}F'(\theta)-\frac{\mu_0}{2\chi_\theta^2}\chi'_\theta|\bm|^2.
		\end{align}
	\end{subequations}
	We believe that the results of this article carry over to this system of equations also.
\end{remark}

\subsection{Related works}
As already mentioned above, the work most closely related to this is~\cite{Nochetto2016b} where the diffuse interface model for ferrofluids is derived (with the exception of the last term in equation~\eqref{eq:psi}, as explained above). The same article proposes a stable fully discrete numerical scheme and simulations of the Rosensweig instability. Furthermore, they present a convergence proof for the case that $\bh$ is a given function (i.e. remove equation~\eqref{eq:maxwell1} above), which implies existence of global weak solutions for that case. Later works are concerned with the design of linearly-implicit energy stable schemes for similar two-phase flow models~\cite{Keram2025,Zhang2021,Zhang2024}, however, these works do not provide convergence proofs for their schemes and therefore do not yield proofs of existence of weak solutions of the corresponding models.
Other mathematical works are dedicated to single-phase flows of ferrofluids. In~\cite{Nochetto2019} global existence of weak solutions and relaxation to the equilibrium is shown for the single-phase Rosensweig model of ferrohydrodynamics.  In~\cite{Nochetto2016}, a stable numerical scheme is designed for the single-phase flow Rosensweig system. Other analysis works such as~\cite{Amirat2008,Scrobogna2019} consider regularized systems where an additional diffusion term is added in the magnetization equation~\eqref{eq:eqm}. However, it is unclear if this term can be justified from the physics of ferrofluids and it does not appear in the original works by Rosensweig and Shliomis either. The existence of local-in-time strong solutions of the single-phase Shliomis model without additional regularization was proved in~\cite{Amirat2009} and the existence of local-in-time strong solutions of the single-phase Rosensweig model was proved in~\cite{Amirat2010}.
Therefore, we believe our work is the first rigorous study of global weak solutions for a two-phase flow ferrofluid model without additional regularization terms. 
\subsection{Outline of this article}
The rest of this article is organized as follows: In Section~\ref{S2}, we present the definition of weak solutions, the main results and some auxiliary identities and preliminary results. In Section~\ref{sec:existenceproof}, the existence of global weak solutions is proved. In Section~\ref{sec:relaxation}, we send the relaxation time $\tau\to 0$ and prove the convergence of weak solutions of~\eqref{eq:CHF} to weak solutions of the limiting system. We end with an appendix containing the formal derivation of the system~\eqref{eq:CHF} and a real analysis result.

\section{Main results}\label{S2}
\subsection{Notation and definition of global weak solutions}
We start by introducing necessary notation: We denote by $\hdiv$ the space of divergence free $L^2(\dom)$ functions and by $\V$ the space of divergence free functions in $H^1_0(\dom)$ (these can be obtained as the closures of $C^\infty_0(\dom)\cap \{\Div u=0\}$ in $L^2(\dom)$ and $H^1_0(\dom)$ respectively (c.f.~\cite{Temam2001}):
\begin{equation*}
	\begin{split}
		\hdiv &= \{u\in L^2(\dom);\, \Div u = 0,\,  \text{ a.e. in }\, \dom,\, u\cdot n|_{\partial \dom}=0\}, \\
		\V &=\{u\in H^1(\dom);\, \Div u = 0,\,  \text{ a.e. in }\, \dom,\, u|_{\partial \dom}=0\}. 
	\end{split}
\end{equation*}
For a Banach space $X$ we let $C_w(0,T;X)$ be the space of functions that are weakly continuous in time, that is, if $v\in C_w(0,T;X)$, then for any $s\to t$,
\begin{equation*}
	\langle v(s),g\rangle\to \langle v(t), g\rangle,\quad \text{for } s\to t, \quad \forall \, g\in X^*,
\end{equation*}
where we denoted by $X^*$ the dual space of $X$ and by $\langle\cdot,\cdot\rangle$ the dual product on $X,X^*$.
We use $(\cdot,\cdot)$ to denote the $L^2(\dom)$-inner product,
\begin{equation*}
	(f,g):= \int_{\dom} f(x)\cdot g(x) dx.
\end{equation*}
Now we can define a notion of weak solution for the diffuse interface ferrofluids system~\eqref{eq:CHF}:
\begin{definition}[Distributional solution of~\eqref{eq:CHF}]
	\label{def:weaksol}
	Let $T>0$, $\dom\subset\R^d$, $d=2,3$, a smooth, simply connected domain.   $\sol:=(\bu,\theta,\bm,\bh,\psi)$ is a global weak solution of the system \eqref{eq:CHF} if the following conditions are satisfied:
	\begin{enumerate}[label=(\roman*)]
		\item {\textbf{Regularity.}} The solution $\sol:=(\bu,\theta,\bm,\bh,\psi)$ satisfies the regularity requirements:
		\begin{align*}
			\bu&\in L^\infty(0,T;\hdiv)\cap L^2(0,T;\V)\cap C_w(0,T;L^2(\dom))\\
			\theta & \in L^\infty(0,T;H^1(\dom)) \cap C(0,T;L^2(\dom))\\
			\bm & \in L^\infty(0,T;L^2(\dom))\cap C_w(0,T;L^2(\dom))\\
			\bh & \in L^\infty(0,T;L^2(\dom))\cap C_w(0,T;L^2(\dom))\\
			\psi & \in L^2(0,T;H^1(\dom));
		\end{align*}
		\item \textbf{Magnetostatics equations.} The function $\bh$ is such that $\bh=\Grad \varphi$, where $\varphi\in L^\infty(0,T;H^1(\dom))$ and satisfies for all $\phi\in H^1(\dom)$, and a.e. $t\in [0,T]$,
		\begin{equation}\label{eq:ellipticeqforvphi}
			\int_{\dom} (\Grad\varphi+\bm) \cdot\Grad\phi dx =\int_{\partial\dom} \bh_a\cdot\mathbf{n}  \phi\, dS,
		\end{equation}
		with $\int_{\dom} \varphi dx =0$. 
		\item \textbf{Weak form of evolution equations.} Equations \eqref{eq:nse1}--\eqref{eq:psi} hold weakly, that is for any test functions $v_i\in C^1_c([0,T)\times\dom)$ $i=1,\dots,5$ with vanishing trace on $\partial\dom$ and $\Div v_1(t,x) = 0$ for all $(t,x)\in[0,T]\times\dom$, 
		\begin{align}
			&\lint \left[\bu\cdot\partial_tv_1+((\bu\cdot\Grad)v_1)\cdot\bu -\nu_\theta T(\bu):T (v_1) \right]dx dt + \int_{\dom} \bu_0\cdot v_1(0,x) dx\label{eq:nseweak}\\ 
			&\qquad=\lint\left[ \mu_0\left(\left(((\bm+\bh)\cdot\Grad)v_1\right)\cdot\bh \right) - \varepsilon \lambda(\Grad\theta\otimes\Grad\theta):\Grad v_1\right] dx dt\notag\\
			&\int_{\dom}\bu(t,x)\cdot\Grad v_2(t,x) dx = 0,\quad \text{a.e. } t\in [0,T]\label{eq:nsedivconstraint}\\
			&\lint \left[\bm\cdot\partial_t v_3+(\bu\cdot\Grad) v_3\cdot\bm \right]dx dt + \int_{\dom} \bm_0\cdot v_3(0,x) dx=\lint\left[\frac{1}{\tau}(\bm-\kappa_\theta\bh)\cdot v_3\right] dx dt\label{eq:mweak}\\
			&\lint \left[ \theta\cdot \partial_t v_4 + (\bu\cdot \Grad)v_4\cdot \theta\right] dx dt +\int_{\dom}\theta_0 v_4(0,x)dx = \kappa \lint  \Grad\psi\cdot \Grad v_4  dx dt\\
			&\lint \psi v_5 dx dt = \lint \left[\varepsilon \lambda\Grad\theta\cdot \Grad v_5 + \frac{\lambda}{\varepsilon} F'(\theta)v_5- \frac{\mu_0}{2\chi_\theta^2}\chi_\theta' |\bm|^2 v_5\right] dx dt.
		\end{align}
		where $\bu_0\in\hdiv$, $\bm_0\in L^2(\dom)$ and $\theta_0\in H^1(\dom)$ are the initial conditions;
		
		\item {\textbf{Energy balance.}} The solution $\sol:=(\bu,\theta,\bm,\bh,\psi)$ satisfies the following energy inequality,
		\begin{equation}\label{eq:energy0}
			\int_{\dom} E(\sol)(t) dx +\int_0^t\int_{\dom} \mathcal{D}(\sol)(s)\,dx ds\leq \int_{\dom} E(\sol)(0) dx +\mu_0 \int_0^t\int_{\partial\dom}\partial_s\bh_a\cdot\varphi \,dSds.
		\end{equation} 
		where the energy density $E$ is defined by
		\begin{equation}\label{eq:defE}
			E(\sol) = \frac{1}{2}\left(|\bu|^2 +\varepsilon\lambda|\Grad\theta|^2+ \frac{\mu_0}{\chi_\theta}|\bm|^2+\mu_0|\bh|^2+\frac{2\lambda}{\varepsilon}F(\theta)\right),
		\end{equation}
		and the dissipation functional $\mathcal{D}$ is defined by
		\begin{equation}\label{eq:defD}
			\mathcal{D}(\sol) = \left(\nu_\theta |T(\bu)|^2 +\kappa|\Grad\psi|^2+\frac{\mu_0}{\tau \chi_\theta}|\bm-\chi_\theta\bh|^2 \right).
		\end{equation}
	\end{enumerate}
\end{definition}

\subsection{Main theorems}
The main result of this work is the following existence theorem for global weak solutions to the system \eqref{eq:CHF}: 
\begin{restatable}{mainthm}{existence}\label{thm:existence}
	Let $T>0$ be a given final time. Assume that $\bu_0\in \hdiv$, $\theta_0\in H^1(\dom)$, $\bm_0\in L^p(\dom)$ for $p>12/5$ and $\bh_a\in W^{1,\infty}(0,T;L^2(\partial\dom))$. Assume furthermore that the viscosity $\nu_\theta$ satisfies~\eqref{eq:viscosity} and the susceptibility $\chi_\theta$ is given by~\eqref{eq:chidef}. Assume that $F$ is smooth and that $F'$ grows at most linearly for large $\theta$, i.e., $|F'(\theta)|\leq C(1+|\theta|)$ for some $C>0$. Then there exists a weak solution of~\eqref{eq:CHF}, in the sense of Definition~\ref{def:weaksol}.
\end{restatable}

We will prove this theorem using a sequence of approximation steps, detailed in Section \ref{sec:existenceproof}. First, since the equations as well as the energy involve divisions by $\chi_\theta$, c.f. equation~\eqref{eq:psi}, we regularize $\chi_\theta$ by a parameter $\delta$ to make it uniformly positive. Then we show existence of weak solutions under this uniform positivity assumption for $\chi_\theta$ using Galerkin truncations for the variables $\bu$, $\theta$ and $\bm$. 

The main technical difficulty is that, even though the continuous system \eqref{eq:CHF} formally satisfies the energy balance \eqref{eq:energy0} which would imply boundedness and compactness properties of the system, deriving this balance rigorously requires an exact cancellation of several highly nonlinear terms. We have not been able to find a discretization scheme which respects this structure, and hence no similar energy balance is available at the discretized level. 

To circumvent this core difficulty, we instead consecutively send our Galerkin truncations to infinity. Each time a limit is taken, this allows us to obtain improved estimates because one or more of the equations in the PDE system \eqref{eq:CHF} is now satisfied at the continuous level, implying exact cancellation of relevant nonlinear terms. We start with the limit for the approximation in $\bm$, based on which we derive an approximate energy inequality for the limiting system. In turn, this gives sufficient control to send the Galerkin truncation for $\theta$ to infinity. Finally, we pass to the limit in the approximation for $\bu$ to obtain a weak solution of the system in the case that $\chi_\theta$ is uniformly positive. This part of the argument goes through for $\bm_0\in L^2(\dom)$.

Section~\ref{sec:vanishing-mag} is then dedicated to sending the approximation parameter $\delta$ in $\chi_\theta$ to zero and obtain existence of a weak solution for $\chi_\theta$ given in~\eqref{eq:chidef},~\eqref{eq:heavyside}. It is only in this final part of the argument, we require the slightly stronger technical assumption that $\bm_0\in L^p(\dom)$ for some $p>12/5$, to deal with the potentially singular behavior of $1/\chi_\theta$, when $\theta \to -\infty$.

\subsubsection{Limit of zero relaxation time $\tau \to 0$.} In applications, the relaxation time $\tau$ is generally very small~\cite{Rinaldi2002,Shliomis2002}. Therefore, one might want to consider the equations~\eqref{eq:CHF} with $\tau=0$. 
Formally, one then obtains the system
\begin{subequations}\label{eq:CHFtau}
	\begin{align}
		\Bu_t+(\Bu\cdot\Grad)\Bu+\Grad P&=\Div(\nu_\Theta T(\Bu))+\mu_0(\Bm\cdot\Grad)\Bh-\eps\lambda\Delta\Theta\Grad\Theta,\label{eq:nse1T}\\
		\Div\Bu&=0,\\
		0&=\Bm-\chi_\Theta\Bh,\label{eq:eqmT}\\
		\Curl\Bh &=0,\quad \Div(\Bm+\Bh)=0,\label{eq:maxwellT}\\
		\Theta_t+(\Bu\cdot\Grad)\Theta&=\kappa\Delta\Psi,\label{eq:phasefield1T}\\
		\Psi&= -\eps\lambda\Delta\Theta+\frac{\lambda}{\eps}F'(\Theta)-\frac{\mu_0}{2 }\chi'_\Theta|\Bh|^2.\label{eq:psiT}
	\end{align}
\end{subequations}
Our second main result is a consequence of Theorem \ref{thm:existence}, and makes this limit mathematically rigorous. We denote the solution of~\eqref{eq:CHF} as $\mathcal{U}^\tau=(\but,\tht,\bmt,\bht,\psit)$ to make its dependence on $\tau$ explicit for the rest of this subsection. Then, we can prove:
\begin{restatable}{mainthm}{tauzerolimit}
	\label{thm:tauzerolimit}
	Assume that
$\bu_0\in \hdiv$, $\bm_0\in L^2(\dom)$, $\theta_0\in H^1(\dom)$ and $\bh_a\in W^{1,\infty}(0,T;L^2(\partial\dom))$. Assume furthermore that the viscosity $\nu_\theta$ satisfies~\eqref{eq:viscosity} and the susceptibility $\chi_\theta$ satisfies~\eqref{eq:assumptions}. Assume that $F$ is smooth and that $F'$ grows at most linearly for large $\theta$, i.e., $|F'(\theta)|\leq C(1+|\theta|)$ for some $C>0$. Then, as $\tau\to 0$ a subsequence of weak solutions $\mathcal{U}^\tau=(\but,\tht,\bmt,\bht,\psit)$ of~\eqref{eq:CHF} as in Definition~\ref{def:weaksol} converges to a weak solution of~\eqref{eq:CHFtau}.
For $\but$, $\bht$ and $\bmt$, the convergence is strong in $L^2([0,T]\times\dom)$, for $\tht$ the convergence is strong in $L^2(0,T;H^1(\dom))\cap C(0,T;L^2(\dom))$, and for $\psit$ the convergence is weak in $L^2(0,T;H^1(\dom))$.
\end{restatable}
We prove this result in Section~\ref{sec:vanishing-mag}. The proof uses some of the identities derived in the course of the proof of the existence Theorem~\ref{thm:existence}, as we have to pass to the limit in many of the same terms. The proof also relies on additional tools that were developed in the earlier work~\cite{Nochetto2019}.

\subsection{Useful identities and preliminary results}
Next, we state some useful identites and preliminary results which will be used in the proofs of Theorem~\ref{thm:existence} and~\ref{thm:tauzerolimit}.
\subsubsection{Energy inequality}
	Since $\bh_a$ is harmonic, we can rewrite the last term $\int_0^t\int_{\partial\dom} \partial_s \bh_a \cdot\varphi dS ds$ in the energy inequality~\eqref{eq:energy0} using Gauss' divergence theorem
	\begin{align*}
		\int_0^t\int_{\partial\dom} \partial_s \bh_a \cdot\varphi dS ds& = \int_0^t \int_{\dom} \Div (\partial_s \bh_a \varphi) dx ds\\
		&= \int_0^t\int_{\dom} (\Div(\partial_s \bh_a) \varphi + \partial_s \bh_a\cdot\Grad \varphi )dx ds \\
		&= \int_0^t \int_{\dom} \partial_s \bh_a \cdot\Grad\varphi dx ds \\
		&= \int_0^t\int_{\dom}\partial_s \bh_a \cdot\bh dx ds.
	\end{align*}
Then, we can use Young's inequality to bound the term as
\begin{equation*}
	\left| 	\int_0^t\int_{\partial\dom} \partial_s \bh_a \cdot\varphi dS ds  \right|\leq \frac{\mu_0}{4}\int_0^t\norm{\bh(s)}_{L^2}^2 ds +\frac{1}{\mu_0}\int_0^t \norm{\partial_s\bh_a(s)}_{L^2}^2 ds.
\end{equation*}
Next, we use
 Gr\"onwall's inequality in~\eqref{eq:energy0}, to obtain the regularity for the solution required in~\ref{def:weaksol} (i) after using that $\theta$ and $\psi$ have bounded mass and Korn's inequality~\cite{Korn1909} for the symmetric gradient. In the same way, we can rewrite the weak formulation~\eqref{eq:ellipticeqforvphi} as
 \begin{equation}
 	\label{eq:ellipticforvphi2}
 		\int_{\dom} (\Grad\varphi+\bm) \cdot\Grad\phi dx =\int_{ \dom} \bh_a\cdot \Grad \phi\, dx.
 \end{equation}

\begin{remark}\label{rem:energysmoothsol}
	For smooth solutions, the energy inequality~\eqref{eq:energy0} in fact becomes an equality. It can be obtained by taking the inner product of equation~\eqref{eq:nse1} with $\bu$, equation~\eqref{eq:eqm} with $(\frac{\mu_0}{\chi_\theta}\bm-\mu_0\bh)$ and the time-differentiated equation~\eqref{eq:maxwell1} with $-\mu_0\varphi$, equation~\eqref{eq:phasefield1} with $\psi$ and equation~\eqref{eq:psi1} with $-\theta_t$, adding all of them, integrating over $\dom$, and then integrating by parts where suitable.
\end{remark}

We will need the following auxiliary results that will be useful later in the proof of existence of weak solutions.

\subsubsection{Kelvin force}
The term $\mu_0 (\bm\cdot\Grad)\bh$ in~\eqref{eq:nse1}, ~\eqref{eq:nse2} is called Kelvin force. It can be rewritten as follows using equation~\eqref{eq:maxwell1}, 
\begin{lemma}\label{lem:integrableweaksol}\cite[Lemma 2.3]{Nochetto2019}
	Any sufficiently smooth solution of \eqref{eq:CHF} satisfies
	\begin{equation}\label{eq:kelvin}
		(\bm\cdot\Grad)\bh = \Div\left((\bm+\bh)\otimes \bh \right) -\frac{1}{2}\Grad \left(|\bh|^2\right).
	\end{equation}
\end{lemma}
The right hand side is the version that is used in the definition of weak solutions since the regularity of the weak solution does not allow for derivatives of $\bh$.
%\end{remark}
%\begin{remark}
\subsubsection{Capillary force}
	The capillary force term $\varepsilon \Delta \theta\Grad\theta$ in~\eqref{eq:nse1} can be (formally) rewritten in several ways, using~\eqref{eq:psi}. First, we use chain rule to rewrite it as  
	\begin{align}\label{eq:capillary}
		\varepsilon\Delta \theta \Grad\theta & = \varepsilon \left(\sum_{j=1}^d \partial_{jj} \theta\partial_i\theta\right)_{i=1}^d\\
		& = \varepsilon  \left(\sum_{j=1}^d \left(\partial_j(\partial_j\theta\partial_i \theta) - \partial_j\theta\partial_j\partial_i \theta \right)\right)_{i=1}^d\notag\\
		& = \varepsilon  \left(\sum_{j=1}^d \left(\partial_j(\partial_j\theta\partial_i \theta) -\frac12 \partial_i |\partial_j\theta|^2  \right)\right)_{i=1}^d\notag\\
		&  = \varepsilon  \Div (\Grad\theta\otimes\Grad \theta) - \frac{1}{2}\Grad |\Grad\theta|^2. \notag
	\end{align}
	Testing this with a divergence free test function and integrating by parts, we obtain the expression used in~\eqref{eq:nseweak}. On the other hand, replacing $\Delta \theta$ by~\eqref{eq:psi}, we calculate
	\begin{align*}
		\varepsilon\Delta \theta \Grad\theta &= -\psi \Grad\theta + \frac{\lambda}{\varepsilon}F'(\theta)\Grad\theta -\frac{\mu_0}{2\chi_\theta^2}\chi_\theta' |\bm|^2 \Grad\theta\\
		& = -\Grad(\psi\theta)+\Grad\psi \theta +\frac{\lambda}{\varepsilon}\Grad F(\theta)+\frac{\mu_0}{2} |\bm|^2 \Grad\left(\frac{1}{\chi_\theta}\right).
	\end{align*}
	The first and the third term can be absorbed into the pressure since we are using divergence free test functions and so the capillary force term in~\eqref{eq:nse1} can also be replaced by
	\begin{equation*}
		\Grad\psi \theta  +\frac{\mu_0}{2} |\bm|^2 \Grad\left(\frac{1}{\chi_\theta}\right).
	\end{equation*}
	
\section{Existence of Weak Solutions: Proof of Theorem \ref{thm:existence}}
\label{sec:existenceproof}
We note that the equations \eqref{eq:CHF} involve division by $\chi_\theta$. To avoid potential singularities in our proof of Theorem \ref{thm:existence}, we will thus first consider a uniformly lower bounded approximation $\chi^\delta_\theta \ge \delta >0$ of $\chi_\theta$, and prove existence of weak solutions for fixed $\delta$. This is the subject of the next subsection, which will occupy most of this paper. Once existence is shown for $\chi^\delta_\theta$, we will then consider the limit $\delta \to 0$. Passage to $\delta \to 0$ is the subject of subsection \ref{sec:vanishing-mag}.

\subsection{Existence of global weak solutions for non-vanishing susceptibility}

In order to show existence of global weak solutions of~\eqref{eq:CHF}, we need to construct a suitable approximation scheme. To do so, let us assume that $\chi$ is smooth and satisfies the following properties:
\begin{equation}\label{eq:assumptions}
	0<\delta\leq \chi_\theta\leq \chi_0<\infty,\quad |\chi_\theta'|\leq C<\infty,\quad \lim_{\theta\to\pm\infty}\chi_\theta' = 0,\quad \lim_{\theta\to\pm\infty}|\theta\chi'_\theta|\leq C<\infty.
\end{equation}
For example, this is achieved by setting
\begin{equation}\label{eq:susceptibility}
	\chi_\theta = \delta+(\chi_0-\delta)\mathcal{H}(\theta/\eps),
\end{equation}
for some $0<\delta<\chi_0$ where $\mathcal{H}$ is our regularization of the Heavyside function~\eqref{eq:heavyside}.  For relevant estimates that imply \eqref{eq:assumptions}, we refer to the discussion following the statement of Proposition \ref{prop:existence}, below.  
\begin{proposition}\label{prop:existence}
	Let $T>0$ be a given final time. Assume that $\bu_0\in \hdiv$, $\bm_0\in L^2(\dom)$, $\theta_0\in H^1(\dom)$ and $\bh_a\in W^{1,\infty}(0,T;L^2(\partial\dom))$. Assume furthermore that the viscosity $\nu_\theta$ satisfies~\eqref{eq:viscosity} and the susceptibility $\chi_\theta$ satisfies~\eqref{eq:assumptions}. Assume that $F$ is smooth and that $F'$ grows at most linearly for large $\theta$, i.e., $|F'(\theta)|\leq C(1+|\theta|)$ for some $C>0$. Then, there exists a weak solution of~\eqref{eq:CHF} as in the sense of Definition~\ref{def:weaksol}.
\end{proposition}
The only difference between Proposition \ref{prop:existence} and Theorem \ref{thm:existence} are the uniform bounds \eqref{eq:assumptions}.
We will prove this proposition using a sequence of approximation steps in the next couple of subsections. Before coming to the proof, we briefly discuss the relevance of the properties \eqref{eq:assumptions}.

\subsubsection{Estimates for the Heavyside approximation $\mathcal{H}$}%
We now consider our prototypical choice \eqref{eq:heavyside}, i.e., $\mathcal{H}(x) = \frac{1}{\pi} \arctan(x) + \frac{1}{2}$,
with the aim of deriving relevant estimates on $\chi^\delta_\theta = \delta  + (\chi_0 - \delta) \mathcal{H}(\theta/\epsilon)$.  Then we have:
\begin{lemma}
\label{lem:heavyside}
Consider $\chi_\theta = \chi^\delta_\theta$, given by \eqref{eq:susceptibility}. There exists a constant $C>0$, independent of $\delta>0$, such that the properties in \eqref{eq:assumptions} hold. Furthermore, $\chi_\theta = \chi^\delta_\theta$ also satisfies the following bounds,
\begin{equation}\label{eq:heavysidelemma}
\left| \frac{\chi_\theta'}{\chi_{\theta}}\right|\leq C, 
\qquad  \left| \frac{\chi_\theta' \theta}{\chi_{\theta}}\right|\leq C,
\quad \left|\frac{\chi_\theta'}{\chi_{\theta}^2}\right|\leq C \quad \forall  \theta\in\R.
\end{equation}
uniformly in $\delta>0$.
\end{lemma}

\begin{proof}
First, we note that $\mathcal{H}$ is bounded, $0\le \mathcal{H}(x) \le 1$, and therefore we have $\delta \le \chi_\theta \le \chi_0$. Moreover, $\mathcal{H}(x)$ has asymptotics $\mathcal{H}(x) \to 1$ as $x\to \infty$ and 
\begin{equation}
\label{eq:asymp-inf}
\mathcal{H}(x) = \frac{1}{\pi} \arctan(x) + \frac{1}{2} 
= - \frac{1}{\pi x} + O\left(\frac{1}{x^3}\right), \quad x\to -\infty.
\end{equation}
We also note that
\begin{equation}
\label{eq:hderiv}
\mathcal{H}'(x) = \frac{1}{\pi} \frac{1}{1+x^2}.
\end{equation}
The last identity \eqref{eq:hderiv} and the definition of $\chi_\theta = \chi^\delta_\theta$ yields
\[
|\chi_\theta'| 
\le \frac{\chi_0}{\varepsilon} \mathcal{H}'(\theta/\epsilon)
= \frac{\chi_0}{\pi\varepsilon \left[1+(\theta/\varepsilon)^2\right]},
\]
where the right-hand side is independent of $\delta$. This immediately implies the properties \eqref{eq:assumptions}. Similarly, we see that
\begin{align*}
\left|
\frac{(\chi_{\theta})'}{\chi_{\theta}}
\right|
&\le
 \frac{1}{\varepsilon} \frac{1}{\left[ 1+(\theta/\varepsilon)^2 \right]} \frac{1}{\left[\arctan(\theta/\varepsilon) + \pi/2\right]},
\end{align*}
is uniformly bounded by the function on the right-hand side for any $\delta>0$. Furthermore, we observe that $(\chi_\theta)'/\chi_\theta \to 0$ as $\theta \to +\infty$, and
\[
\left|
\frac{(\chi_{\theta})'}{\chi_{\theta}}
\right|
\le \frac{1}{\varepsilon}  \frac{1}{\left[ 1+(\theta/\varepsilon)^2 \right]} \frac{1}{\left[\varepsilon/|\theta| + O(\varepsilon^3/|\theta|^3)\right]} 
= 
O\left( \frac{1}{|\theta|} \right)
\to 0, \quad \text{as } \theta \to -\infty.
\]
This implies the second estimate in~\eqref{eq:heavysidelemma}.
For the last estimate in~\eqref{eq:heavysidelemma}, we calculate,
\[
\left|
\frac{(\chi_{\theta})'}{\chi_{\theta}^2}
\right|
\le \frac{1}{\varepsilon}  \frac{1}{\left[ 1+(\theta/\varepsilon)^2 \right]} \frac{1}{\left[\varepsilon/|\theta| + O(\varepsilon^3/|\theta|^3)\right]^2} 
= 
O\left( 1 \right),\quad \text{as } \, \theta\to-\infty,
\]
whereas if $\theta>0$, since $\chi_\theta$ is increasing, and $\chi'_\theta$ is bounded, this is clearly satisfied.
This concludes our proof.
\end{proof}

\subsubsection{Detailed outline of the proof of Proposition \ref{prop:existence}}
As pointed out in the introduction, our existence proof relies on a Galerkin truncation of the system \eqref{eq:CHF}, where we will consecutively send the Galerkin truncations in $\bm$, $\theta$, $\bu$ to infinity. More precisely, our existence proof involves the following steps:
\begin{itemize}
\item In subsection \ref{sec:scheme}, we introduce our Galerkin truncation  (in space) for $\bm$, $\theta$, $\bu$, leading to a discretized system of ODEs. This Galerkin truncation is described by parameters $N_m$, $N_\theta$, $N_u$, defined as the dimensions of the discretizations in $\bm$, $\theta$ and $\bu$, respectively.
\item In subsection \ref{sec:global-in-time}, we first show global-in-time existence of solutions to the discretized system with $N_m, N_\theta, N_u$ fixed. This will result in estimates that are uniform in $N_m$ (but depending on $N_\theta$, $N_u$). 
\item In subsection \ref{sec:Nm-limit}, we use these a priori bounds to pass to the limit $N_m \to \infty$.
\item Then, in subsection \ref{sec:Nth-limit} we analyze the limiting system, now only involving truncation parameters $N_u$, $N_\theta$. We derive improved a priori estimates, which only depend on $N_u$, but are independent of $N_\theta$. In turn, we can pass $N_\theta\to \infty$.
\item Finally, in subsection \ref{sec:Nu-limit}, we consider the limiting system depending only on $N_u$. Here, we derive bounds that allow us to pass $N_u \to \infty$, yielding a weak solution of the limiting system \eqref{eq:CHF}.
\end{itemize}
The details of this derivation are provided below.

\subsubsection{Approximation scheme}
\label{sec:scheme}
We start by approximating the initial condition for $\bm$, $\bm_0$ by a sequence of smooth functions $\bm_0^\sigma$ such that $\bm_0^\sigma\to \bm_0$ in $L^2(\dom)$ and $\norm{\bm_0^\sigma}_{L^2}\leq \norm{\bm_0}_{L^2}$.
Then, we approximate $\bu$, $\bm$ and $\theta$ using a Galerkin approximation. We let $\{a_j\}_{j=1}^\infty$ be a basis of Stokes eigenfunctions on $\V$, $\{d_j\}_{j=1}^\infty$ a sequence of smooth eigenfunctions, of the Laplace operator
\begin{equation*}
	-\Delta d_j = \lambda_j d_j,\quad \text{in }\,\dom,\quad \frac{\partial d_j}{\partial n}=0,\quad \text{on }\, \partial \dom,\quad \int_{\dom} d_j dx =0,
\end{equation*}
with $0<\lambda_1\leq \lambda_2\leq ...\leq \lambda_j\leq \lambda_{j+1}\leq ...$
that forms an orthogonal basis of the space $\ths:=\{v\in H^1(\dom)\,:\, \int_{\dom} v dx =0\}$ with respect to the inner product of $L^2(\dom)$, and $\{c_j\}_{j=1}^\infty:\dom\to \R^d$ a sequence of smooth eigenfunctions of the vector Laplacian,
\begin{equation*}
	-\Delta c_j = \eta_j c_j,\quad \text{in }\, \dom,\quad c_j=0,\quad \text{on }\, \partial\dom,
\end{equation*} 
that form an orthonormal basis of $L^2(\dom;\R^d)$.
We denote the finite dimensional subspaces, which are projections onto the linear span of the first $N_i$ basis functions:
\begin{equation*}
	\mathcal{U}_{N_u} = \text{span}\{a_j\}_{j=1}^{N_u},\quad \mathcal{W}_{N_\theta} = \text{span}\{\{d_j\}_{j=1}^{N_\theta}\cup \{1\}\},\quad  \mathcal{M}_{N_m} = \text{span}\{c_j\}_{j=1}^{N_m},
\end{equation*}
where we denoted by $1$ the constant function.
We denote by $P^i_N$ the $L^2$-projection onto the first $N$ basis functions of the space for the variable $i$ ($i\in \{\bu,\bm,\theta\}$). Specifically, these projections satisfy with respect to the $L^2$-inner product
\begin{align*}
	(P_{N}^u v,w) &= (v,w),\quad \forall w\in \mathcal{U}_N,\\
		(P_{N}^\theta v,w) &= (v,w),\quad \forall w\in \mathcal{W}_N,\\
		(P_{N}^m v,w) &= (v,w),\quad \forall w\in \mathcal{M}_N.
\end{align*}
We look for functions
\begin{equation*}
	\bu^{\sigma,\eta,N_u,N_\theta,N_m} = \sum_{j=1}^{N_u} \bu_j a_j,\quad \theta^{\sigma,\eta,N_u,N_\theta,N_m}= \sum_{j=1}^{N_\theta} \theta_j d_j,\quad 
	 \bm^{\sigma,\eta,N_u,N_\theta,N_m} = \sum_{j=1}^{N_m} \bm_j c_j,
\end{equation*}
that solve (we omit the superscripts for the moment for better readability)
\begin{subequations}\label{eq:Galerkinsolve}
	\begin{align}
		(\bu_t,v)+b_u(\bu,\bu,v)&=-(\nu_\theta T(\bu),T(v))+\mu_0((\bm\cdot\Grad)\bh,v)-\lambda\eps(\Delta\theta\Grad\theta,v), \quad \forall v\in \mathcal{U}_{N_u}\label{eq:NSdiscrete}\\
		(\bm_t,w)+b_m(\bu,\bm,w)&=-\frac{1}{\tau}(\bm-(\chi_\theta\bh)_\eta,w),\quad \forall w\in \mathcal{M}_{N_m} \label{eq:meqdiscrete}\\
	 \Grad\varphi&=\bh,\quad -\Delta \varphi =\Div\bm,\quad (t,x)\in [0,T]\times \dom,\label{eq:maxwelldiscrete}\\
	  \Grad\varphi\cdot \mathbf{n}& = (\bh_a-\bm)\cdot\mathbf{n},\quad (t,x)\in [0,T]\times\partial\dom\label{eq:bcmaxwell} \\
		(\theta_t,z )+b_\theta(\bu,\theta,z)&=-\kappa(\Grad\psi,\Grad z),\quad \forall z\in\mathcal{W}_{N_\theta}\label{eq:phasefielddiscrete}\\
		(\psi,z)&= -\lambda\eps(\Delta\theta,z) +\frac{\lambda}{\eps}(F'(\theta),z)-\left(\frac{\mu_0}{2\chi_\theta^2}\chi'_\theta|\bm|^2,z\right),\quad \forall z\in \mathcal{W}_{N_\theta}\label{eq:psi1}
	\end{align}
\end{subequations}
where $\eta>0$ and $(f)_\eta$, which appears on the right-hand side of \eqref{eq:meqdiscrete}, is defined as
\begin{equation*}
	(f)_\eta(x) = \int_{\R^d} f(x-y)\omega_\eta(y) dy,
\end{equation*}
where $\omega_\eta(y) = \eta^{-d} \omega(y/\eta)$, $\omega$ is a smooth, radially symmetric and compactly supported within $B_1(0)$ mollifier with $\int_{\R^d}\omega dx =1$. If needed, we extend $f$ by zero outside of $\dom$. The trilinear forms $b_m$, $b_u$ and $b_\theta$ are defined as
\begin{equation*}
	b_u(u,v,w) = \int_{\dom} (u\cdot\Grad ) v \cdot w dx,  \quad \forall u\in L^2_{\text{div}}(\dom;\R^d), v,w\in H^1(\dom;\R^d),
\end{equation*} 
\begin{equation*}
	b_m(u,v,w) = \int_{\dom} (u\cdot\Grad ) v \cdot w dx,  \quad \forall u\in L^2_{\text{div}}(\dom;\R^d), v,w\in H^1(\dom;\R^d),
\end{equation*}
and 
\begin{equation*}
	b_\theta(u,v,w) = \int_{\dom} (u\cdot\Grad ) v \cdot w dx,  \quad \forall u\in L^2_{\text{div}}(\dom;\R^d), v,w\in H^1(\dom).
\end{equation*}
Note that all three are skew-symmetric. $\bh$, $\varphi$ respectively, also depend on the parameters $\sigma, \eta, N_u, N_\theta, N_m$.
In the following, we will first send $N_m\to\infty$, then $\eta\to 0$ and $N_\theta\to \infty$ and finally $\sigma\to 0$ and $N_u\to \infty$. We will therefore use $\eta = 1/N_\theta$ and $\sigma = 1/N_u$ and omit the dependence of the approximations on $\eta$ and $\sigma$.

\subsubsection{Global existence of solutions of ODE system independently of $N_m$}
\label{sec:global-in-time}
First, we need to show that this "semi-discrete" system has a solution for all times. 
We note that we can eliminate $\varphi$ from the system by solving the linear elliptic PDE~\eqref{eq:maxwelldiscrete} with boundary conditions~\eqref{eq:bcmaxwell} and obtain an expression that depends on $(-\Delta)^{-1} c_j$, $m_j$ and the boundary conditions. 
Plugging this into the other equations, we obtain a system of ODEs for the coefficients $\bu_j$, $\bm_j$ and $\theta_j$ with Lipschitz continuous right hand side (polynomials in $\bu_j$, $\bm_j$ and $\theta_j$). Thus, the Picard-Lindel\"of theorem yields existence of a unique solution of~\eqref{eq:Galerkinsolve} in a small time interval $[0,T^*]$ for $T^*$ chosen sufficently small. We want to extend this to arbitrary times $T>0$. 
\begin{remark}
	For example we could decompose $\varphi$ into $\varphi = \sum_{j=0}^{N_m} \varphi_j$ where $\varphi_0$ solves
	\begin{equation*}
		-\Delta \varphi_0 = 0,\quad \int \varphi_0 dx = 0,\quad -\Grad \varphi_0\cdot\mathbf{n} = (\bh_a -\bm)\cdot \mathbf{n} , 
	\end{equation*}
	and $\varphi_j$, $j=1,\dots, N_m$ solve
	\begin{equation*}
		-\Delta \varphi_j = m_j \Div c_j,\quad \int \varphi_j dx =0,\quad -\Grad\varphi_j \cdot \mathbf{n} = 0.
	\end{equation*}
\end{remark}

To do so, we first establish  a priori   estimates. First, we will try to obtain a bound on $\bma$. We use $w=\bma$ as a test function in~\eqref{eq:meqdiscrete} to obtain (use the skew symmetry of $b_m$)
\begin{equation}\label{eq:mbound}
	\frac12\frac{d}{dt}\int_{\dom} |\bma(t,x)|^2 dx  = -\frac{1}{\tau}\int_{\dom} |\bma(t,x)|^2 dx + \frac{1}{\tau}\int_{\dom} \bma \cdot (\chi_{\tha} \bha)_{\eta} dx. 
\end{equation}
Furthermore, from the elliptic equation~\eqref{eq:maxwelldiscrete}, we obtain, multiplying by $\pha$,
\begin{equation*}
	\int_{\dom} |\Grad \pha|^2 dx  = -\int_{\dom} \bma \cdot \Grad \pha dx +\int_{\partial\dom} \bh_a \cdot \mathbf{n}\pha dS.
\end{equation*}
Using Cauchy-Schwarz and the trace inequality, we can estimate the right hand side by
\begin{equation*}
	\begin{split}
		\int_{\dom} |\Grad \pha|^2 dx  
	&\leq \int_{\dom} |\bma|^2 dx +\frac14\int_{\dom}|\Grad \pha|^2 dx +C(\dom)^2 \int_{\partial\dom} |\bh_a|^2 dS \\
	&\quad +\frac{1}{4}\int_{\dom}|\Grad \pha|^2 dx.
		\end{split}\end{equation*}
Rearranging terms, we obtain
\begin{equation}\label{eq:boundh}
	\int_{\dom} |\bha|^2 dx =	\int_{\dom} |\Grad \pha|^2 dx \leq 2  \int_{\dom} |\bma|^2 dx  +2 C(\dom)^2 \int_{\partial\dom} |\bh_a|^2 dS ,
\end{equation}
thus if we can bound $\bma$, a bound on $\bha$ in $L^2$ automatically follows. Going back to~\eqref{eq:mbound}, and using that the convolution of a function satisfies $\norm{(f)_\eta}_{L^p}\leq \norm{f}_{L^p}$ for all $p\geq 1$, we find
\begin{equation*} 
	\begin{split}
	\frac12\frac{d}{dt}\int_{\dom} |\bma(t,x)|^2 dx & \leq  -\frac{1}{\tau}\int_{\dom} |\bma(t,x)|^2 dx + \frac{1}{2\tau}\int_{\dom} |\bma(t,x)|^2dx \\
	&\quad + \frac{1}{2\tau}\int_{\dom} |(\chi_{\tha} \Grad \pha)_{\eta}|^2 dx\\
	& \leq 
    C(\chi_0,\tau)\int_{\dom} |\bma(t,x)|^2 dx + C(\chi_0,\tau,\dom)\int_{\partial\dom} |\bh_a|^2 dS.
	\end{split}
\end{equation*}
Thus, we obtain from Gr\"onwall's inequality 
\begin{equation}\label{eq:uniformmbound}
	\norm{\bma(t)}_{L^2(\dom)}^2 \leq \norm{\bm_0}_{L^2(\dom)}^2 e^{Ct},
\end{equation}
therefore, since the initial data $\bm_0$ is in $L^2(\dom)$ and $\bh_a\in W^{1,\infty}([0,T];L^2(\partial\dom))$, then $\norm{\bma}_{L^\infty(0,T;L^2(\dom))}$ will be bounded uniformly in $N_m$, $N_u$ and $N_\theta$. From the bound~\eqref{eq:boundh}, we conclude that also $\norm{\bha}_{L^\infty(0,T;L^2(\dom))}$ is uniformly bounded in   $N_m$, $N_u$ and $N_\theta$. 
We continue to collect additional estimates. 
We use $z=\tha$ as a test function in~\eqref{eq:phasefielddiscrete} to obtain (using~\eqref{eq:psi1} with $z=\Delta \tha$ as a test function)
\begin{multline*}
	\frac12\frac{d}{dt}\norm{\tha}_{L^2}^2= -\kappa \lambda\varepsilon(\Delta\tha,\Delta\tha)\\
	+ \frac{\lambda\kappa}{\varepsilon}(F'(\tha),\Delta\tha) - \frac{\kappa\mu_0}{2}\left(\frac{\chi_{\tha}'}{\chi_{\tha}^2} |\bma|^2,\Delta\tha\right).
\end{multline*}
We recall that $F'$ satisfies $|F'(\theta)|\leq C(1+|\theta|)$.
We use Young's inequality to estimate the second term and obtain
\begin{equation*}
	\begin{split}
	\frac12\frac{d}{dt}\norm{\tha}_{L^2}^2&\leq  -\frac{\kappa\lambda \varepsilon}{2}\norm{\Delta\tha}_{L^2}^2+ \frac{\lambda\kappa}{2\varepsilon^3}\norm{F'(\tha)}_{L^2}^2 \\
	&\quad - \frac{\kappa\mu_0}{2}\left(\frac{\chi_{\tha}'}{\chi_{\tha}^2} |\bma|^2,\Delta\tha\right)\\
	&\leq  -\frac{\kappa \lambda\varepsilon}{2}\norm{\Delta\tha}_{L^2}^2+ \frac{C\lambda\kappa}{2\varepsilon^3}(1+\norm{\tha}_{L^2}^2 )\\
	&\quad - \frac{\kappa\mu_0}{2}\left(\frac{\chi_{\tha}'}{\chi_{\tha}^2} |\bma|^2,\Delta\tha\right)
	\end{split}
\end{equation*}

Using the boundary condition for $\tha$, we integrate by parts twice in the first term on the right hand side to obtain
\begin{multline*}
	\frac12\frac{d}{dt}\norm{\tha}_{L^2}^2\leq  -\frac{\kappa\lambda \varepsilon}{2}\norm{\Grad^2\tha}_{L^2}^2+  \frac{C\lambda\kappa}{2\varepsilon^3}(1+\norm{\tha}_{L^2}^2 )\\
	- \frac{\kappa\mu_0}{2}\left(\frac{\chi_{\tha}'}{\chi_{\tha}^2} |\bma|^2,\Delta\tha\right)
\end{multline*}
Since we are currently not interested in sending $\varepsilon \to 0$ and keep all other parameters fixed, we will lump all of them together in the constant $C$. Next, we use that $\chi$ and its derivatives are bounded 
(see equation \eqref{eq:assumptions}) to estimate the last term on the right hand side
\begin{equation*}
	\frac12\frac{d}{dt}\norm{\tha}_{L^2}^2\leq  -\frac{\kappa\lambda \varepsilon}{2}\norm{\Grad^2\tha}_{L^2}^2+  C(1+\norm{\tha}_{L^2}^2 )+ C \norm{\bma}_{L^2}^2\norm{\Delta\tha}_{L^\infty}
\end{equation*}
We have already shown in~\eqref{eq:uniformmbound} that $\bma\in L^\infty(0,T;L^2(\dom))$ uniformly in $N_i$, thus, we can further estimate the right hand side
\begin{equation*}
	\frac12\frac{d}{dt}\norm{\tha}_{L^2}^2\leq  -\frac{\kappa\lambda \varepsilon}{2}\norm{\Grad^2\tha}_{L^2}^2+  C(1+\norm{\tha}_{L^2}^2 )+ C \norm{\Delta\tha}_{L^\infty}.
\end{equation*}
Next, we use the fact that $\tha$ is expanded in a finite-dimensional smooth Galerkin basis. In particular, there exists $\Lambda_{N_\theta}>0$ depending only on $N_\theta$, such that $\norm{\Delta\tha}_{L^\infty} \le \Lambda_{N_\theta} \norm{\tha}_{L^2}$. Plugging this back into the above bound, we obtain
\begin{equation*}
	\frac12\frac{d}{dt}\norm{\tha}_{L^2}^2  \leq C (1+\Lambda_{N_\theta}) (1+\norm{\tha}_{L^2}^2 ).
\end{equation*}
Using Gr\"onwall's inequality we first obtain a bound on $\norm{\tha(t)}_{L^2}^2$ and using once more the finite-dimensionality and smoothness of the Galerkin expansion to bound second derivatives in terms of the $L^2$-norm, it follows that 
\begin{equation}
	\label{eq:almostuniformphibound}
	\norm{\tha(t)}_{L^2}^2+\norm{\Grad^2\tha(t)}^2_{L^2}
    \leq C({N_\theta},\norm{\theta_0}_{L^2},T), \quad \forall t\in [0,T],
\end{equation}
where $\theta_0$ is the initial condition of~\eqref{eq:phasefield1}. To clarify, the bound on the second-order derivatives is a consequence of the bound on the $L^2$-norm, but we state it here explicitly for later reference.
Thus, we obtained a bound on $\norm{\tha}_{L^\infty(0,T;L^2(\dom))}$ and $\norm{\tha}_{L^2(0,T;H^2(\dom))}$ that is independent of $N_u$ and $N_m$ (but depends on $N_\theta$).

Next, we derive  a bound for $\bua$. To do so,
we use $v=\bua$ as a test function in~\eqref{eq:NSdiscrete} to obtain
\begin{equation}
	\label{eq:basicuestimate}
	\begin{split}
		\frac{d}{dt}\frac12\norm{\bua}_{L^2}^2&
	= -\int_{\dom}\nu_{\tha} |T(\bua)|^2 dx + \mu_0 \int_{\dom} (\bma\cdot\Grad)\bha \cdot \bua dx \\
	&\quad -\varepsilon\lambda\int_{\dom} \Delta\tha \Grad\tha\cdot\bua dx
	\end{split}
\end{equation}
We will use the following crude estimates for the last two terms on the right hand side:
\begin{align*}
	&\left|   \mu_0 \int_{\dom} (\bma\cdot\Grad)\bha \cdot \bua dx  \right| \\
	 &= 	\left|   \mu_0 \int_{\dom} ((\bma+\bha)\cdot\Grad)\bua\cdot \bha dx  \right| \\
	& \leq C\left(\norm{\bha}_{L^2}^2 +\norm{\bma}_{L^2}^2\right)\norm{\Grad \bua}_{L^\infty}\\
	& \leq C(N_u) \left(\norm{\bha}_{L^2}^2 +\norm{\bma}_{L^2}^2\right)\norm{\bua}_{L^2}\\
	& \leq C(N_u) \norm{\bua}_{L^2}\\
	& \leq C(N_u)^2 + \norm{\bua}^2_{L^2},
\end{align*}
where for the first line, we used~\cite[Lemma 2.3]{Nochetto2019}, c.f.,~\eqref{eq:kelvin}, and for the second to last inequality, we used the previously established estimate~\eqref{eq:uniformmbound}. For the remaining term, we estimate
\begin{align*}
	&\left|\varepsilon\lambda\int_{\dom} \Delta\tha \Grad\tha\cdot\bua dx\right|\\
	& \leq \norm{\Delta \tha}_{L^2}\norm{\Grad\tha}_{L^2}\norm{\bua}_{L^\infty}\\
	& \leq C(N_u)\sqrt{\lambda_{N_\theta}} \norm{\Delta\tha}_{L^2}\norm{\tha}_{L^2}\norm{\bua}_{L^2}\\
	& \leq C(N_u,N_\theta)
    + \norm{\bua}_{L^2}^2,
\end{align*}
where we used the estimate~\eqref{eq:almostuniformphibound} for the last inequality. Substituting these back into~\eqref{eq:basicuestimate}, we obtain,

\begin{multline}
	\label{eq:basicuestimate2}
	\frac{d}{dt}\frac12\norm{\bua}_{L^2}^2
	+\int_{\dom}\nu_{\tha} |T(\bua)|^2 dx \leq 2\norm{\bua}_{L^2}^2 
    +  C(N_u, N_\theta) .
\end{multline}
Now applying Gr\"onwall's inequality and Korn's inequality~\cite{Korn1909}, we obtain
\begin{equation}\label{eq:almostubound}
	\norm{\bua(t)}_{L^2}^2 + \norm{\Grad\bua}_{L^2(0,t;L^2(\dom))}^2 \leq C(N_\theta,N_u,T) \quad \forall t\in [0,T].
\end{equation}
Thus, estimates~\eqref{eq:uniformmbound},~\eqref{eq:boundh}, \eqref{eq:almostubound} and~\eqref{eq:almostuniformphibound} imply that the solution of~\eqref{eq:Galerkinsolve} exists for all times $t\in [0,T]$. We note that this estimates are uniform with respect to $N_m$ (and $\eta$) but not with respect to $N_u$ and $N_\theta$. We will next send $N_m\to\infty$ to derive better a priori estimates.  But first, we note that we can estimate the derivatives of $\bma$ as follows:
We take $w=-\Delta \bma$ as a test function in~\eqref{eq:meqdiscrete}. Integrating the derivatives by part and using the boundary conditions, we obtain
\begin{align*}
	\frac12\frac{d}{dt}\norm{\Grad \bma}_{L^2}^2& = -\frac{1}{\tau}\norm{\Grad\bma}_{L^2}^2 + \frac{1}{\tau}\int_{\dom} \Grad\bma \Grad(\chi_{\tha}\Grad\pha)_{\eta} dx\\
	&\quad  + \int_{\dom} ((\Grad\bua\cdot\Grad)\bma) \cdot \Grad\bma dx 
\end{align*}
We estimate the second term on the right by  (recall that $\eta = 1/N_\theta$),
\begin{align*}
	&\left|\frac{1}{\tau}\int_{\dom} \Grad\bma \Grad(\chi_{\tha}\Grad\pha)_{\eta} dx\right|\\
	& \leq \frac{1}{\tau}\norm{\Grad\bma}_{L^2}\norm{\Grad(\chi_{\pha}\Grad\pha)_{\eta}}_{L^2}\\
	& \leq \frac{\chi_0}{\tau }\norm{\Grad\bma}_{L^2}\norm{\bha}_{L^2}\norm{\Grad\tha}_{L^\infty} \\
	&\quad + \frac{\chi_0}{\tau}\norm{\Grad\bma}_{L^2}\norm{\Grad\bha}_{L^2}\norm{\tha}_{L^\infty} \\
	& \leq \frac{1}{2\tau}\norm{\Grad\bma}_{L^2}^2 + C(N_\theta) \norm{\bha}_{H^1}^2.
\end{align*}
For the third term, we have
\begin{align*}
	\left| \int_{\dom} ((\Grad\bua\cdot\Grad)\bma) \cdot \Grad\bma dx \right|& \leq \norm{\Grad\bua}_{L^\infty}\norm{\Grad\bma}_{L^2}^2\\
	& \leq C(N_u)\norm{\bua}_{L^2}\norm{\Grad\bma}_{L^2}^2\\
	&\leq C(N_u , N_\theta,T)\norm{\Grad\bma}_{L^2}^2
\end{align*}
by the previous estimate~\eqref{eq:almostubound}. Combining these, we obtain
\begin{multline*}
	\frac12\frac{d}{dt}\norm{\Grad \bma}_{L^2}^2 \leq  -\frac{1}{2\tau}\norm{\Grad\bma}_{L^2}^2 + C(N_\theta) \norm{\bha}_{H^1}^2 \\
	+C(N_u , N_\theta,T)\norm{\Grad\bma}_{L^2}^2.
\end{multline*}
Note that thanks to elliptic regularity, we have 
\begin{equation*}
	\norm{\Grad\bha}_{L^\infty(0,T;L^2(\dom))} \leq \norm{\Grad\bma}_{L^\infty(0,T;L^2(\dom))}. 
\end{equation*}
Thus, with Gr\"onwall's inequality, we conclude that
\begin{equation}\label{eq:mgradbound}	\norm{\Grad\bma}_{L^\infty(0,T;L^2(\dom))}\leq C(N_u , N_\theta,T).
\end{equation}
This estimate is uniform with respect to $N_m$.
Again, by elliptic regularity, this also implies the same estimate on $\bha$:
\begin{equation}\label{eq:hgradbound}
	\norm{\Grad\bha}_{L^\infty(0,T;L^2(\dom))}\leq C(N_u , N_\theta,T).
\end{equation}
Next, we note that $\psi^{N_u,N_\theta,N_m}$ is given by
\begin{equation}\label{eq:psiprojection}
	\psi^{N_u,N_\theta,N_m}= P^\theta_{N_\theta}\left(-\varepsilon\lambda\Delta \tha +\frac{\lambda}{\varepsilon}F'(\tha)-\frac{\mu_0\chi'_{\tha}}{2\chi_{\tha}^2}|\bma|^2\right).
\end{equation}
Thus, since we have already shown that both $\tha\in L^2(0,T;H^2(\dom))$ and $\bma\in L^\infty(0,T;H^1(\dom))$ are bounded, it follows that at least $\psi^{N_u,N_\theta,N_m}\in L^2(0,T;L^2(\dom))$ uniformly with respect to $N_m$.
Before being able to pass to the limit $N_m\to \infty$, we need to derive estimates on the time derivatives. We take $w=\partial_t \bma$ as a test function in~\eqref{eq:meqdiscrete} to estimate
\begin{align*}
	&\norm{\bma_t}^2_{L^2}\leq \norm{\bua}_{L^\infty}\norm{\Grad \bma}_{L^2}\norm{\bma_t}_{L^2}\\
	&\quad  + \frac1{\tau}\norm{\bma}_{L^2}\norm{\bma_t}_{L^2}+\frac{\chi_0}{\tau}\norm{\bha}_{L^2}\norm{\bma_t}_{L^2}.
\end{align*}
We divide by $\norm{\bma_t}_{L^2}$ and further estimate the right hand side using that $\bua$ has $N_u$ terms,
\begin{equation*}
	\begin{split}
	\norm{\bma_t}_{L^2}&\leq \norm{\bua}_{L^\infty}\norm{\Grad \bma}_{L^2} + \frac1{\tau}\norm{\bma}_{L^2}+\frac{\chi_0}{\tau}\norm{\bha}_{L^2}\\
	&\leq C(N_u)\norm{\bua}_{L^2}\norm{\Grad \bma}_{L^2} + \frac1{\tau}\norm{\bma}_{L^2}+\frac{\chi_0}{\tau}\norm{\bha}_{L^2}\\
	&\leq C(N_u,N_\theta,T)
\end{split}
\end{equation*}
using the previous estimates. This implies that $\partial_t\bma$ is uniformly in $L^\infty(0,T;L^2(\dom))$ with respect to $N_m$. Furthermore we have that $\bua$, $\tha$ have continuous time derivatives and therefore are $C^1$ with respect to time uniformly in $N_m$. 
To derive a uniform estimate on the time derivative of $\bha$, we take the time derivative in equation~\eqref{eq:maxwelldiscrete}, then take $\partial_t \pha$ as a test function and integrate by parts,
\begin{align*}
	\norm{\bha_t}_{L^2}^2 & = \int_{\dom}\bma_t \bha_t dx + \int_{\partial\dom} (\bh_a)_t\cdot\mathbf{n} \pha_t dS\\
	&\leq \norm{\bha_t}_{L^2}\norm{\bma_t}_{L^2} + \norm{(\bh_a)_t}_{L^2(\partial\dom)}\norm{\pha_t}_{L^2(\partial\dom)},
\end{align*}
where we have also used that $\partial_t \bh_a\in L^\infty(0,T;L^2(\partial\dom))$ by the assumptions.
Using trace inequalities and then dividing by $\norm{\bha_t}_{L^2}$, we obtain the bound
\begin{equation*}
	\norm{\bha_t}_{L^2} \leq \norm{\bma_t}_{L^2} +C(\dom) \norm{(\bh_a)_t}_{L^2(\partial\dom)}\leq C(N_u , N_\theta,T).
\end{equation*}
\subsubsection{Passing $N_m \to \infty$}
\label{sec:Nm-limit}

 Combining the  a priori estimates~\eqref{eq:almostubound},~\eqref{eq:almostuniformphibound},~\eqref{eq:uniformmbound},~\eqref{eq:mgradbound},~\eqref{eq:hgradbound},
\begin{align}\label{eq:apriori1}
	\begin{split}
		&\{\bu^{N_u,N_\theta,N_m}\}_{N_m\in\N} \subset L^\infty(0,T;L^2_{\text{div}}(\dom))\cap L^2(0,T;\V) \\
		&\{\theta^{N_u,N_\theta,N_m}\}_{N_m\in\N} \subset L^\infty(0,T;L^2(\dom))\cap L^2(0,T;H^2(\dom)) \\
		&\{\bh^{N_u,N_\theta,N_m}\}_{N_m\in\N} \subset L^\infty(0,T;H^1(\dom)) \\
	&\{\bm^{N_u,N_\theta,N_m} \}_{N_m\in\N} \subset  L^\infty(0,T;H^1(\dom)) \\
		&\{\psi^{N_u,N_\theta,N_m} \}_{N_m\in\N} \subset L^2(0,T;L^2(\dom))
	\end{split}
\end{align}
with the estimates on the time derivatives, we can apply the  Aubin-Lions-Simon lemma~\cite{Simon1987}, to obtain convergence of subsequences (which we will still denote by $N_m$ for readability) of $\bm^{N_u,N_\theta,N_m}\to \bm^{N_u,N_\theta}$ and $\bh^{N_u,N_\theta,N_m}\to \bh^{N_u,N_\theta}$ in $L^p([0,T]\times\dom)$ for all $p\in [1,6)$ and also in $C(0,T;L^2(\dom))$ where the limits are in $ L^\infty(0,T;H^1(\dom))$ with time derivatives in $C(0,T;L^2(\dom))$. Subsequences of $\{\bu^{N_u,N_\theta,N_m}\}_{N_m}$ and $\{\theta^{N_u,N_\theta,N_m}\}_{N_m}$ also converge to respective limits $\bu^{N_u,N_\theta}$ and $\theta^{N_u,N_\theta}$  thanks to their time derivatives being in $C^1([0,T])$ (they are vector fields that are time dependent only due to the finite dimensionality of the Galerkin projections). In the following, we will keep taking subsequences of subsequences but not always explicitely state it. We will denote all subsequences by the same indexes $N_u$, $N_\theta$ and $N_m$.
For the sequence $\{\psi^{N_u,N_\theta,N_m}\}_{N_m}$, we obtain from the estimates for $\tha$ and $\bma$ and the expression for $\psi^{N_u,N_\theta,N_m}$,~\eqref{eq:psiprojection}, that their time derivatives are at least uniformly bounded, and therefore a subsequence converges strongly in $C(0,T;L^2(\dom))$ to a limit $\psi^{N_u,N_\theta}$.
 Furthermore, we obtain from~\eqref{eq:apriori1} that $\bm^{N_u,N_\theta,N_m}\weakstar \bm^{N_u,N_\theta}$ and $\bh^{N_u,N_\theta,N_m}\weakstar \bh^{N_u,N_\theta}$ in $L^\infty(0,T;H^1(\dom))$ up to a subsequence. From the estimates on the time derivatives, we obtain $\bm_t^{N_u,N_\theta,N_m}\weakstar \bm_t^{N_u,N_\theta}$ and $\bh_t^{N_u,N_\theta,N_m}\weakstar \bh_t^{N_u,N_\theta}$ in $L^\infty(0,T;L^2(\dom))$ up to a subsequence. We take $w\in L^1(0,T;L^2(\dom))$ and then take $P_{N_m}^m w$ as a test function in~\eqref{eq:meqdiscrete} and pass to the limit $N_m\to\infty$ using the previously established convergences after integrating over the time interval:
\begin{multline*}
	\underbrace{\int_0^T(\bma_t,w) dt}_{(i)} + \underbrace{\int_0^T b_m(\bua,\bma,w) dt}_{(ii)} \\
	+\underbrace{\frac1{\tau}\int_0^T (\bma-(\chi_{\tha}\bha)_\eta,w) dt}_{(iii)} 
	= 	\underbrace{\int_0^T(\bma_t,w-P_{N_m}^m w) dt}_{(iv)} \\
	+ \underbrace{\int_0^T b_m(\bua,\bma,w-P_{N_m}^m w) dt}_{(v)} +\underbrace{\frac1{\tau}\int_0^T (\bma-(\chi_{\tha}\bha)_\eta,w-P_{N_m}^m w) dt}_{(vi)}
\end{multline*}
For the first term, we have using the weak convergence of the time derivative of $\bma$:
\begin{equation*}
	(i)\stackrel{N_m\to\infty}{\longrightarrow} \int_0^T(\bmb_t,w) dt.
\end{equation*}

For the second term, we use the weak star convergence of $\bma$ in $L^\infty(0,T;H^1(\dom))$ combined with the strong convergence of $\bua$ in $C([0,T])$ to pass to the limit:
\begin{equation*}
	(ii)\stackrel{N_m\to\infty}{\longrightarrow}\int_0^T b_m(\bub,\bmb,w) dt.
\end{equation*}
For the third term, we use the weak star convergence of $\bma$ and $\bha$ in $L^\infty(0,T;L^2(\dom))$ combined with the strong convergence of $\tha$ in $C([0,T])$ to pass to the limit:
\begin{equation*}
	(iii)\stackrel{N_m\to\infty}{\longrightarrow}\frac1{\tau}\int_0^T (\bmb-(\chi_{\thb}\bhb)_\eta,w) dt. 
\end{equation*}
We need to show that the terms on the right hand side go to zero. This is the case because $\{c_j\}_j$ form a basis of $L^2(\dom)$, thus
\begin{equation*}
	\norm{P_{N_m}^mw-w}_{L^1(0,T;L^2(\dom))}\stackrel{N_m\to\infty}{\longrightarrow} 0.
\end{equation*}
Then due to the a priori bounds~\eqref{eq:apriori1} and the fact that $\bua$ is uniformly bounded with respect to $N_m$ in $L^\infty$, we can see that $(iv), (v), (vi)\to 0$ as $N_m\to \infty$. Due to the weak convergence of $\bma$ and $\bha$ in $L^\infty(0,T;H^1(\dom))$, and their time continuity, we can also pass to the limit in the weak formulation for~\eqref{eq:maxwelldiscrete}:
\begin{equation*}
	(\Grad\pha,\Grad\phi) = -(\bma,\Grad\phi) + \int_{\partial\dom} \bh_a\cdot\mathbf{n}\phi dS
\end{equation*}
where $\phi\in H^1(\dom)$ to obtain that  the limit $\bhb$, $\phb$ respectively satisfies
\begin{equation*}
	(\Grad\phb,\Grad\phi) = -(\bmb,\Grad\phi) + \int_{\partial\dom} \bh_a\cdot\mathbf{n}\phi dS.
\end{equation*}
For the equations for $\bua$ and $\tha$, we use the strong convergence of $\bua$ and $\tha$ in $C([0,T])$ combined with the convergences for $\bma$ and $\bha$ to pass to the limit. We just consider the terms involving the functions $\bma$ and $\bha$:
In the equation for $\bua$,~\eqref{eq:NSdiscrete}, we have, using the identity~\eqref{eq:kelvin}, and the strong convergence of $\bma$ and $\bha$ in $C(0,T;L^2(\dom))$,
\begin{align*}
	\mu_0 (\bma\cdot\Grad \bha,v)&=-\mu_0 \int_{\dom}((\bma+\bha)\cdot\Grad)v \cdot\bha dx\\ &\qquad \stackrel{N_m\to\infty}{\longrightarrow} -\mu_0 \int_{\dom}((\bmb+\bhb)\cdot\Grad)v \cdot\bhb dx\\
	&\qquad  \qquad \quad =\mu_0 (\bmb\cdot\Grad)\bhb,v).
\end{align*}
 The remaining term involving $\bma$ is the last term in equation~\eqref{eq:psi1}: Using the strong convergence of $\bma$ in $C([0,T];L^2(\dom))$, the strong convergence of $\tha$ in $C([0,T])$ and the boundedness of $\chi_{\tha}$ and its derivative, we obtain
\begin{equation*}
	\frac{\mu_0}{2}\left(\frac{\chi_{\tha}'}{\chi_{\tha}^2}|\bma|^2,z\right)\stackrel{N_m\to\infty}{\longrightarrow}	\frac{\mu_0}{2}\left(\frac{\chi_{\thb}'}{\chi_{\thb}^2}|\bmb|^2,z\right).
\end{equation*}
Thus the limit $(\bub,\bmb,\bhb,\thb,\psi^{N_u,N_\theta})$ satisfies
\begin{subequations}\label{eq:Galerkinsolve2}
	\begin{align}
		&(\bub_t,v)+b_u(\bub,\bub,v)=-(\nu_{\thb} T(\bub),T(v))+\mu_0((\bmb\cdot\Grad)\bhb,v)\label{eq:NSdiscrete2}\\
		&   \hphantom{(\bub_t,v)+b_u(\bub,\bub,v)=-}       -\lambda\eps(\Delta\thb\Grad\thb,v), \quad \forall v\in \mathcal{U}_{N_u},\notag \\
		&\int_0^T(\bmb_t,w)dt+\int_0^T b_m(\bub,\bmb,w)dt =-\frac{1}{\tau}\int_0^T (\bmb-(\chi_{\thb}\bhb)_\eta,w)dt , \label{eq:meqdiscrete2}\\
		&\hphantom{ \int_0^T(\bmb_t,w)dt+\int_0^T b_m(\bub,\bmb,w)dt =-\frac{1}{\tau}\int_0^T   }        \forall w\in L^1(0,T;L^2(\dom)),\notag\\
		&\Grad\phb=\bhb,\quad (\Grad \phb,\Grad \phi) =-(\bmb,\Grad\phi)+\int_{\partial\dom}\bh_a\cdot\mathbf{n}\phi dS,\quad \forall \phi\in H^1(\dom),\label{eq:bcmaxwell2}\\
		&(\thb_t,z )+b_\theta(\bub,\thb,z)=-\kappa(\Grad\psib,\Grad z),\quad \forall z\in\mathcal{W}_{N_\theta},\label{eq:phasefielddiscrete2}\\
		&(\psib,z)= -\lambda\eps(\Delta\thb,z) +\frac{\lambda}{\eps}(F'(\thb),z)-\left(\frac{\mu_0}{2\chi_{\thb}^2}\chi'_{\thb}|\bmb|^2,z\right),\label{eq:psi2}\\
		&\hphantom{(\psib,z)= -\lambda\eps(\Delta\thb,z) +\frac{\lambda}{\eps}(F'(\thb),z)}\quad \forall z\in \mathcal{W}_{N_\theta}.\notag
	\end{align}
\end{subequations}
\subsubsection{Passing $N_\theta\to\infty$}
\label{sec:Nth-limit}
Next, we derive better a priori estimates on $\thb$ that will allow us to send $N_\theta\to\infty$. First, we observe that each of the components of $\bmb$ satisfies a transport equation with smooth in space and continuously differentiable in time transport velocity and source that is smooth in space and at least Lipschitz continuous in time, since $\thb_t$ and $\bhb_t$ are bounded in $L^\infty(0,T;L^2(D))$ by the estimates of the last section. Additionally, we recall the initial data $\bmb|_{t=0} = \bm^\sigma_0$ is a smooth approximation of $\bm_0$ with fixed smoothing parameter $\sigma = 1/N_u$. Thus, using a Lagrangian reformulation of the equations in terms of ODEs, c.f.~\cite{DiPernaLions1989}, it follows that $\bmb$ is unique as a solution of~\eqref{eq:meqdiscrete2} and at least $C^1$-smooth. Therefore the equation for $\bmb$ holds pointwise, so that we can consider equation~\eqref{eq:meqdiscrete2} without the time integral: 
\begin{equation}\label{eq:meqpointwise}
	(\bmb_t,w)+ b_m(\bub,\bmb,w) =-\frac{1}{\tau}(\bmb-(\chi_{\thb}\bhb)_\eta,w) , \quad \forall w\in L^2(\dom).
\end{equation}
In this identity, we take $w=|\bmb|^2\bmb $ as a test function  (which  is uniformly  in $L^2(\dom)$ thanks to the Sobolev embedding $L^6\subset H^1$) to obtain that $\bmb\in L^\infty(0,T;L^4(\dom))$ uniformly in $N_\theta$: Indeed, we have with Cauchy-Schwarz and Young's inequality after integrating over $[0,t]$,
\begin{align}\label{eq:mL4}
	\begin{split}
	&\frac14\norm{\bmb(t)}_{L^4}^4 - \frac14\norm{\bm^{N_u}_0}_{L^4}^4\\
	& = -\frac1{\tau}\int_0^t \norm{\bmb}^4_{L^4} ds + \frac1{\tau}\int_0^t \int_{\dom}|\bmb|^2 \bmb\cdot (\chi_{\thb}\bhb)_{\eta} dx ds\\
	& \leq  -\frac1{\tau}\int_0^t \norm{\bmb}^4_{L^4} ds + \frac{\chi_0}{\tau}\int_0^t \norm{\bmb}_{L^4}^3\norm{\bhb}_{L^4} ds\\
		& \leq  -\frac1{\tau}\int_0^t \norm{\bmb}^4_{L^4} ds + \frac1{2\tau}\int_0^t \norm{\bmb}_{L^4}^4 ds +\frac{3\chi_0^4}{8\tau}\int_0^t\norm{\bhb(s)}_{L^4}^4 ds\\
		&\leq \frac{3\chi_0^4}{8\tau}\int_0^t\norm{\bhb}_{L^4}^4 ds.
		\end{split}
\end{align}
Elliptic regularity,  e.g.~\cite[Theorem 4.6]{Simader2006}, yields that $\norm{\bhb(t)}_{L^4}\leq C \norm{\bmb(t)}_{L^4}$, and then using Gr\"onwall's inequality, we obtain that $\bmb\in L^\infty(0,T;L^4(\dom))$ uniformly in  $N_\theta$   and we obtain the same bound for $\bhb$ also. In fact, this bound only depends on the approximation of the initial data, $\bm_0^\sigma$, i.e., the parameter $\sigma$.
We proceed to  improving the bounds for $\thb$. Taking again $z=\thb$ as a test function in~\eqref{eq:phasefielddiscrete2} and~\eqref{eq:psi2}, we obtain, as before
\begin{align*}
	&	\frac12\frac{d}{dt}\norm{\theta^{N_u,N_\theta}}_{L^2}^2\\
	&= -\kappa\lambda \varepsilon(\Delta\theta^{N_u,N_\theta},\Delta\theta^{N_u,N_\theta})+ \frac{\lambda\kappa}{\varepsilon}(F'(\theta^{N_u,N_\theta}),\Delta\theta^{N_u,N_\theta}) - \frac{\kappa\mu_0}{2}\left(\frac{\chi_{\theta^{N_u,N_\theta}}'}{\chi_{\theta^{N_u,N_\theta}}^2} |\bm^{N_u,N_\theta}|^2,\Delta\theta^{N_u,N_\theta}\right)\\
		&\leq  -\frac{\kappa \lambda\varepsilon}{2}\norm{\Delta\theta^{N_u,N_\theta}}_{L^2}^2+ \frac{C\lambda\kappa}{2\varepsilon^3}(1+\norm{\theta^{N_u,N_\theta}}_{L^2}^2 )- \frac{\kappa\mu_0}{2}\left(\frac{\chi_{\theta^{N_u,N_\theta}}'}{\chi_{\theta^{N_u,N_\theta}}^2} |\bm^{N_u,N_\theta}|^2,\Delta\theta^{N_u,N_\theta}\right)\\
		&\leq  -\frac{\kappa\lambda \varepsilon}{2}\norm{\Delta\theta^{N_u,N_\theta}}_{L^2}^2+ \frac{C\lambda\kappa}{2\varepsilon^3}(1+\norm{\theta^{N_u,N_\theta}}_{L^2}^2 )+\frac{\kappa\varepsilon\lambda}{4}\norm{\Delta\theta^{N_u,N_\theta}}_{L^2}^2 \\
		&\qquad +  \frac{\kappa\mu_0^2}{2\varepsilon\lambda }\frac{\sup |\chi_{\theta^{N_u,N_\theta}}'|^2}{\delta^4}\norm{\bm^{N_u,N_\theta}}_{L^4}^4\\
		& \leq -\frac{\kappa\lambda \varepsilon}{4}\norm{\Delta\theta^{N_u,N_\theta}}_{L^2}^2+ \frac{C\lambda\kappa}{2\varepsilon^3}(1+\norm{\theta^{N_u,N_\theta}}_{L^2}^2 ) +  \frac{\kappa\mu_0^2}{2\varepsilon\lambda}\frac{\sup |\chi_{\thb}'|^2}{\delta^4}\norm{\bmb}_{L^4}^4.
\end{align*}
Using the previously established bound on the $L^4$-norm of $\bmb$ and Gr\"onwall's inequality, we obtain that
\begin{equation}\label{eq:thetabound}
	\norm{\thb}_{L^\infty(0,T;L^2(\dom))}^2 + \norm{\Grad^2\thb}_{L^2(0,T;L^2(\dom))}^2 \leq C(\sigma)= C(N_u^{-1}).
\end{equation}
Note that since by integration by parts and the boundary conditions, we have
\begin{equation*}
	\int_{\dom} |\Grad\thb|^2 dx  = -\int_{\dom} \thb\Delta \thb dx \leq \norm{\thb}_{L^2}\norm{\Delta \thb}_{L^2} = \norm{\thb}_{L^2}\norm{\Grad^2\thb}_{L^2},
\end{equation*}
we obtain that $\Grad\thb\in L^4(0,T;L^2(\dom))$ uniformly in $N_\theta$.

Now, we will be able to derive a proper energy inequality for all the variables.  We start by taking $z=P^\theta_{N_\theta}\psi^{N_u,N_\theta}$ as a test function in~\eqref{eq:phasefielddiscrete2} and $z=\thb_t$ as a test function in~\eqref{eq:psi2}. Subtracting the resulting identities from each other, we obtain
\begin{align*}
	b_\theta(\bub,\thb,P^\theta_{N_\theta}\psib) &= -\kappa(\Grad\psib,\Grad\psib)-\frac{\varepsilon\lambda}{2}\frac{d}{dt}\norm{\Grad\thb}_{L^2}^2 -\frac{\lambda}{\varepsilon}\frac{d}{dt}\int_{\dom} F(\thb) dx\\
	&\quad  -\frac{\mu_0}{2}\int_{\dom} \partial_t\left(\frac{1}{\chi_{\thb}}\right) |\bmb|^2 dx,
\end{align*}
where we also integrated by parts and used chain rule where needed. Rearranging terms and plugging in~\eqref{eq:psi2} for $\psib$, we obtain
\begin{multline}\label{eq:thetatemp}
	\kappa\norm{\Grad\psib}_{L^2}^2+\frac{\varepsilon\lambda}{2}\frac{d}{dt}\norm{\Grad\thb}_{L^2}^2 +\frac{\lambda}{\varepsilon}\frac{d}{dt}\int_{\dom} F(\thb) dx +\frac{\mu_0}{2}\int_{\dom} \partial_t\left(\frac{1}{\chi_\thb}\right) |\bmb|^2 dx \\
	=	\varepsilon \lambda\int_{\dom} (\bub\cdot\Grad)\thb \cdot\Delta \thb dx  -\frac{\lambda}{\varepsilon}\int_{\dom} (\bub\cdot\Grad)\thb\cdot P_{N_\theta}^\theta F'(\thb) dx\\
	+\frac{\mu_0}{2}\int_{\dom} (\bub\cdot\Grad)\thb\cdot P_{N_\theta}^\theta \left(\frac{\chi_{\thb}'}{\chi_{\thb}^2}|\bmb|^2 \right)dx \\
	= 	\varepsilon\lambda \int_{\dom} (\bub\cdot\Grad)\thb \cdot\Delta \thb dx  +\frac{\lambda}{\varepsilon}\int_{\dom} (\bub\cdot\Grad)\thb\cdot (I-P_{N_\theta}^\theta) F'(\thb) dx\\
	-\frac{\mu_0}{2}\int_{\dom} (\bub\cdot\Grad)\thb\cdot (I-P_{N_\theta}^\theta) \left(\frac{\chi_{\thb}'}{\chi_{\thb}^2}|\bmb|^2 \right)dx \\
	- \frac{\mu_0}{2}\int_{\dom} \left(\bub\cdot\Grad\left(\frac{1}{\chi_{\thb}}\right) \right)\cdot|\bmb|^2 dx 
\end{multline}
where we used $I$ to denote the identity mapping.
Next, we take $v=\bub$ as a test function in~\eqref{eq:NSdiscrete2}
\begin{multline}\label{eq:utemp}
		\frac{d}{dt}\frac12\norm{\bub}_{L^2}^2
	= -\int_{\dom}\nu_{\thb} |T(\bub)|^2 dx + \mu_0 \int_{\dom} (\bmb\cdot\Grad)\bhb \cdot \bub dx \\
	-\lambda\varepsilon\int_{\dom} \Delta\thb \Grad\thb\cdot\bub dx
\end{multline}
Next, we then take $w=\mu_0(\chi_{\thb}^{-1}\bmb-\bhb)$ as a test function in~\eqref{eq:meqpointwise} to obtain
\begin{multline*}
	\frac{\mu_0}{2}\frac{d}{dt}\int_{\dom} \frac{|\bmb|^2}{\chi_{\thb}} dx - \frac{\mu_0}{2}\int_{\dom} |\bmb|^2 \partial_t\left(\frac{1}{\chi_{\thb}}\right) dx   -\mu_0\int_{\dom} \partial_t \bmb\cdot\bhb dx\\
	 + \frac{\mu_0}{2}\int_{\dom}\bub\cdot \Grad \left(\frac{|\bmb|^2}{\chi_{\thb}}\right) dx 
	-\frac{\mu_0}{2}\int_{\dom} |\bmb|^2 \bub\cdot\Grad\left(\frac{1}{\chi_{\thb}}\right) dx\\
	 - \mu_0\int_{\dom} ((\bub\cdot\Grad)\bmb )\cdot\bhb dx \\
	= - \frac{\mu_0}{\tau}\int_{\dom}\frac{1}{\chi_{\thb}}\left(\bmb-\chi_{\thb}\bhb\right)\left(\bmb-(\chi_{\thb}\bhb)_{\eta}\right) dx 
\end{multline*}
Next, we take the time derivative of equation~\eqref{eq:bcmaxwell2} and then
 $\phi=\mu_0\phb$ as a test function:
 \begin{equation*}
 	\frac{\mu_0}{2}\frac{d}{dt}\norm{\bhb}_{L^2}^2 = -\mu_0(\bmb_t,\bhb) + \mu_0\int_{\partial\dom}(\bh_a)_t\cdot\mathbf{n}\phb dS
 \end{equation*}
 Adding the last two identities together and using that $\bub$ is divergence free, we have:
 \begin{multline*}
 	\frac{\mu_0}{2}\frac{d}{dt}\int_{\dom} \left(\frac{|\bmb|^2}{\chi_{\thb}}+|\bhb|^2\right) dx - \frac{\mu_0}{2}\int_{\dom} |\bmb|^2 \partial_t\left(\frac{1}{\chi_{\thb}}\right) dx   \\
 		-\frac{\mu_0}{2}\int_{\dom} |\bmb|^2 \bub\cdot\Grad\left(\frac{1}{\chi_{\thb}}\right) dx
  - \mu_0\int_{\dom} ((\bub\cdot\Grad)\bmb )\cdot\bhb dx \\
 	= - \frac{\mu_0}{\tau}\int_{\dom}\frac{1}{\chi_{\thb}}\left(\bmb-\chi_{\thb}\bhb\right)\left(\bmb-(\chi_{\thb}\bhb)_{\eta}\right) dx \\
 	+\mu_0\int_{\partial\dom}(\bh_a)_t\cdot\mathbf{n}\phb dS
 \end{multline*}
 Now, we sum  this with the identity for $\bub$,~\eqref{eq:utemp}, and the identity for $\thb$,~\eqref{eq:thetatemp}, to obtain several cancellations of terms:
 \begin{multline}\label{eq:energyuthetadisc}
 	\frac{1}{2}\frac{d}{dt}\int_{\dom} \left(|\bub|^2 + \mu_0\left(\frac{|\bmb|^2}{\chi_{\thb}}+|\bhb|^2\right) \right)dx\\
 	+ \int_{\dom}\nu_{\thb} |T(\bub)|^2 dx  +\kappa\norm{\Grad\psib}_{L^2}^2+\frac{\varepsilon\lambda}{2}\frac{d}{dt}\norm{\Grad\thb}_{L^2}^2 +\frac{\lambda}{\varepsilon}\frac{d}{dt}\int_{\dom} F(\thb) dx\\
 	 = \frac{\mu_0}{2}\int_{\dom} |\bmb|^2 \partial_t\left(\frac{1}{\chi_{\thb}}\right) dx   	+\frac{\mu_0}{2}\int_{\dom} |\bmb|^2 \bub\cdot\Grad\left(\frac{1}{\chi_{\thb}}\right) dx\\
 -\frac{\mu_0}{2}\int_{\dom} \partial_t\left(\frac{1}{\chi_{\thb}}\right) |\bmb|^2 dx	- \frac{\mu_0}{2}\int_{\dom} \left(\bub\cdot\Grad\left(\frac{1}{\chi_{\thb}}\right) \right)\cdot|\bmb|^2 dx \\
  - \frac{\mu_0}{\tau}\int_{\dom}\frac{1}{\chi_{\thb}}\left(\bmb-\chi_{\thb}\bhb\right)\left(\bmb-(\chi_{\thb}\bhb)_{\eta}\right) dx \\
 		+ \mu_0\int_{\dom} ((\bub\cdot\Grad)\bmb )\cdot\bhb dx +\mu_0\int_{\partial\dom}(\bh_a)_t\cdot\mathbf{n}\phb dS
 	+ \mu_0 \int_{\dom} (\bmb\cdot\Grad)\bhb \cdot \bub dx\\
 	 -\varepsilon\lambda\int_{\dom} \Delta\thb \Grad\thb\cdot\bub dx
 	+	\varepsilon\lambda \int_{\dom} (\bub\cdot\Grad)\thb \cdot\Delta \thb dx \\
 	 +\frac{\lambda}{\varepsilon}\int_{\dom} (\bub\cdot\Grad)\thb\cdot (I-P_{N_\theta}^\theta) F'(\thb) dx-\frac{\mu_0}{2}\int_{\dom} (\bub\cdot\Grad)\thb \cdot (I-P_{N_\theta}^\theta) \left(\frac{\chi_{\thb}'}{\chi_{\thb}^2}|\bmb|^2 \right)dx \\ 
 	 = -\underbrace{\frac{\mu_0}{\tau}\int_{\dom}\frac{1}{\chi_{\thb}}\left(\bmb-\chi_{\thb}\bhb\right)\left(\bmb-(\chi_{\thb}\bhb)_{\eta}\right) dx}_{(i)} \\
 	 +\underbrace{\mu_0\int_{\partial\dom}(\bh_a)_t\cdot\mathbf{n}\phb dS}_{(ii)} 
 	 +\underbrace{\frac{\lambda}{\varepsilon}\int_{\dom} (\bub\cdot\Grad)\thb\cdot (I-P_{N_\theta}^\theta) F'(\thb) dx}_{(iii)}\\
 	  -\underbrace{\frac{\mu_0}{2}\int_{\dom} (\bub\cdot\Grad)\thb\cdot (I-P_{N_\theta}^\theta) \left(\frac{\chi_{\thb}'}{\chi_{\thb}^2}|\bmb|^2 \right)dx }_{(iv)}
 \end{multline}

 We estimate all the terms on the right hand side. For the first term, we have using Cauchy-Schwarz inequality,
 \begin{equation*}
 	|(i)|\leq \frac{2\mu_0(1+\chi_0^2)}{\tau\delta}\left(\norm{\bmb}_{L^2}^2 + \norm{\bhb}_{L^2}^2\right).
 \end{equation*}
 For the second term, we estimate, using trace inequalities
 \begin{equation*}
 	|(ii)|\leq \mu_0 \norm{(\bh_t)}_{L^2(\partial\dom)}\norm{\phb}_{L^2(\partial\dom)}\leq C \mu_0 \norm{(\bh_a)_t}_{L^2(\partial\dom)}\norm{\bhb}_{L^2(\dom)}\leq C (\norm{(\bh_a)_t}_{L^2(\partial\dom)}^2 + \norm{\bhb}_{L^2(\dom)}^2).
 \end{equation*}
 For the third term, we estimate, using that $|F'(\theta)|\leq C(1+|\theta|)$,
 \begin{align*}
 	|(iii)|&\leq C\frac{\lambda}{\varepsilon}\norm{\bub}_{L^4}\norm{\Grad\thb}_{L^2}\norm{F'(\thb)}_{L^4}\\
 	&\leq C\norm{\bub}_{L^4}\norm{\Grad\thb}_{L^2}(1+\norm{\thb}_{L^4})\\
 	& \leq C\norm{\bub}_{L^2}^{1/4}\norm{\Grad\bub}_{L^2}^{3/4}\norm{\Grad\thb}_{L^2}(1+C \norm{\thb}_{L^2}^{1/4}\norm{\Grad\thb}_{L^2}^{3/4})\\
 	&\leq C\norm{\bub}_{L^2}^{1/4}\norm{\Grad\bub}_{L^2}^{3/4}\norm{\Grad\thb}_{L^2}
C\norm{\bub}_{L^2}^{1/4}\norm{\Grad\bub}_{L^2}^{3/4}\norm{\Grad\thb}_{L^2}^{7/4} \\
 	&\leq \frac{\min\{\nu_f,\nu_w\}}{4}\norm{\Grad\bub}_{L^2}^2 + \frac{C}{\min\{\nu_f,\nu_w\}^{3/5}}\norm{\bub}_{L^2}^{2/5}\norm{\Grad\thb}_{L^2}^{14/5}  \\
 	&\quad +\frac{C}{\min\{\nu_f,\nu_w\}^{3/5}}\norm{\bub}_{L^2}^{2/5}\norm{\Grad\thb}_{L^2}^{8/5}\\
 	& \leq \frac{\min\{\nu_f,\nu_w\}}{4}\norm{\Grad\bub}_{L^2}^2 + C(\nu_f,\nu_w)\norm{\bub}_{L^2}^{2}+ C(\nu_f,\nu_w)\norm{\Grad\thb}_{L^2}^{14/4}  \\
 	&\quad + C(\nu_f,\nu_w)\norm{\Grad\thb}_{L^2}^{2},
 	 \end{align*}
 	 where we have used the Gagliardo-Nirenberg inequality, Young's inequality, and that the  projection $P_{N_\theta}^\theta$ is stable on $H^1(\dom)$ (and therefore on $L^4(\dom)$), and that $\thb\in L^\infty(0,T;L^2(\dom))$ uniformly in $N_\theta$ by~\eqref{eq:thetabound}.
 Finally, for the fourth term, we have
 \begin{align*}
	|(iv)|&\leq \frac{\mu_0\sup_{\theta}|\chi_\theta'|}{2\delta^2}\norm{\bub}_{L^4}\norm{\Grad\thb}_{L^4}\norm{\bmb}_{L^4}^2\\
	& \leq  C \norm{\bub}_{L^2}^{1/4}\norm{\Grad\bub}_{L^2}^{3/4} \norm{\Grad \thb}_{L^2}^{1/4}\norm{\Grad^2\thb}_{L^2}^{3/4}\norm{\bmb}_{L^4}^2\\
	&\quad + C \norm{\bub}_{L^2}\norm{\Grad\thb}_{L^2}\norm{\bmb}_{L^4}^2\\
	& \leq C \norm{\bub}_{L^2}^{1/2}\norm{\Grad \thb}_{L^2}^{1/2}\norm{\Grad^2\thb}_{L^2}^{3/2}\norm{\bmb}_{L^4}^4\\
	&\quad  + C \norm{\bub}_{L^2}^{1/2}\norm{\Grad\bub}_{L^2}^{3/2} \norm{\Grad \thb}_{L^2}^{1/2}\norm{\bmb}_{L^4}^4+ C \norm{\bub}_{L^2}^2\\
	&\quad + C \norm{\Grad\thb}_{L^2}^2\norm{\bmb}_{L^4}^4\\
	& \leq \frac{\min\{\nu_w,\nu_f\}}{4}\norm{\Grad\bub}_{L^2}^2 + \frac{C}{\min\{\nu_w,\nu_f\}^3} \norm{\bub}_{L^2}^2\norm{\Grad\thb}_{L^2}^2\norm{\bmb}_{L^4}^{16} \\
	&\quad + C \norm{\Grad^2\thb}_{L^2}^2 + C \norm{\bub}_{L^2}^2\norm{\Grad\thb}_{L^2}^2\norm{\bmb}_{L^4}^{16}\\
	&\quad + C \norm{\bub}_{L^2}^2+ C \norm{\Grad\thb}_{L^2}^2\norm{\bmb}_{L^4}^4.
\end{align*}
where we used the Gagliardo-Nirenberg inequality and that the  projection $P_{N_\theta}^\theta$ is stable on $H^1(\dom)$.
Using these estimates for the right hand side of~\eqref{eq:energyuthetadisc}, we obtain
 \begin{multline}\label{eq:energyuthetadisc2}
	\frac{1}{2}\frac{d}{dt}\int_{\dom} \left(|\bub|^2 + \mu_0\left(\frac{|\bmb|^2}{\chi_{\thb}}+|\bhb |^2\right) \right)dx+ \int_{\dom}\nu_{\thb} |T(\bub)|^2 dx \\
	 +\kappa\norm{\Grad\psib}_{L^2}^2+\frac{\varepsilon\lambda}{2}\frac{d}{dt}\norm{\Grad\thb}_{L^2}^2 +\frac{\lambda}{\varepsilon}\frac{d}{dt}\int_{\dom} F(\thb) dx\\
	\leq \frac{2\mu_0(1+\chi_0^2)}{\tau\delta}\left(\norm{\bmb}_{L^2}^2 + \norm{\bhb}_{L^2}^2\right)+C (\norm{(\bh_a)_t}_{L^2(\partial\dom)}^2 + \norm{\bhb}_{L^2(\dom)}^2) \\
	+\frac{\min\{\nu_f,\nu_w\}}{2}\norm{\Grad\bub}_{L^2}^2 + C(\nu_f,\nu_w)\norm{\bub}_{L^2}^{2}+ C(\nu_f,\nu_w)\norm{\Grad\thb}_{L^2}^{14/4}  \\
	+ C(\nu_f,\nu_w)\norm{\Grad\thb}_{L^2}^{2}
	+ \frac{C}{\min\{\nu_w,\nu_f\}^3} \norm{\bub}_{L^2}^2\norm{\Grad\thb}_{L^2}^2\norm{\bmb}_{L^4}^{16} \\
	+ C \norm{\Grad^2\thb}_{L^2}^2 + C \norm{\bub}_{L^2}^2\norm{\Grad\thb}_{L^2}^2\norm{\bmb}_{L^4}^{16}
	+ C \norm{\Grad\thb}_{L^2}^2\norm{\bmb}_{L^4}^4.
\end{multline}
Using that $\bmb\in L^\infty(0,T;L^4(\dom))$, $\Grad^2\thb\in L^2([0,T]\times\dom)$, $\thb\in L^\infty(0,T;L^2(\dom))$ and that $\Grad\thb\in L^4(0,T;L^2(\dom))$ uniformly in $N_\theta$, we can use Gr\"onwall's inequality, to obtain that
\begin{multline}\label{eq:uniformboundsenergy}
	\norm{\bub}_{L^\infty(0,T;L^2(\dom))} + \norm{\bmb}_{L^\infty(0,T;L^2(\dom))}+\norm{\bhb}_{L^\infty(0,T;L^2(\dom))}\\
	 +\norm{\Grad\thb}_{L^\infty(0,T;L^2(\dom))}+ \norm{\Grad\psib}_{L^2(0,T;L^2(\dom))}+ \norm{\bub}_{L^2(0,T;H^1_0(\dom))} \leq C(\sigma)=C(N_u),
\end{multline}
where $C$ is a constant that depends on $\sigma=1/N_u$, but is independent of $N_\theta$. 
From~\eqref{eq:mL4} and~\eqref{eq:thetabound}, we also have
\begin{equation}
	\label{eq:uniforminallbutsigma}
	\norm{\bmb}_{L^\infty(0,T;L^4(\dom))}+ \norm{\bhb}_{L^\infty(0,T;L^4(\dom))}+\norm{\Grad^2\thb}_{L^2(0,T;L^2(\dom))}\leq C(\sigma) = C(N_u).
\end{equation}
We next improve the estimate on the gradient of $\bmb$ with respect to $N_\theta$. Taking $-\Delta\bmb  $ as a test function in~\eqref{eq:meqpointwise} (this is allowed because $\bmb$ is smooth in space), integrating in time, and integrating by parts, we obtain
\begin{align*}
	\frac12\norm{\Grad\bmb(t)}_{L^2}^2 - \frac12\norm{\Grad\bm^\sigma_0}_{L^2}^2&  = -\int_0^t \int_{\dom}(\Grad\bub\cdot\Grad)\bmb : \Grad \bmb dxds\\
	& -\frac{1}{\tau}\int_0^t \int_{\dom} |\Grad\bmb(s)|^2 dxds \\
	&+\frac{1}{\tau}\int_0^t \int_{\dom} \Grad\bmb \cdot (\Grad(\chi_{\thb}\bhb ))_\eta dx ds\\
	&\leq \int_0^t\norm{\bub}_{L^\infty(\dom)}\norm{\Grad\bmb}_{L^2}^2 ds-\frac{1}{\tau}\int_0^t \int_{\dom} |\Grad\bmb(s)|^2 dxds \\
	& +\frac{1}{\tau}\int_0^t \int_{\dom} \Grad\bmb \cdot (\chi_{\thb}\Grad \bhb )_\eta dx ds \\
	&+\frac{1}{\tau}\int_0^t \int_{\dom}  \Grad\bmb\cdot (\chi'_{\thb}\Grad\thb\bhb)_\eta dx ds\\
	& \leq C(N_u) \int_0^t\norm{\bub}_{L^2(\dom)}\norm{\Grad\bmb}_{L^2}^2 ds \\
	&+ \frac{\chi_0}{\tau}\int_0^t \norm{\Grad\bmb}_{L^2}\norm{\Grad\bhb}_{L^2}ds \\
	& +\frac{C}{\tau}\int_0^t \int_{\dom}  |\Grad\bmb| |\Grad\thb| |\bhb| dx ds\\
	&\leq C(N_u)\norm{\bub}_{L^\infty(0,T;L^2(\dom))} \int_0^t\norm{\Grad\bmb}_{L^2}^2 ds \\
	&+ \frac{\chi_0}{2\tau}\int_0^t \left(\norm{\Grad\bmb}_{L^2}^2+\norm{\Grad\bhb}_{L^2}^2\right) ds \\
	& +\frac{C}{\tau}\int_0^t \norm{\Grad\bmb}_{L^2}\norm{\Grad\thb}_{L^4}\norm{\bhb}_{L^4} ds\\
&\leq C(N_u)\int_0^t\left(\norm{\Grad\bmb}_{L^2}^2  +\norm{\Grad\bhb}_{L^2}^2\right)ds \\
& +\frac{C}{\tau}\int_0^t \left(\norm{\Grad\bmb}_{L^2}^2 + \norm{\Grad\thb}_{L^4}^2\norm{\bhb}_{L^4}^2 \right)ds\\
	&\leq C(N_u)\int_0^t\norm{\Grad\bmb}_{L^2}^2 ds + \frac{\chi_0}{2\tau}\int_0^t\norm{\Grad\bhb}_{L^2}^2ds \\
	& +\frac{C}{\tau}\int_0^t \norm{\Grad\bmb}_{L^2}^2 + \norm{\thb}_{H^2}^2\norm{\bhb}_{H^1(\dom)}^2 ds
\end{align*}
where we have used the Sobolev embedding theorem.
Using that by elliptic regularity $\norm{ \bhb}_{H^1}\leq C \norm{\bmb}_{H^1}$, bound~\eqref{eq:thetabound} and Gr\"onwall's inequality, we obtain
\begin{equation}\label{eq:mestimateindepofkappa}
	\norm{\Grad\bmb}_{L^\infty(0,T;L^2(\dom))}+\norm{\Grad\bhb}_{L^\infty(0,T;L^2(\dom))}\leq C(N_u).
\end{equation}
In contrast to the previous estimate for $\Grad\bmb$, \eqref{eq:mgradbound}, this one is independent of $N_\theta$.
Next, we derive an $L^2(0,T;H^1(\dom))$-bound for $\psi^{N_u,N_\theta}$. The energy inequality~\eqref{eq:energyuthetadisc} already yields a bound on  the gradient of $\psi^{N_u,N_\theta}$. We want to use the Poincar\'e inequality,
\begin{equation}\label{eq:poincare}
	\norm{f-f_{av}}_{L^2}\leq C(\dom)\norm{\Grad f}_{L^2},
\end{equation}
where $f_{av} = \frac{1}{|\dom|}\int_{\dom} f dx$ is the average of $f$ over the domain, to bound the $L^2$-norm of $\psib$. So we need a bound on the average of $\psi^{N_u,N_\theta}$. We take $z\equiv 1$ as a test function in~\eqref{eq:psi2} to obtain
\begin{equation*}
	\begin{split}
	\int_{\dom} \psi^{N_u,N_\theta} dx& = -\varepsilon\lambda\int_{\dom} \Delta \theta^{N_u,N_\theta} dx + \frac{\lambda}{\varepsilon} \int_{\dom} F'(\theta^{N_u,N_\theta})dx -\int_{\dom} \frac{\mu_0}{2\chi_{\theta^{N_u,N_\theta}}^2}\chi_{\theta^{N_u,N_\theta}}' |\bm^{N_u,N_\theta}|^2 dx \\
	& =  \frac{\lambda}{\varepsilon} \int_{\dom} F'(\theta^{N_u,N_\theta})dx-\int_{\dom} \frac{\mu_0}{2\chi_{\theta^{N_u,N_\theta}}^2}\chi_{\theta^{N_u,N_\theta}}' |\bm^{N_u,N_\theta}|^2 dx,
	\end{split}
\end{equation*}
thus
\begin{equation}\label{eq:psiaverageestimate}
	\left|\int_{\dom} \psi^{N_u,N_\theta} dx \right|\leq C(1 +  \norm{\theta^{N_u,N_\theta}}_{L^1(\dom)})+ C \norm{\bm^{N_u,N_\theta}}_{L^2}\leq C.
\end{equation}
This implies with the Poincar\'e inequality that $\psi^{N_u,N_\theta} \in L^2(0,T;H^1(\dom))$ uniformly with respect to $N_\theta$.

Next, we derive uniform bounds on the time derivatives.  Integrating equation~\eqref{eq:NSdiscrete2} in time and taking $v\in L^4(0,T;\V)$ as a test function in~\eqref{eq:NSdiscrete2}, and using the identity~\eqref{eq:kelvin} for the Kelvin force and~\eqref{eq:capillary} for the capillary force, we can rewrite it as
\begin{align*}
		\int_0^T(\bub_t,v)dt &=\underbrace{\int_0^Tb_u(\bub,P^u_{N_u}v,\bub)dt}_{(i)} -\underbrace{\int_0^T(\nu_{\thb} T(\bub),T( P^u_{N_u}v)) dt}_{(ii)}\\
		& -\underbrace{\mu_0\int_0^T(((\bmb+\bhb)\cdot\Grad)P^u_{N_u}v,\bhb)dt}_{(iii)} +\underbrace{\eps\lambda\int_0^T(\Grad\thb\otimes\Grad\thb,\Grad P^u_{N_u}v)dt}_{(iv)} 
\end{align*}
using the definition of the $L^2$-projection $P^u_{N_u}$. We estimate each of the terms on the right hand side. Using the Ladyshenskaya inequality and the energy estimate bounds~\eqref{eq:uniformboundsenergy}, we have for the first term,
\begin{align*}
	|(i)|& \leq \int_0^T\norm{\bub}_{L^4}^2\norm{\Grad P^u_{N_u} v}_{L^2} dt \\
	& \leq \norm{\bub}_{L^\infty(0,T;L^2(\dom))}^{1/2}\int_0^T\norm{\Grad\bub}_{L^2}^{3/2}\norm{\Grad P^u_{N_u} v}_{L^2} dt \\
	& \leq C \norm{\Grad \bub}_{L^2(0,T;L^2(\dom))}\int_0^T \norm{\Grad v}_{L^2}^4 dt\\
	& \leq C \norm{v}_{L^4(0,T;H^1(\dom))}.
\end{align*}
For the second term, we have
\begin{align*}
	|(ii)|& \leq \max\{\nu_f,\nu_w\}\norm{\Grad\bub}_{L^2([0,T]\times\dom)}\norm{\Grad P^u_{N_u} v}_{L^2([0,T]\times\dom)}\\
	& \leq C \norm{v}_{L^2([0,T];H^1(\dom))}.
\end{align*}
 For the third term, we integrate by parts and then estimate using~\eqref{eq:uniforminallbutsigma}
 \begin{align*}
 	|(iii)|&\leq \mu_0 \int_0^T\norm{\bhb}_{L^4}\left(\norm{\bhb}_{L^4}+\norm{\bmb}_{L^4}\right)\norm{\Grad P^u_{N_u} v}_{L^2}dt \\
 	& \leq C \int_0^T \norm{\Grad v}_{L^2(\dom)} dt\\
 	& \leq C \norm{v}_{L^2(0,T;H^1(\dom))}.
 \end{align*}
 Finally, we can estimate the fourth term using the a priori bounds,~\eqref{eq:uniformboundsenergy},~\eqref{eq:uniforminallbutsigma}, and the Gagliardo-Nirenberg inequality,
 \begin{align*}
 	|(iv)|& = \left|\varepsilon\lambda\int_0^T\int_{\dom} (\Grad\thb\cdot\Grad) P_{N_u}^u v \cdot \Grad\thb dx dt\right|\\
 	&\leq \varepsilon\lambda \int_0^T \norm{\Grad\thb}_{L^4}^2\norm{\Grad P^u_{N_u}v}_{L^2} dt\\
 	& \leq C \int_0^T \left( \norm{\Grad^2 \thb}_{L^2}^{3/2}\norm{\Grad\thb}_{L^2}^{1/2}+\norm{\Grad\thb}_{L^2}\right)\norm{\Grad v}_{L^2} dt\\
 	&  \leq C( \norm{\Grad^2 \thb}_{L^2(0,T;L^2(\dom))}^{2} +1)\int_0^T \norm{\Grad v}_{L^2}^{4}dt\\
 	& \leq C \norm{v}_{L^4(0,T;H^1(\dom))}.
 \end{align*}
 Thus, taking the supremum over $v\in L^4(0,T;\V)$, we obtain
 \begin{equation*}
 	\sup_{v\in L^4(0,T;\V)}\frac{ \left| \int_0^T (\bub_t,v) dt \right|}{\norm{v}_{L^4(0,T;\V)}} \leq C,
 \end{equation*}
 and hence $\bub\in L^{4/3}(0,T;(\V)^*)$ uniformly in  $N_\theta$. We continue to derive a bound on the time derivative of $\thb$.  We take as a test function $P^\theta_{N_\theta} z$ for $z\in L^4(0,T;H^1(\dom))$ in~\eqref{eq:phasefielddiscrete2} and integrate in time:
 \begin{align*}
 	\int_0^T(\thb_t,z ) dt &= \underbrace{\int_0^T b_\theta(\bub,P_{N_\theta}^\theta z,\thb) dt}_{(i)} -\underbrace{\kappa\int_0^T(\Grad\psib,\Grad P_{N_\theta}^\theta z) dt}_{(ii)},
 \end{align*}
 using the definition of the $L^2$-projection $P^\theta_{N_\theta}$. We estimate the terms on the right hand side, starting with the first one: Using the uniform estimates~\eqref{eq:uniformboundsenergy}, Ladyshenskaya's inequality, and the stability of $P^\theta_{N_\theta}$ on $H^1(\dom)$,
 \begin{align*}
 	|(i)|& \leq \int_0^T \norm{\Grad P_{N_\theta}^\theta z}_{L^2(\dom)}\norm{\bub}_{L^4(\dom)}\norm{\thb}_{L^4(\dom)}dt \\
 	&\leq C\int_0^T \norm{\Grad P_{N_\theta}^\theta z}_{L^2(\dom)}\norm{\bub}_{L^2(\dom)}^{1/4}\norm{\Grad\bub}_{L^2(\dom)}^{3/4}\norm{\thb}_{L^2(\dom)}^{1/4} \norm{\Grad \thb}_{L^2(\dom)}^{3/4}dt \\
 	&\leq C \int_0^T \norm{\Grad z}_{L^2(\dom)}\norm{\Grad\bub}_{L^2(\dom)}^{3/4} dt \\
 	&\leq C\left( \int_0^T \norm{\Grad z}_{L^2(\dom)}^4 dt\right)^{1/4} \norm{\Grad\bub}_{L^2(0,T;L^2(\dom))}\\
 	&\leq C \norm{z}_{L^4(0,T;H^1(\dom))}.
 \end{align*}
 We continue to estimate the second term on the right hand side
 \begin{equation*}
 	|(ii)| \leq \kappa \norm{\Grad\psib}_{L^2(0,T;L^2(\dom))}\norm{\Grad P_{N_\theta}^\theta z}_{L^2([0,T]\times\dom)}\leq C \norm{z}_{L^2(0,T;H^1(\dom))}.
 \end{equation*}
 Thus, taking the supremum over $z\in L^4(0,T;H^1(\dom))$, we obtain
 \begin{equation}\label{eq:timedertheta}
 	\sup_{z\in L^4(0,T;H^1(\dom))}\frac{ \left| \int_0^T (\thb_t,z) dt \right|}{\norm{z}_{L^4(0,T;H^1(\dom))}} \leq C,
 \end{equation}
 thus $\partial_t \thb\in L^{4/3}(0,T;(H^1(\dom))^*)$ uniformly with respect to   $N_\theta$. 
 Next, we derive an estimate for the time derivative of $\bmb$. Again, we take a test function $w\in L^2(0,T;H^3(\dom))$ in~\eqref{eq:meqdiscrete2} and estimate all the terms using the a priori estimates. We have
 \begin{align*}
 	\int_0^T (\bmb_t,w) dt & = -\underbrace{ \int_0^T b_m(\bub,\bmb,w) dt}_{(i)} -\underbrace{\frac{1}{\tau}\int_0^T\int_{\dom} (\bmb-(\chi_{\thb}\bhb)_\eta)\cdot w dx }_{(ii)}.
 \end{align*}

 We estimate the two terms on the right hand side:
 \begin{equation*}
 	|(i)| \leq \int_0^T \norm{\bub}_{L^2}\norm{\bmb}_{L^2}\norm{\Grad w}_{L^\infty} dt \leq C \norm{w}_{L^2(0,T;H^3(\dom))}, 
 \end{equation*}
 using Sobolev embeddings. For the second term, we have
 \begin{equation*}
 	|(ii)|\leq C(\norm{\bmb}_{L^2([0,T]\times\dom)}+\norm{\bhb}_{L^2([0,T]\times\dom)})\norm{w}_{L^2([0,T]\times\dom)}.
 \end{equation*}
 Thus, taking the supremum over $w\in L^2(0,T;H^3(\dom))$, we obtain
 \begin{equation}\label{eq:mtimecont}
 	\sup_{w\in L^2(0,T;H^3(\dom))}\frac{ \left| \int_0^T (\bmb_t,w) dt \right|}{\norm{w}_{L^2(0,T;H^3(\dom))}} \leq C,
 \end{equation}
 uniformly in the parameter  $N_\theta$. Thus $\bmb_t\in L^2(0,T; H^{-3}(\dom))$. To obtain an estimate for $\bhb$, we note that we can decompose $w\in L^2(\dom)$ using the Leray decomposition (or Helmholtz decomposition) as $w =g + \Grad f$, where $f\in H^1(\dom)$ and $g\in L^2_{\text{div}}(\dom)$ and $g$ and $\Grad f$ are orthogonal with respect to the $L^2$-inner product. Then taking a test function $w=(g+\Grad f)\in L^2(0,T;H^3(\dom))$, we have
 \begin{equation*}
 	(\bhb_t,w) = (\Grad\phb_t,w) = (\Grad\phb_t, \Grad f),
 \end{equation*}
 since $g$ is divergence free and $g\cdot\mathbf{n}=0$ on $\partial\dom$.
Then using  equation~\eqref{eq:bcmaxwell2} (after taking time derivatives), we have
 \begin{equation*}
 	(\bhb_t,w) = (\Grad\phb_t,\Grad f) = -(\bmb_t,\Grad f)+\int_{\partial\dom} (\bh_a)_t\cdot \mathbf{n} f dS,
 \end{equation*}
 and the terms on the right hand side are bounded since $f\in H^4(\dom)$ and we already showed that $\bmb_t\in L^2(0,T;H^{-3}(\dom))$ (and we assume that $(\bh_a)_t\in W^{1,\infty}(0,T;L^2(\partial\dom))$). So we obtain that also $\bhb_t\in L^2(0,T;H^{-3}(\dom))$ uniformly in the parameter  $ N_\theta$.  To get a bound for $\phb_t$, we use that the divergence operator from $H^1_0(\dom)\to L^2(\dom)/\R$ is onto, c.f.~\cite[Lemma 2.4, Chapter 1]{Temam2001}. Clearly, this is still true, when we restrict to $H^3(\dom)$ as it is a linear mapping. Thus given $g\in L^2(0,T;H^3(\dom))$, there is $w\in H^3(\dom)\cap H^1_0(\dom)$ such that $\Div w =g$. Hence, we have
 \begin{equation}\label{eq:phitimederivative}
 \int_0^T (\phb_t,g) dt = \int_0^T (\phb_t,\Div w)dt = -\int_0^T (\Grad\phb_t,w)dt = -\int_0^T(\bhb_t,w)dt.
  \end{equation}
  This shows that since $\bhb_t\in L^2(0,T;H^{-3}(\dom))$, we must have $\phb_t\in L^2(0,T;H^{-3}(\dom))$ also.
 Using these estimates on the time derivatives combined with the a priori estimates from the energy inequality~\eqref{eq:uniformboundsenergy}, we obtain from the Aubin-Lions-Simon lemma~\cite{Simon1987} that up to a subsequence, still denoted by $N_\theta$,
 \begin{align}
 \label{eq:Nu-conv}
 \begin{aligned}
 	&\bu^{N_u,N_\theta}\stackrel{N_\theta\to\infty}{\longrightarrow} \bu^{N_u},\quad \text{in }\, L^2(0,T;L^2_{\text{div}}(\dom))\\
 	& \theta^{N_u,N_\theta}\stackrel{N_\theta\to\infty}{\longrightarrow} \theta^{N_u},\quad \text{in }\, L^2(0,T;H^1(\dom))\cap C(0,T;L^2(\dom))\\
 		&\bm^{N_u,N_\theta}\stackrel{N_\theta\to\infty}{\longrightarrow} \bm^{N_u},\quad \text{in }\, C(0,T;L^4(\dom))\\
 		&\bh^{N_u,N_\theta}\stackrel{N_\theta\to\infty}{\longrightarrow} \bh^{N_u},\quad \text{in }\, C(0,T;L^4(\dom)).
 \end{aligned}
 \end{align}
 Using the Banach-Alaoglu theorem, we also get up to a subsequence
 \begin{equation*}
 	\psi^{N_u,N_\theta}\stackrel{N_\theta\to\infty}{\weak} \psi^{N_u}\quad \text{in }\, L^2(0,T;H^1(\dom)).
 \end{equation*}

 Thus the limit $(\bu^{N_u},\bm^{N_u},\bh^{N_u},\theta^{N_u},\psi^{N_u})$ satisfies, combining weak and strong convergences
 \begin{subequations}\label{eq:Galerkinsolve3}
 	\begin{align}
 		&(\buc_t,v)+b_u(\buc,\buc,v)=-(\nu_{\thc} T(\buc),T(v))-\mu_0(((\bmc+\bhc)\cdot\Grad)v,\bhc)\label{eq:NSdiscrete3}\\
 		&\hphantom{(\buc_t,v)+b_u(\buc,\buc,v)=-}+\eps\lambda(\Grad\thc\otimes\Grad\thc,\Grad v), \quad \forall v\in \mathcal{U}_{N_u},\notag \\
 		&\int_0^T(\bmc_t,w)dt+\int_0^T b_m(\buc,\bmc,w)dt =-\frac{1}{\tau}\int_0^T (\bmc-\chi_{\thc}\bhc,w)dt ,\label{eq:meqdiscrete3}\\
 		&\hphantom{\int_0^T(\bmc_t,w)dt+\int_0^T b_m(\buc,\bmc,w)dt =-\frac{1}{\tau}\int_0^T } \forall w\in L^4(0,T;L^2(\dom)) ,\notag \\
 		&\Grad\phc=\bhc,\quad (\Grad \phc,\Grad \phi) =-(\bmc,\Grad\phi)+\int_{\partial\dom}\bh_a\cdot\mathbf{n}\phi dS,\quad \forall \phi\in H^1(\dom),\label{eq:bcmaxwell3}\\
 &\int_0^T	(\thc_t,z )+b_\theta(\buc,\thc,z) dt =-\kappa\int_0^T(\Grad\psic,\Grad z)dt ,\quad \forall z\in L^4(0,T;H^1(\dom)),\label{eq:phasefielddiscrete3}\\
 	&\int_0^T(\psic,z)dt =\int_0^T\left( \lambda\eps(\Grad\thc,\Grad z) +\frac{\lambda}{\eps}(F'(\thc),z)-\left(\frac{\mu_0}{2\chi_{\thc}^2}\chi'_{\thc}|\bmc|^2,z\right)\right) dt , \label{eq:psi3}\\
 	&\hphantom{\int_0^T(\psic,z)dt =\int_0^T( \eps\lambda(\Grad\thc,\Grad z) +\frac{\lambda}{\eps}(F'(\thc),z)-} \forall z\in L^2([0,T]\times\dom).\notag
 \end{align}
 \end{subequations}
 In the energy inequality~\eqref{eq:energyuthetadisc}, we obtain when passing  $N_\theta\to \infty$, (after integrating in time)
 \begin{multline}\label{eq:energyuthetadisc3}
 	\frac12\int_{\dom}\left(|\bu^{N_u}(t)|^2 + \mu_0\frac{|\bm^{N_u}(t)|^2}{\chi_{\theta^{N_u}(t)}}+ \mu_0|\bh^{N_u}(t)|^2+\varepsilon\lambda|\Grad\theta^{N_u}(t)|^2 + \frac{2\lambda}{\varepsilon}F(\theta^{N_u}(t))\right) dx\\
 	 + \int_0^t\int_{\dom} \left(\nu_{\theta^{N_u}} |T(\bu)|^2 + \kappa |\Grad\psi^{N_u}|^2 \right)dx ds \\
 	\leq 	\frac12\int_{\dom}\left(|\bu_0|^2 + \mu_0\frac{|\bm^{N_u}_0|^2}{\chi_{\theta_0}}+ \mu_0|\bh_0|^2+\varepsilon\lambda|\Grad\theta_0|^2 + \frac{2\lambda}{\varepsilon}F(\theta_0)\right) dx\\
 	- \int_0^t\frac{\mu_0}{\tau}\int_{\dom}\frac{1}{\chi_{\theta^{N_u}}}\left|\bm^{N_u}-\chi_{\theta^{N_u}}\bh^{N_u}\right|^2 dx ds
 	+ \int_0^t \mu_0\int_{\partial\dom}(\bh_a)_t\cdot\mathbf{n}\varphi^{N_u} dS ds.
 \end{multline}
 since the other two terms go to zero:  
 \begin{equation*}
 	\begin{split}
 			\left|\int_0^t(iii)ds\right| & = \left|\frac{\lambda}{\varepsilon}\int_0^t\int_{\dom} (\bub\cdot\Grad)\thb\cdot (I-P_{N_\theta}^\theta) F'(\thb) dxds \right|\\
 			& \leq C \norm{\bub}_{L^\infty([0,T]\times\dom)}\norm{\Grad\thb}_{L^\infty(0,T;L^2(\dom))}\int_0^T \norm{(I-P_{N_\theta}^\theta)F'(\thb)}_{L^2}dt \\
 		&\leq C(N_u) \int_0^T \norm{(I-P_{N_\theta}^\theta)F'(\thb)}_{L^2}dt, 
 	\end{split}
  \end{equation*}
  where we use the fact that $\bub$ belongs to a finite-dimensional Galerkin space and \eqref{eq:uniformboundsenergy} to bound the $\bub$ and $\nabla \thb$-terms. Since $\thb \to \theta^{N_u}$ in $C(0,T;L^2(\dom))$ by \eqref{eq:Nu-conv}, it follows that also $F'(\thb) \to F'(\theta^{N_u})$. This implies that $\int_0^T \norm{(I-P_{N_\theta}^\theta)F'(\thb)}_{L^2}dt \to 0$ as $N_\theta \to \infty$.
  For the other term, we similarly have
 \begin{equation*}
 	\begin{split}
 		\left|\int_0^t (iv)ds \right|& \leq \left|\frac{\mu_0}{2}\int_0^t\int_{\dom} (\bub\cdot\Grad)\thb \cdot (I-P_{N_\theta}^\theta) \left(\frac{\chi_{\thb}'}{\chi_{\thb}^2}|\bmb|^2 \right)dx ds\right| \\
        & \leq C(N_u) \int_0^T \norm{(I-P_{N_\theta}^\theta)\left(\frac{\chi_{\thb}'}{\chi_{\thb}^2}|\bmb|^2\right)}_{L^{2}(\dom)} dt.
	\end{split}
 \end{equation*}
 Since, once again, $\frac{\chi'_{\thb}}{\chi^2_{\thb}} |\bmb|^2$ converges strongly in $C(0,T;L^2(\dom))$ by \eqref{eq:Nu-conv} (and owing to the uniform bounds on $\chi'_\theta$ and $\chi_\theta$), it follows that the $(I-P_{N_\theta}^\theta)$-projection converges to zero as $N_\theta\to \infty$.
 Using Gr\"onwall's inequality, the energy inequality~\eqref{eq:energyuthetadisc3} yields uniformly in $N_u$ and $\sigma=1/N_u$ that
 \begin{equation}
 	\label{eq:alluniformbounds}
 	\begin{split}
 	\norm{\buc}_{L^\infty(0,T;L^2(\dom))}+
 	\norm{\bmc}_{L^\infty(0,T;L^2(\dom))}+\norm{\bhc}_{L^\infty(0,T;L^2(\dom))} + \norm{\Grad\thc}_{L^\infty(0,T;L^2(\dom))} \\
 	+ \norm{\Grad\buc}_{L^2([0,T]\times\dom)} + \norm{\Grad\psic}_{L^2([0,T]\times\dom)}\leq C.
 	\end{split}
 \end{equation}
 To obtain a uniform $L^2$-bound for $\psi^{N_u}$, we need to use Poincar\'e inequality with a bound on the average of $\psi^{N_u}$ again. To do so, we note that the bound~\eqref{eq:psiaverageestimate} is still valid as all the terms on the right hand side are uniformly bounded with respect to $N_u$. Similarly, we derive a bound on $\norm{\thc(t)}_{L^2}$. Using $z=\mathbf{1}_{[0,t]}(t)$ as a test function in~\eqref{eq:phasefielddiscrete3}, we have
 \begin{equation*}
 	\int_{\dom} \thc(t,x)dx -\int_{\dom} \theta_0(x) dx = 0.
 \end{equation*}
 Thus the mass of $\thc$ is preserved over time. In particular, since it is initially bounded, it will stay bounded for all times. Then we can use the Poincar\'e inequality,~\eqref{eq:poincare} to obtain that $\thc\in L^\infty(0,T;L^2(\dom))$ uniformly with respect to $N_u$.
 \subsubsection{Sending $N_u\to\infty$}
 \label{sec:Nu-limit}
 The energy inequality~\eqref{eq:energyuthetadisc3} yields the uniform bounds~\eqref{eq:alluniformbounds}. 
 Now, we derive  new bounds on the time derivatives of $\buc$ and $\thc$ so that we can use the Aubin-Lions-Simon lemma once more to pass to the limit in $N_u$.
  Integrating equation~\eqref{eq:NSdiscrete3} in time and taking $v\in L^2(0,T;\V\cap H^3(\dom))$ as a test function,  
  we have
 \begin{align*}
 	\int_0^T(\buc_t,v)dt &=\underbrace{\int_0^Tb_u(\buc,P^u_{N_u}v,\buc)dt}_{(i)} -\underbrace{\int_0^T(\nu_{\thc} T(\buc),T( P^u_{N_u}v)) dt}_{(ii)}\\
 	&\quad -\underbrace{\mu_0\int_0^T(((\bmc+\bhc)\cdot\Grad)P^u_{N_u}v,\bhc)dt}_{(iii)} +\underbrace{\lambda\eps\int_0^T(\Grad\thc\otimes\Grad\thc,\Grad P^u_{N_u}v)dt}_{(iv)} 
 \end{align*}
 using the definition of the $L^2$-projection $P^u_{N_u}$. We estimate each of the terms on the right hand side. Using the Ladyshenskaya inequality and the energy estimate bounds~\eqref{eq:uniformboundsenergy}, we have for the first term, as before,
 \begin{align*}
 	|(i)|& \leq \int_0^T\norm{\buc}_{L^4}^2\norm{\Grad P^u_{N_u} v}_{L^2} dt \\
 	& \leq \norm{\buc}_{L^\infty(0,T;L^2(\dom))}^{1/2}\int_0^T\norm{\Grad\buc}_{L^2}^{3/2}\norm{\Grad P^u_{N_u} v}_{L^2} dt \\
 	& \leq C \norm{\Grad \buc}_{L^2(0,T;L^2(\dom))}\int_0^T \norm{\Grad v}_{L^2}^4 dt\\
 	& \leq C \norm{v}_{L^4(0,T;H^1(\dom))}.
 \end{align*}
 For the second term, we have, also similarly as before,
 \begin{align*}
 	|(ii)|& \leq 2\max\{\nu_f,\nu_w\}\norm{\Grad\buc}_{L^2([0,T]\times\dom)}\norm{\Grad P^u_{N_u} v}_{L^2([0,T]\times\dom)}\\
 	& \leq C \norm{v}_{L^2([0,T];H^1(\dom))}.
 \end{align*}
 For the third term, we integrate by parts and then estimate
 \begin{align*}
 	|(iii)|&\leq \mu_0 \int_0^T\norm{\bhc}_{L^2}\left(\norm{\bhc}_{L^2}+\norm{\bmc}_{L^2}\right)\norm{\Grad P^u_{N_u} v}_{L^\infty}dt \\
 	&\leq C\int_0^T\norm{\bhc}_{L^2}\left(\norm{\bhc}_{L^2}+\norm{\bmc}_{L^2}\right)\norm{\Grad P^u_{N_u} v}_{H^2}dt \\
 	& \leq C \int_0^T \norm{ v}_{H^3(\dom)} dt\\
 	& \leq C \norm{v}_{L^1(0,T;H^3(\dom))}.
 \end{align*}
 Finally, we can estimate the fourth term using the a priori bounds,
 \begin{align*}
 	|(iv)|& = \left|\varepsilon\lambda\int_0^T\int_{\dom} (\Grad\thc\cdot\Grad) P_{N_u}^u v \cdot \Grad\thc dx dt\right|\\
 	&\leq \varepsilon\lambda \int_0^T \norm{\Grad\thc}_{L^2}^2\norm{\Grad P^u_{N_u}v}_{L^\infty} dt\\
 	&\leq C\int_0^T \norm{\Grad\thc}_{L^2}^2\norm{\Grad P^u_{N_u}v}_{H^2} dt\\
 		&\leq C\int_0^T \norm{v}_{H^3} dt\\
 	& = C \norm{v}_{L^1(0,T;H^3(\dom))}.
 \end{align*}
 Thus, taking the supremum over $v\in L^2(0,T;\V\cap H^3(\dom))$, we obtain
 \begin{equation}\label{eq:utimecont}
 	\sup_{v\in L^2(0,T;\V\cap H^3(\dom))}\frac{ \left| \int_0^T (\buc_t,v) dt \right|}{\norm{v}_{L^2(0,T;\V\cap H^3(\dom))}} \leq C,
 \end{equation}
 and hence $\buc_t\in L^{2}(0,T;(\V\cap H^3(\dom))^*)$ uniformly in $N_u$.  The estimate for the time derivative of $\thc$,~\eqref{eq:timedertheta}, is already uniform in $N_u$, therefore it still applies. Now, we pass $N_u\to\infty$. Since $\sigma = 1/N_u$, we then clearly have $\bm_0^{\sigma}\to \bm_0$ in $L^2(\dom)$. For the variables $\bu^{N_u}$ and $\theta^{N_u}$, we again use the Aubin-Lions-Simon lemma to conclude that 
 \begin{align*}
 	&\bu^{N_u}\stackrel{N_u\to\infty}{\longrightarrow} \bu,\quad \text{in }\, L^2(0,T;L^2_{\text{div}}(\dom))\\
 	& \theta^{N_u}\stackrel{N_u\to\infty}{\longrightarrow} \theta,\quad \text{in }\, C(0,T;L^2(\dom)).
 \end{align*}
 We can also use the Aubin-Lions Lemma for the variable $\varphi^{N_u}$ since the weak bound on the time derivative  of $\varphi^{N_u}$,~\eqref{eq:phitimederivative}, still applies and is uniform with respect to the parameter $N_u$. Thus
 \begin{equation}
 	\label{eq:varphiconv}
 	\varphi^{N_u}\stackrel{N_u\to\infty}{\longrightarrow} \varphi,\quad \text{in }\, L^2(0,T;L^2(\dom)).
 \end{equation}
 Furthermore, from the a priori bounds~\eqref{eq:alluniformbounds}, we obtain that 
 \begin{align*}
 	&\bm^{N_u}\stackrel{N_u\to\infty}{\weakstar} \bm,\quad \text{in }\, L^\infty(0,T;L^2(\dom))\\
 	& \bh^{N_u}\stackrel{N_u\to\infty}{\weakstar} \bh,\quad \text{in }\, L^\infty(0,T;L^2(\dom))\\
 	& \psi^{N_u}\stackrel{N_u\to\infty}{\weak} \psi,\quad \text{in }\, L^2(0,T;H^1(\dom)).
 \end{align*}
 Our goal is to derive strong convergence in $L^2([0,T]\times\dom)$ for $\bh^{N_u}$ and $\bm^{N_u}$. We will need this to be able to pass to the limit in the terms that are nonlinear in $\bm^{N_u}$ and $\bh^{N_u}$ in~\eqref{eq:Galerkinsolve3}.
  To do so, we first find the limiting equations for $\bm$ and $\bh$. Since the equation for $\bh^{N_u}$ is linear, we can pass to the limit in all the terms. For the equation for $\bm^{N_u}$, we pass to the limit using that $\bu^{N_u}$ and $\theta^{N_u}$ converge strongly in $L^2([0,T]\times\dom)$ and $\bm^{N_u}$ converges weakly in $L^2([0,T]\times\dom)$ choosing the test function sufficiently smooth. 
   The limiting formulation holds for a larger class of test functions in $L^{2}(0,T;H^3(\dom))$ using density of smooth functions in $L^{2}(0,T;H^3(\dom))$  and that $\bu\in L^\infty(0,T;L^2_{\text{div}}(\dom))\cap L^2(0,T;\V)$ and $\bm\in L^\infty(0,T;L^2(\dom))$. 
 So we have that $\bh$ and $\bm$ satisfy 
 \begin{align}
 	\int_0^T(\bm_t,w)dt-\int_0^T b_m(\bu,w,\bm)dt &=-\frac{1}{\tau}\int_0^T (\bm-\chi_{\theta}\bh,w)dt ,\quad \forall w\in L^{2}(0,T;H^3(\dom)), \label{eq:meqdiscrete4}\\
 	\Grad\varphi=\bh,\quad \int_0^T(\Grad \varphi,\Grad \phi)dt  &=-\int_0^T(\bm,\Grad\phi)dt+\int_0^T\int_{\partial\dom}\bh_a\cdot\mathbf{n}\phi dSdt ,\quad \forall \phi\in L^1(0,T;H^1(\dom)).\label{eq:bcmaxwell4}
 \end{align}
  Arguing as in~\cite[Section 3.2]{Nochetto2019}, we find that $\bm$ is a renormalized solution as in the sense of~\cite[Lemma 3.4]{Nochetto2019} and satisfies a similar identity as that in~\cite[Lemma 3.5]{Nochetto2019}, i.e., that we have
 \begin{equation}\label{eq:midentity}
 	 	\frac12\int_{\dom} |\bm(t,x)|^2 dx = \frac12\int_{\dom} |\bm_0|^2 dx -\frac{1}{\tau}\int_0^t\int_{\dom}(|\bm(s,x)|^2-\bm(s,x)\cdot \chi_{\theta} \bh(s,x) dx ds.
 \end{equation}
 On the other hand, we have that the approximations satisfy
  \begin{equation}\label{eq:midentityapprox}
  	%\begin{split}
  	%&
  		\frac12\int_{\dom} |\bm^{N_u}(t,x)|^2 dx\leq \frac12\int_{\dom} |\bm^{N_u}_0|^2 dx%\\ &
  	-\frac{1}{\tau}\int_0^t\int_{\dom}(|\bm^{N_u}(s,x)|^2-\chi_{\theta^{N_u}}\bm^{N_u}(s,x)\cdot  \bh^{N_u}(s,x)) dx ds.
  	%\end{split}
 	\end{equation}
 Sending $N_u\to \infty$ in this identity, we obtain for the weak limits of $|\bm^{N_u}|^2$ and $\bm^{N_u}\cdot \bh^{N_u}$, (where we use bars to denote the weak limits)
  \begin{equation}\label{eq:midentity1}
 %	\begin{split}
 		%&	
 		\frac12\int_{\dom} \overline{|\bm(t,x)|^2} dx\leq \frac12\int_{\dom} |\bm_0|^2 dx-\frac{1}{\tau}\int_0^t\int_{\dom}(\overline{|\bm(s,x)|^2}-\chi_{\theta}\overline{\bm(s,x)\cdot  \bh(s,x)}) dx ds.
 %	\end{split}
 \end{equation}
 Subtracting~\eqref{eq:midentity} from this, we have
   \begin{multline}\label{eq:midentity2}
 	%	\begin{split}
 		%&	
 		\frac12\int_{\dom}\left( \overline{|\bm(t,x)|^2} - |\bm(t,x)|^2\right) dx\\
 		\leq -\frac{1}{\tau}\int_0^t\int_{\dom}(\overline{|\bm(s,x)|^2}-|\bm(s,x)|^2-\chi_{\theta}\overline{\bm(s,x)\cdot  \bh(s,x)}+\chi_{\theta}\bm(s,x)\cdot  \bh(s,x)) dx ds.
 		%	\end{split}
 \end{multline}
 Next, we derive an identity for $\overline{\bh\cdot\bm}$. For this, we consider the equation for $\varphi$. We take  $\phi=\varphi g$ as a test function in~\eqref{eq:bcmaxwell4}, where $g$ is a smooth function. Then we have
 \begin{equation}
 	\label{eq:hidentity}
 	\begin{split}
 		\int_0^T \int_{\dom} \left(|\bh|^2 g + \bh\cdot\Grad g \varphi \right) dxdt&  =	\int_0^T(\bh,\Grad(\varphi g))dt  = -\int_0^T(\bm,\Grad(\varphi g))dt + \int_0^T\int_{\partial\dom} \bh_a\cdot \mathbf{n}\varphi g dSdt \\
 		& =  -\int_0^T\int_{\dom}\left(\bm\cdot\bh g + \bm\cdot\Grad g \varphi\right)dxdt + \int_0^T\int_{\partial\dom} \bh_a\cdot \mathbf{n}\varphi g dSdt.
 	\end{split}
 \end{equation}
 On the other hand, we could have taken $\phi = \varphi^{N_u}g$ as a  test function in~\eqref{eq:bcmaxwell3} and integrated over time, which would yield
  \begin{equation}
 	\label{eq:hidentity2}
 	\begin{split}
 		\int_0^T \int_{\dom} \left(|\bh^{N_u}|^2 g + \bh^{N_u}\cdot\Grad g \varphi^{N_u} \right) dxdt&  =	\int_0^T (\bh^{N_u},\Grad(\varphi^{N_u} g))  dt\\
 		& = -\int_0^T(\bm^{N_u},\Grad(\varphi^{N_u} g))dt + \int_0^T\int_{\partial\dom} \bh_a\cdot \mathbf{n}\varphi^{N_u} g dSdt \\
 		& =  -\int_0^T\int_{\dom} \left(\bm^{N_u}\cdot\bh^{N_u} g + \bm^{N_u}\cdot\Grad g \varphi^{N_u}\right)dx dt \\
 		&\quad + \int_0^T\int_{\partial\dom} \bh_a\cdot \mathbf{n}\varphi^{N_u} g dSdt.
 	\end{split}
 \end{equation}
 Passing to the limit on the right and left hand side and using that $\varphi^{N_u}\to \varphi$ strongly, and $\bhc\weak\bh$, $\bmc\weak\bm$ weakly in $L^2([0,T]\times\dom)$, we obtain (using again bars to denote the weak limits):
 \begin{equation}
 	\label{eq:hidentitywl}
 	\begin{split}
 		\int_0^T \int_{\dom} \left(\overline{|\bh|^2} g + \bh\cdot\Grad g \varphi \right) dxdt
 		& =  -\int_0^T\int_{\dom} \left(\overline{\bm\cdot\bh} g + \bm\cdot\Grad g \varphi\right)dx dt + \int_0^T\int_{\partial\dom} \bh_a\cdot \mathbf{n}\varphi g dSdt.
 	\end{split}
 \end{equation}
 We subtract this identity from~\eqref{eq:hidentity} and rearrange terms to obtain
  \begin{equation}
 	\label{eq:hidentitycombined}
 	\begin{split}
 		\int_0^T \int_{\dom} \left(\overline{|\bh|^2}-|\bh|^2\right) g dxdt
 		& =  -\int_0^T\int_{\dom}\left(\overline{\bm\cdot\bh}-\bm\cdot\bh\right) g  dx dt.
 	\end{split}
 \end{equation}
 Since this is true for any smooth $g$, we must have almost everywhere that $\overline{|\bh|^2}-|\bh|^2 = -\overline{\bm\cdot\bh}+\bm\cdot\bh$. We use this in~\eqref{eq:midentity2} to estimate the right hand side as
   \begin{equation}\label{eq:midentity3}
 	%	\begin{split}
 		%&	
 		\frac12\int_{\dom}\left( \overline{|\bm(t,x)|^2} - |\bm(t,x)|^2\right) dx\leq -\frac{1}{\tau}\int_0^t\int_{\dom}(\overline{|\bm(s,x)|^2}-|\bm(s,x)|^2+\chi_{\theta}\left(\overline{|\bh(s,x)|^2}-|\bh(s,x)|^2\right)) dx ds.
 		%	\end{split}
 \end{equation}
 Weak limits satisfy, c.f~\cite[Corollary 3.33]{Novotny2004}, 
 \begin{equation}\label{eq:wl}
 |f|^2\leq \overline{|f|^2},\quad \int |f|^2\leq \int \overline{|f|^2}
 \end{equation}
 ($f\in \{\bh,\bm\}$),
 therefore since $\chi_\theta\geq 0$, the right hand side of~\eqref{eq:midentity3} must be less than zero. So we have
 \begin{equation}\label{eq:midentityfinal}
 	\frac12\int_{\dom}\left( \overline{|\bm(t,x)|^2} - |\bm(t,x)|^2\right) dx\leq 0,
 \end{equation} 
 however, by~\eqref{eq:wl}, we must also have that the intergrand on the  left hand side is greater or equal to zero almost everywhere, therefore, we must have $\overline{|\bm(t,x)|^2} - |\bm(t,x)|^2=0$ a.e. . By~\cite[Theorem 1.1.1 (iii)]{Evans1990}, this implies that $\bm^{N_u}\to \bm$ strongly in $L^2([0,T]\times\dom)$. Using this strong convergence of $\bmc$ in~\eqref{eq:hidentitycombined} on the right hand side, we see that also $|\bh|^2 = \overline{|\bh|^2}$ almost everywhere and therefore $\bh^{N_u}$ also converges strongly to $\bh$ in $L^2([0,T]\times\dom)$. Now we want to pass to the limit in the remaining equations. For equation~\eqref{eq:phasefielddiscrete3}, we use the strong convergence of $\bu^{N_u}$ in $L^2([0,T]\times\dom)$ and the weak convergence of $\Grad\thc$ and of $\Grad\psic$ in $L^2([0,T]\times\dom)$ to pass to the limit in all terms and obtain
 \begin{equation}
 	\label{eq:phasefieldfinal}
 	 \int_0^T	(\theta_t,z )+b_\theta(\bu,\theta,z) dt =-\kappa\int_0^T(\Grad\psi,\Grad z)dt
 \end{equation}
 for all sufficiently smooth test functions.  In equation~\eqref{eq:psi3}, we use the strong convergence of $\theta^{N_u}$ in $L^2([0,T]\times\dom)$ to pass to the limit in the nonlinear terms in $\theta$, the weak convergence of $\psi$ and $\Grad\theta$ for the two linear terms and finally the strong convergence of $\theta^{N_u}$ in $L^2([0,T]\times\dom)$ -- which implies almost everywhere convergence -- combined with the strong convergence of $\bm^{N_u}$ to pass to the limit in the term $\int_0^T\left(\frac{\mu_0}{2\chi_\theta^2}\chi'_\theta|\bm|^2,z\right)dt$. Here we also use Lemma~\ref{lem:kennethlemma} to combine the bounded convergence with $L^1$-convergence. Thus the limit satisfies
 \begin{equation}
 	\label{eq:phasefield2final}
 		\int_0^T(\psi,z)dt =\int_0^T\left( \lambda\eps(\Grad\theta,\Grad z) +\frac{\lambda}{\eps}(F'(\theta),z)-\left(\frac{\mu_0}{2\chi_\theta^2}\chi'_\theta|\bm|^2,z\right)\right) dt .
 \end{equation}
 It remains to deal with the equation for $\bu^{N_u}$:
 \begin{align*}
 	&(\bu^{N_u}_t,v)+b_u(\bu^{N_u},\bu^{N_u},v) =-(\nu_{\theta^{N_u}} T(\bu^{N_u}),T( v))-\mu_0(((\bm^{N_u}+\bh^{N_u})\cdot\Grad)v,\bh^{N_u})\\
 	&\qquad +\lambda\eps(\Grad\theta^{N_u}\otimes\Grad\theta^{N_u},\Grad v), \quad \forall v\in \mathcal{U}_{N_u}.
 \end{align*}
 In order to pass to the limit in this term, we need strong convergence of the gradient of $\theta^{N_u}$. We achieve this with an argument that is similar to the one we used to derive strong convergence for $\bm^{N_u}$. 
 We consider the equation for $\psi$, \eqref{eq:phasefield2final}, take  $z=\theta$ as a test function (which is admissible since $\theta\in L^\infty(0,T;H^1(\dom))$ and $\theta \chi_\theta'$ is bounded by~\eqref{eq:assumptions})
 \begin{equation}\label{eq:psilim}
 \int_0^T	\int_{\dom}\psi\theta dxdt  =\int_0^T\left[ \lambda\eps\int_{\dom}|\Grad\theta|^2 dx +\frac{\lambda}{\eps}\int_{\dom} F'(\theta)\theta dx -\int_{\dom}\frac{\mu_0}{2\chi_\theta^2}\chi'_\theta\theta |\bm|^2 dx\right] dt. 
 \end{equation}
 By the assumptions~\eqref{eq:assumptions}, the integrand in the last term is in $L^1$ and so the term is bounded. Similarly, we use as a test function in~\eqref{eq:psi3} $z=\theta^{N_u}$:
 \begin{equation}\label{eq:psilimapprox}
 		\int_0^T(\psi^{N_u},\theta^{N_u})dt =\int_0^T\left(\lambda \eps(\Grad\theta^{N_u},\Grad \theta^{N_u}) +\frac{\lambda}{\eps}(F'(\theta^{N_u}),\theta^{N_u})-\left(\frac{\mu_0}{2\chi_{\theta^{N_u}}^2}\chi'_{\theta^{N_u}}|\bm^{N_u}|^2,\theta^{N_u}\right)\right) dt
 \end{equation}
 and pass to the limit.
 By the previous considerations, a subsequence of $\theta^{N_u}$ converges strongly in $L^p([0,T]\times\dom)$ for $p<6$, $\bm^{N_u}$ converges in $L^2([0,T]\times\dom)$, 
 $\frac{\chi'_{\theta^{N_u}}\theta^{N_u}}{\chi^2_{\theta^{N_u}}}$ is a bounded function of $\theta^{N_u}$ by the assumptions~\eqref{eq:assumptions}
 and therefore converges almost everywhere (up to extraction of a subsequence), $\psi^{N_u}$ converges  weakly in $L^2([0,T]\times\dom)$, and $F'(\theta^{N_u})\theta^{N_u}$ converges strongly in $L^1([0,T]\times\dom)$ by the properties of $F$,~\eqref{eq:Fass}. Therefore, we obtain for the limits:
 \begin{equation}\label{eq:weaklim}
 	\int_0^T\int_{\dom}\psi\theta dxdt  =\int_0^T\left[ \lambda\eps\int_{\dom}\overline{|\Grad\theta|^2} dx +\frac{\lambda}{\eps}\int_{\dom} F'(\theta)\theta dx -\int_{\dom}\frac{\mu_0}{2\chi_\theta^2}\chi'_\theta\theta |\bm|^2 dx\right]dt. 
 \end{equation}
 The bar denotes the weak limit, as before. Combining~\eqref{eq:psilim} and~\eqref{eq:weaklim}, we see that
 \begin{equation}\label{eq:convofnorms}
 	\int_0^T\int_{\dom}\overline{|\Grad\theta|^2} dxdt  = \int_0^T\int_{\dom}|\Grad\theta|^2 dxdt,
 \end{equation}
 and hence $\Grad\thc$ converges strongly in $L^2([0,T]\times\dom)$, using again~\cite[Theorem 1.1.1 (iii)]{Evans1990}. This is enough to pass to the limit in the capillary force term in the fluids equation:  For a sufficiently smooth function $v\in C^3_c([0,T]\times\dom)$, we take $P^u_{N_u}v$ as a test function in~\eqref{eq:NSdiscrete3} and integrate over $[0,T]$. We have
  \begin{align*}
 	&\underbrace{\int_0^T(\bu^{N_u}_t,v)dt}_{(i)} +\underbrace{\int_0^T(\bu^{N_u}_t,P^u_{N_u}v-v)dt}_{(ii)} -\underbrace{\int_0^T b_u(\bu^{N_u},v,\bu^{N_u})dt}_{(iii)} \\
 	&-\underbrace{\int_0^T b_u(\bu^{N_u},P^u_{N_u}v-v,\bu^{N_u})dt }_{(iv)} =-\underbrace{\int_0^T(\nu_{\theta^{N_u}} T(\bu^{N_u}), T( v))dt}_{(v)}-\underbrace{\int_0^T(\nu_{\theta^{N_u}} T(\bu^{N_u}),T( P^u_{N_u}v)-T( v))dt}_{(vi)}\\ &-\underbrace{\mu_0\int_0^T(((\bm^{N_u}+\bh^{N_u})\cdot\Grad)v,\bh^{N_u})dt}_{(vii)} -\underbrace{\mu_0\int_0^T(((\bm^{N_u}+\bh^{N_u})\cdot\Grad)(P^u_{N_u}v-v),\bh^{N_u})dt}_{(viii)}\\
 	&\qquad +\underbrace{\eps\lambda\int_0^T(\Grad\theta^{N_u}\otimes \Grad\theta^{N_u},\Grad v)dt}_{(ix)} +\underbrace{\eps\lambda\int_0^T(\Grad\theta^{N_u}\otimes \Grad\theta^{N_u},\Grad (P^u_{N_u}v-v))dt }_{(x)}.
 \end{align*}
 We compute the limits of all the terms. We integrate by parts in the first term and use the weak convergence of $\bu^{N_u}$, 
 \begin{align*}
 	(i) = -\int_0^T (\bu^{N_u},v_t) dt +(\bu^{N_u}_0,v(0,\cdot))\stackrel{N_u\to\infty}{\longrightarrow} -\int_0^T(\bu,v_t)dt +(\bu_0,v(0,\cdot)) = \int_0^T(\bu_t,v)dt.
 \end{align*}
 For the second term, we use the $L^2$-orthogonality of the projection to conclude that 
 \begin{align*}
 	(ii)=0.
 \end{align*}
 For the third term, we use the strong convergence of $\bu^{N_u}$ in $L^2([0,T]\times\dom)$ to deduce that
 \begin{equation*}
 	(iii)\stackrel{N_u\to \infty}{\longrightarrow} \int_0^T b_u(\bu,v,\bu) dt.
 \end{equation*}
 For the fourth term, we use the properties of the projection $P^u_{N_u}$:
 \begin{equation*}
 	\begin{split}
 		|(iv)|&\leq \int_0^T\norm{\bu^{N_u}}_{L^4}^2 \norm{v-P_{N_u}^u v}_{L^2} dt\\
 		& \leq C\int_0^T \norm{\Grad\bu^{N_u}}_{L^2(\dom)}^{3/2}\norm{\bu^{N_u}}_{L^2(\dom)}^{1/2} \norm{v-P_{N_u}^u v}_{L^2(\dom)}dt\\
 			& \leq C\norm{\bu^{N_u}}_{L^\infty(0,T;L^2(\dom))}^{1/2} \left(\int_0^T \norm{\Grad\bu^{N_u}}_{L^2(\dom)}^{2}dt\right)^{3/4}\delta_{N_u}\left(\int_0^T\norm{\Grad v}_{L^2(\dom)}^4 dt\right)^{1/4}\\
 			&\leq C \delta_{N_u}\longrightarrow 0,\quad \text{as } N_u\to \infty,
 	\end{split}
 	\end{equation*}
 	where $\delta_{N_u}\to 0$ as $N_u\to \infty$. For term (v) we use the weak convergence of $\Grad\bu^{N_u}$, the strong convergence of $\theta^{N_u}$, and the boundedness of $\nu_{\theta^{N_u}}$, to conclude that
 	\begin{equation*}
 		(v)\stackrel{N_u\to\infty}{\longrightarrow} \int_0^T (\nu_\theta T(\bu),T(v))dt .
 	\end{equation*}
 For term (vi), we estimate
 \begin{equation*}
 	|(vi)| \leq \max\{\nu_w,\nu_f\}\delta_{N_u}\int_0^T \norm{\Grad \bu^{N_u}}_{L^2(\dom)}\norm{\Grad^2 v}_{L^2(\dom)}^2 dt \leq C \delta_{N_u}\stackrel{N_u\to\infty}{\longrightarrow} 0.
 \end{equation*}
 For term (vii), we use the strong convergenc of $\bm^{N_u}$ and $\bh^{N_u}$ in $L^2([0,T]\times\dom)$ to conclude
 \begin{equation*}
 	(vii)\stackrel{N_u\to \infty}{\longrightarrow} \mu\int_0^T(((\bm+\bh)\cdot\Grad)v,\bh)dt.
 \end{equation*}
 The term (viii) goes to zero thanks to the properties of the projection $P^u_{N_u}$:
 \begin{equation*}
 	|(viii)|\leq \left( \norm{\bh^{N_u}}_{L^\infty(0,T;L^2(\dom))}+\norm{\bm^{N_u}}_{L^\infty(0,T;L^2(\dom))}  \right)\norm{\bh^{N_u}}_{L^\infty(0,T;L^2(\dom))}\delta_{N_u} \norm{ v}_{L^1(0,T;H^3(\dom))}\leq C \delta_{N_u}.
 \end{equation*}
 For term (ix), we use the strong convergence of $\Grad \theta^{N_u}$ in $L^2$ to pass to the limit:
 \begin{equation*}
 	(ix)\stackrel{N_u\to \infty}{\longrightarrow} \varepsilon\lambda\int_0^T (\Grad\theta\otimes\Grad\theta,\Grad v) dt.
 \end{equation*}
 For the remaining term (x), we use the properties of the projection operator $P^u_{N_u}$ to show it converges to zero:
 \begin{equation*}
 	|(x)|\leq \varepsilon \lambda\norm{\Grad \theta^{N_u}}_{L^\infty(0,T;L^2(\dom))}^2 \norm{v-P^u_{N_u}v}_{L^1(0,T;W^{1,\infty}(\dom))}\leq C \delta_{N_u}\norm{v}_{L^1(0,T;H^3(\dom))} 
 \end{equation*}
 which goes to zero as $N_u\to \infty$. Combining the estimates (i) to (x), we conclude that the limit $\bu$ satisfies
  \begin{equation*}
 	\int_0^T(\bu_t,v)+b_u(\bu,\bu,v) dt =-\int_0^T(\nu_{\theta} T(\bu),T(v))+\mu_0(((\bm+\bh)\cdot\Grad)v,\bh) -\eps\lambda(\Grad\theta\otimes\Grad\theta,\Grad v)dt, \quad \forall v \in  C^3_c([0,T)\times\dom).
 \end{equation*}
 Furthermore, the limit $\bu$ is divergence free since every $\buc$ is divergence free and thus we can pass to the limit in
 \begin{equation*}
 	0 = \int_{\dom}\buc\cdot \Grad v dx \stackrel{N_u\to \infty}{\longrightarrow}\int_{\dom} \bu\cdot \Grad v dx, \quad \text{a.e. }\, t\in [0,T].
 \end{equation*}
 For the energy inequality, we can use the properties of the weak convergence to pass to the limit in~\eqref{eq:energyuthetadisc3} and obtain that the limiting functions satisfy the same inequality with the superscripts removed. Finally the weak time continuity of the variables $\bu$, $\theta$, $\bm$ and $\bh$ follows from the uniform bounds~\eqref{eq:utimecont},~\eqref{eq:timedertheta} and~\eqref{eq:mtimecont} and then combining their spatial integrability with~\cite[Lemma II.5.9]{Boyer2012}.
Hence $(\bu,\bm,\bh,\theta,\psi)$ is a weak solution of~\eqref{eq:CHF}.
 This concludes the proof of Proposition~\ref{prop:existence}.

\subsection{Existence for possibly vanishing susceptibility}
\label{sec:vanishing-mag}
Now we consider the equations with the susceptibility vanishing in one phase, i.e., in~\eqref{eq:susceptibility} we let $\delta\to 0$. We denote the corresponding approximations by $\chi_{\theta^\delta}^\delta$, $\theta^\delta$, $\psi^{\delta}$, $\bud$, $\bmd$, and $\bhd$.
We will assume that the (general) function $\chi_\theta^\delta$ satisfies  for any $0\leq  \delta\leq 1$,
\begin{equation}\label{eq:conddelta}
	0\leq \chi_\td\leq C,\quad \left|\frac{(\chi^\delta_{\theta})'}{\chi^\delta_{\theta}}\right|\leq C,\quad \forall\theta; \quad\quad   \left|\frac{(\chi_\theta^\delta)'}{(\chi^\delta_{\theta})^2}\right|\leq C \quad \forall \theta.
\end{equation} 
As previously shown in Lemma \ref{lem:heavyside}, such an assumption is for example satisfied for $\chi_\delta$ defined in~\eqref{eq:chidef} with $\mathcal{H}$ the Heavyside function approximation~\eqref{eq:heavyside}. 
Furthermore, we assume that the initial data $\bm_0$ has slightly better integrability than $L^2(\dom)$, specifically, we will assume that
\begin{equation}\label{eq:L52}
	\bm_0\in L^{\frac{12}{5}+\alpha}(\dom),
\end{equation}
for some arbitrary $\alpha>0$.
\subsubsection{Passage to the limit $\delta\to 0$}
We use the same strategy as before.  $\td, \bud, \bmd$ and $\bhd$ will converge strongly in $L^2([0,T]\times\dom)$ along a subsequence using the Aubin-Lions lemma and the trick with renormalized solutions for $\bmd$ and $\bhd$ as earlier.  $\psi^{\delta}$ converges weakly in $L^2(0,T;H^1(\dom))$. So the remaining problematic terms are again the 'capillary force' $\varepsilon \lambda\Div(\Grad\td\otimes\Grad\td)$ and the term $\mu_0(\chi^\delta_{\td})' |\bmd|^2/(\chi_{\td}^\delta)^2$. We start by showing convergence of the latter and then use that to show convergence of the capillary force term. First we recall that if the initial data satisfies
\begin{equation*}
 \int_{\dom} \left[\frac{1}{2}|\bu_0|^2 + \frac{\mu_0}{2\chi^\delta_{\theta_0}}|\bm_0|^2 + \frac{\mu_0}{2}|\bh_0|^2 +\eps\lambda|\Grad\theta_0|^2+\frac{\lambda}{\eps}F(\theta_0) \right] dx\leq C
\end{equation*}
We assume that $C$ is independent of $\delta$. Then for fixed $t>0$ thanks to the energy estimate, we have
\begin{equation*}
	 \int_{\dom} \left[\frac{1}{2}|\bud(t)|^2 + \frac{\mu_0}{2\chi^\delta_{\td(t)}}|\bmd(t)|^2 + \frac{\mu_0}{2}|\bhd(t)|^2 +\eps\lambda|\Grad\td(t)|^2+\frac{\lambda}{\eps}F(\td(t)) \right] dx\leq C
\end{equation*}
for all $t>0$ with $C$ independent of $\delta$.  In particular, $\bud,\td,\Grad\td,\bmd,\bhd\in L^\infty(0,T;L^2(\dom))$.
In addition, we obtain from the energy estimate
\begin{equation*}
	\int_0^T \int_{\dom}\left( \nu_{\td}|T(\bud)|^2 +\kappa |\Grad\psi^\delta|^2+\frac{\mu_0}{\tau}\frac{1}{\chi^\delta_{\td}}|\bmd-\chi^\delta_{\td}\bhd|^2 \right)dx ds\leq C.
 \end{equation*}
 The time continuity estimates~\eqref{eq:timedertheta},~\eqref{eq:utimecont} and~\eqref{eq:mtimecont} still apply, as they are uniform with respect to $\delta>0$. Therefore, we can use the Aubin-Lions Lemma to conclude that $\bud\to \bu$ in $L^2([0,T]\times\dom)$ and $\td\to\theta$ in $C(0,T;L^2(\dom))$ as $\delta\to 0$ up to a subsequence which we will still denote by $\delta$. To derive convergence of $\bmd$ and $\bhd$ in $L^2([0,T]\times\dom)$, we use the same argument in equation~\eqref{eq:midentity}--\eqref{eq:midentityfinal} as above. Going through those calculations, we see that they are valid for any value of $\delta$.  Thus $\bmd\to \bm$, $\bhd\to \bh$ in $L^2([0,T]\times\dom)$.  This allows us to pass to the limit in the weak formulation of~\eqref{eq:CHF} except those involving products of gradients of $\theta$ and divisions by $\chi_\theta$. In order to pass to the limit in those terms, we improve the strong convergence of $\bmd$ in $L^2([0,T]\times\dom)$ to strong convergence in $L^{\frac{12}{5}}([0,T]\times\dom)$. Going through the proof of Proposition~\ref{prop:existence}, we note that under the assumption that $\bm_0\in L^{\frac{12}{5}+\alpha}(\dom)$, we can derive a uniform $L^\infty(0,T;L^{\frac{12}{5}+\alpha}(\dom))$-estimate on $\bm^{\delta}$. To be specific, in Section~\ref{sec:Nth-limit}, we derived an $L^\infty(0,T;L^4(\dom))$-estimate on $\bmb$. Since $\bmb$ is smooth, we could alternatively take $w=|\bmb|^{\frac{2}{5}+\alpha}\bmb$ as a test function in~\eqref{eq:meqpointwise} and obtain after integrating in time, and using H\"older and Young's inequality,
 \begin{align*}
 	&\frac{5}{12+5\alpha}\norm{\bmb(t)}_{L^{\frac{12}{5}+\alpha}}^{\frac{12}{5}+\alpha}-\frac{5}{12+5\alpha}\norm{\bmc_0}_{L^{\frac{12}{5}+\alpha}}^{\frac{12}{5}+\alpha}\\
 	& =-\frac{1}{\tau}\int_0^t\left( \norm{\bmb(s)}_{L^{\frac{12}{5}+\alpha}}^{\frac{12}{5}+\alpha}-\int_{\dom}\chi_{\thb}\bmb\cdot\bhb |\bmb|^{\frac{2}{5}+\alpha} dx \right)ds\\
 	 & \leq -\frac{5}{(12+5\alpha)\tau}\int_0^t\norm{\bmb(s)}_{L^{\frac{12}{5}+\alpha}}^{\frac{12}{5}+\alpha}ds +\frac{5\chi_0^{\frac{12}{5}+\alpha}}{(12+5\alpha)\tau}\int_0^t \norm{\bhb(s)}_{L^{\frac{12}{5}+\alpha}}^{\frac{12}{5}+\alpha}ds.
 \end{align*} 
 Now using elliptic regularity to estimate $\norm{\bhb}_{L^{\frac{12}{5}+\alpha}}\leq C \norm{\bmb}_{L^{\frac{12}{5}+\alpha}}$ and then Gr\"onwall's inequality, we obtain
 \begin{equation*}
 	\norm{\bmb(t)}_{L^{\frac{12}{5}+\alpha}}^{\frac{12}{5}+\alpha}\leq C_T \norm{\bm_0^{N_u}}_{L^{\frac{12}{5}+\alpha}}^{\frac{12}{5}+\alpha}\leq C,
 \end{equation*}
 since we are assuming that $\bm_0\in L^{\frac{12}{5}+\alpha}(\dom)$ and we can choose the approximation $\bm_0^{N_u}$ such that it is uniformly in $L^{\frac{12}{5}+\alpha}(\dom)$ as well. This estimate is uniform in $N_\theta$ and $N_u$ and therefore it holds that $\bmd\in L^\infty(0,T;L^{\frac{12}{5}+\alpha}(\dom))$ uniformly in $\delta$. Using this estimate, we can improve the strong $L^2$-convergence of $\bmd$ to $L^{12/5}([0,T]\times\dom)$-convergence as follows. We have, with H\"older's inequality,
 \begin{align*}
 	\int_0^T\int_{\dom} |\bmd-\bm|^{12/5} dx dt& \leq \norm{\bmd-\bm}_{L^2([0,T]\times\dom)}^{\frac{10\alpha}{2+5\alpha}}\norm{\bmd-\bm}_{L^{\frac{12}{5}+\alpha}([0,T]\times\dom)}^{\frac{24+10\alpha}{10+25\alpha}}\\
 	&\leq C\left(\norm{\bmd}_{L^{\frac{12}{5}+\alpha}}^{\frac{24+10\alpha}{10+25\alpha}}+\norm{\bm}_{L^{\frac{12}{5}+\alpha}}^{\frac{24+10\alpha}{10+25\alpha}}\right)\norm{\bmd-\bm}_{L^2([0,T]\times\dom)}^{\frac{10\alpha}{2+5\alpha}}\stackrel{\delta\to 0}{\longrightarrow} 0. 	
 \end{align*}
 In particular, this implies that $|\bmd|^2\to |\bm|^2$ in $L^{6/5}([0,T]\times\dom)$. We use this to derive convergence of the capillary force term $\Grad\td\otimes\Grad\td$. To do so, we repeat the argument from~\eqref{eq:psilim} --  \eqref{eq:convofnorms}. All terms are treated in the same way, except when passing to the limit in the term
 \begin{equation*}
 	\int_0^T \left(\frac{\mu_0(\chi_{\td}^\delta)'}{2(\chi_{\td}^\delta)^2}|\bmd|^2,\td\right) dt,
 \end{equation*}
 we cannot use that $\chi_{\td}^\delta$ is lower bounded which would imply an upper bound on its inverse. However, we can use the strong $L^{12/5}$-convergence of $\bmd$ that we have just derived. Since $\td\in L^\infty(0,T;H^1(\dom))$, by the Sobolev embedding theorem, we have that $\td\in L^\infty(0,T;L^6(\dom))\subset L^6([0,T]\times\dom)$ uniformly in $\delta$ and therefore $\td\weak \theta$ in $L^6([0,T]\times \dom)$ up to a subsequence. Since $|\bmd|^2\rightarrow |\bm|^2$ in $L^{6/5}([0,T]\times\dom)$ as $\delta\to 0$ up to a subsequence, this implies that $\td|\bmd|^2\weak \theta|\bm|^2$ in $L^1([0,T]\times\dom)$.Since,  at the same time  $\td\to \theta$ strongly in $C([0,T];L^2(\dom))$ and therefore almost everywhere (after extraction of a subsequence), we can use Lemma~\ref{lem:kennethlemma} and the uniform boundedness of $(\chi^\delta_{\td})'/(\chi_{\td}^\delta)^2$ to conclude that
 \begin{equation*}
 		\int_0^T \left(\frac{\mu_0(\chi_{\td}^\delta)'}{2(\chi_{\td}^\delta)^2}|\bmd|^2,\td\right) dt \stackrel{\delta\to 0}{\weak } 	\int_0^T \left(\frac{\mu_0\chi_{\theta}'}{2\chi_{\theta}^2}|\bm|^2,\theta\right) dt,
 \end{equation*}
 (up to a subsequence). Thus, the argument from~\eqref{eq:psilim} --  \eqref{eq:convofnorms} can be repeated and we obtain strong convergence of $\Grad\td$ to $\Grad\theta$ in $L^2([0,T]\times\dom)$. Similarly, we have that the term in the equation for $\psi^\delta$
  \begin{equation*}
 	\int_0^T  \frac{(\chi_{\td}^\delta)'}{(\chi_{\td}^\delta)^2}|\bmd|^2 g dx dt \stackrel{\delta\to 0}{\longrightarrow} 	\int_0^T  \frac{\chi_{\theta}'}{\chi_{\theta}^2}|\bm|^2 g dx dt
 \end{equation*}
 for any test function $g\in L^\infty([0,T]\times\dom)$ thanks to the uniform boundedness of $(\chi^\delta_{\td})'/(\chi_{\td}^\delta)^2$, the strong convergence of $\td$ in $C([0,T];L^2(\dom))$ and the strong convergence of $\bmd$ in $L^2([0,T]\times\dom)$. This concludes the proof of existence of weak solutions for~\eqref{eq:CHF} with possibly vanishing susceptibility, Theorem~\ref{thm:existence}.

\section{Vanishing relaxation time $\tau\to 0$: Proof of Theorem \ref{thm:tauzerolimit}}\label{sec:relaxation}

We recall that the limiting system in the limit $\tau \to 0$ is given by
\begin{subequations}\label{eq:CHFtau1}
	\begin{align}
		\Bu_t+(\Bu\cdot\Grad)\Bu+\Grad P&=\Div(\nu_\Theta T(\Bu))+\mu_0(\Bm\cdot\Grad)\Bh-\eps\lambda\Delta\Theta\Grad\Theta,\label{eq:nse1T}\\
		\Div\Bu&=0,\\
		0&=\Bm-\chi_\Theta\Bh,\label{eq:eqmT}\\
		\Curl\Bh &=0,\quad \Div(\Bm+\Bh)=0,\label{eq:maxwellT}\\
		\Theta_t+(\Bu\cdot\Grad)\Theta&=\kappa\Delta\Psi,\label{eq:phasefield1T}\\
		\Psi&= -\eps\lambda\Delta\Theta+\frac{\lambda}{\eps}F'(\Theta)-\frac{\mu_0}{2 }\chi'_\Theta|\Bh|^2.\label{eq:psiT}
	\end{align}
\end{subequations}
In this section, we will prove Theorem \ref{thm:tauzerolimit}, which makes this limit mathematically rigorous. We denote the solution of~\eqref{eq:CHF} as $\mathcal{U}^\tau=(\but,\tht,\bmt,\bht,\psit)$ to make its dependence on $\tau$ explicit for this section. For the susceptibility $\chi_\theta$ we assume at least that
\begin{equation}
	\label{eq:assonchi}
|\chi_\theta|\leq C,\quad 	|\chi_\theta'|\leq C,\quad |\chi_\theta' \theta|\leq C,\quad \left| \frac{\chi_\theta'}{\chi_{\theta}}\right|\leq C, \quad \left| \frac{\chi_\theta' \theta}{\chi_{\theta}}\right|\leq C,\quad \forall \theta\in \R.
	\end{equation}
As a consequence of Lemma \ref{lem:heavyside}, considering the limit $\delta \to 0$, these bounds hold for our prototypical choice of $\chi_\theta$ in \eqref{eq:chidef}. Our final goal is to prove Theorem \ref{thm:tauzerolimit}, which we restate here:
\tauzerolimit*
\begin{proof}
	This proof follows closely that of~\cite[Theorem 4.1]{Nochetto2019}.
	First, we note that the a priori energy estimate~\eqref{eq:energy0} holds for any $\tau>0$ and therefore implies the uniform in $\tau$ bounds
	\begin{align*}
		&\{\but\}_{\tau>0}\subset L^\infty(0,T;\hdiv)\cap L^2(0,T;\V)\\
		&\{\tht\}_{\tau>0}\subset L^\infty(0,T;H^1(\dom))\\
		&\{\bmt\}_{\tau>0}\subset L^\infty(0,T;L^2(\dom))\\
		&\{\bht\}_{\tau>0}\subset L^\infty(0,T;L^2(\dom))\\
		&\{\psit\}_{\tau>0}\subset L^2(0,T;H^1(\dom)).
	\end{align*}
	In addition, the bounds on the time derivatives on $\but$ and $\tht$, equations~\eqref{eq:timedertheta} and~\eqref{eq:utimecont}, still hold uniformly in $\tau$ and we additionally get the bound
	\begin{equation}\label{eq:relax}
		\frac1{\tau^{1/2}}\norm{\bmt-\chi_{\tht} \bht}_{L^2([0,T]\times\dom)}\leq C,
	\end{equation}
	where $C$ is independent of $\tau$. Thus, we have by Aubin-Lions-Simon Lemma that $\but$   converges strongly up to a subsequence in $L^2([0,T]\times\dom)$ to a limit $\Bu$ and that $\tht$ converges strongly up to a subsequence in $C(0,T;L^2(\dom))$ to a limit $\Theta$. We also obtain that $\psit\weak \Psi$ in $L^2(0,T;H^1(\dom))$ as $\tau\to 0$ and $\bmt\weakstar \Bm$, $\bht\weakstar\Bh$ in $L^\infty(0,T;L^2(\dom))$ as $\tau\to 0$, up to a subsequence. As before, we aim to derive strong convergence of the sequences $\{\bmt\}$ and $\{\bht\}$. First, we observe that the limit $\Bh$ of the sequence $\{\bht\}$ satisfies the weak form of equation~\eqref{eq:maxwellT} in the limit, due to the linearity of the equation:
	\begin{equation}\label{eq:Hweak}
		\int_0^T\int_{\dom}(\Bh+\Bm)\cdot \Grad \phi dx = \int_0^T\int_{\partial\dom} \bh_a\cdot \mathbf{n} \phi dS. 
	\end{equation}
	Furthermore, multiplying the equation for $\bmt$,  equation~\eqref{eq:mweak} by $\tau$ and sending $\tau \to 0$, using the strong convergence of $\but$ and $\tht$ and the weak convergence of $\bht$ and $\bmt$, we obtain in the limit
	\begin{equation}
		\label{eq:Mweak}
		\int_0^T\int_{\dom} \Bm\cdot v_3 dx dt = \int_0^T\int_{\dom} \chi_{\Theta} \Bh \cdot v_3 dx dt.
	\end{equation} 
	This can be extended to test functions $v_3\in L^1(0,T;L^2(\dom))$ due to the bounds on $\Bm$, $\Bh$ and $\chi_\theta$. Using this in~\eqref{eq:Hweak}, we obtain
		\begin{equation*}%\label{eq:Hweak}
		\int_0^T\int_{\dom}(1+\chi_\Theta)\Bh\cdot \Grad \phi dxdt  = \int_0^T\int_{\partial\dom} \bh_a\cdot\mathbf{n} \phi dSdt . 
	\end{equation*}
	Using $\phi=\Phi$ as a test function, where $\Grad\Phi =\Bh$, we have
		\begin{equation}\label{eq:Hweak1}%\label{eq:Hweak}
		\int_0^T\int_{\dom}(1+\chi_\Theta)|\Bh|^2  dx dt = \int_0^T\int_{\partial\dom} \bh_a\cdot \mathbf{n} \Phi dSdt . 
	\end{equation}
	On the other hand, taking $\phi=\varphi^\tau$ as a test function in the equation for $\bht$, we have
	\begin{equation}\label{eq:taupos}
		\int_0^T\int_{\dom} \bht\cdot (\bht+\bmt)dx dt = \int_0^T\int_{\partial\dom}\bh_a\cdot\mathbf{n} \pht dSdt
	\end{equation}
	Sending $\tau\to 0$, on the right hand side, we obtain 
	\begin{equation}\label{eq:taurhs}
		 \int_0^T\int_{\partial\dom}\bh_a\cdot\mathbf{n} \pht dSdt\stackrel{\tau\to 0}{\longrightarrow}  \int_0^T\int_{\partial\dom}\bh_a\cdot\mathbf{n} \Phi dSdt,
	\end{equation}
	using the weak convergence of $\pht$ in $L^2(0,T;H^1(\dom))$. We rewrite the left hand side of~\eqref{eq:taupos}, and then send $\tau\to 0$:
	\begin{equation*}
			\int_0^T\int_{\dom} \bht\cdot (\bht+\bmt)dx dt = 	\int_0^T\int_{\dom}(1+\chi_{\tht})|\bht|^2 dx dt + \int_0^T\int_{\dom} \bht\cdot (\bmt-\chi_{\tht}\bht)dx dt. 
	\end{equation*}
	Sending $\tau\to 0$ on both sides, and using~\eqref{eq:relax} on the right hand side, we obtain (using bars to denote the weak limits)
	\begin{equation*}
		\int_0^T\int_{\dom} \overline{\Bh\cdot(\Bm+\Bh)} dx dt =	\int_0^T\int_{\dom} (1+\chi_{\Theta})\overline{|\Bh|^2} dx dt +\lim_{\tau\to 0}\int_0^T\int_{\dom} \bht\cdot (\bmt-\chi_{\tht}\bht)dx dt,
	\end{equation*}
	where
	\begin{equation*}
		\left|\lim_{\tau\to 0}\int_0^T\int_{\dom} \bht\cdot (\bmt-\chi_{\tht}\bht)dx dt\right|\leq \lim_{\tau\to 0} \norm{\bht}_{L^2}\norm{\bmt-\chi_{\tht}\bht}_{L^2([0,T]\times\dom)} \leq C \lim_{\tau\to 0}\tau^{1/2}=0,
	\end{equation*}
	thanks to~\eqref{eq:relax}. Hence, we have
		\begin{equation}\label{eq:bhtconv1}
		\int_0^T\int_{\dom} \overline{\Bh\cdot(\Bm+\Bh)} dx dt =	\int_0^T\int_{\dom} (1+\chi_{\Theta})\overline{|\Bh|^2} dx dt,
	\end{equation}
	in the limt as $\tau\to 0$. Combining this with~\eqref{eq:taurhs} and~\eqref{eq:Hweak1}, we obtain
	\begin{equation}\label{eq:bhtconv2}
		\int_0^T\int_{\dom}(1+\chi_{\Theta})\overline{|\Bh|^2} dx dt=\lim_{\tau\to 0} \int_0^T\int_{\dom}(1+\chi_{\tht})|\bht|^2 dx dt =\int_0^T\int_{\dom}(1+\chi_{\Theta}){|\Bh|^2} dx dt.
	\end{equation}
	Since $\tht\to \Theta$ in $L^2([0,T]\times\dom)$ and therefore almost everywhere, we have, since $\chi_\theta$ is bounded and $\Bh\in L^2([0,T]\times\dom)$,
	\begin{equation}\label{eq:bhtconv3}
		\lim_{\tau\to 0}\int_0^T\int_{\dom}\chi_{\tht}|\Bh|^2 dx dt = \int_0^T\int_{\dom} \chi_{\Theta} |\Bh|^2 dx dt.
	\end{equation}
	Next using Lemma~\ref{lem:kennethlemma} for $\chi_{\tht}\to \chi_{\Theta}$ a.e. and $\bht\cdot \Bh\weak |\Bh|^2$ in $L^1([0,T]\times\dom)$, we have
	\begin{equation}\label{eq:bhtconv4}
		\int_0^T\int_{\dom} (1+\chi_{\tht})\Bh \cdot \bht dx dt \stackrel{\tau\to 0}{\longrightarrow}\int_0^T\int_{\dom} (1+\chi_{\Theta})|\Bh|^2 dx dt.
	\end{equation}
	Combining~\eqref{eq:bhtconv2}, \eqref{eq:bhtconv3} and~\eqref{eq:bhtconv4}, we therefore obtain
	\begin{align*}
	\lim_{\tau\to 0}	\int_0^T\int_{\dom} |\bht-\Bh|^2 dx dt & \leq \lim_{\tau\to 0}	\int_0^T\int_{\dom} (1+\chi_{\tht})|\bht-\Bh|^2 dx dt\\
	& = \lim_{\tau\to 0}\Bigg(\int_0^T\int_{\dom}(1+\chi_{\tht})|\Bh|^2 dx dt +\int_0^T\int_{\dom}(1+\chi_{\tht})|\bht|^2 dx dt\\
	&\hphantom{ = \lim_{\tau\to 0}\Bigg(}\quad -2\int_0^T\int_{\dom}(1+\chi_{\tht})\Bh\cdot\bht dx dt\Bigg) =0,
	\end{align*}
	therefore, $\bht\to \Bh$ strongly in $L^2([0,T]\times\dom)$ up to a subsequence, as $\tau\to 0$. Next, we derive strong convergence of $\bmt$ from the strong convergence of $\bht$. We have already seen in~\eqref{eq:Mweak} that $\bmt\weak \chi_{\Theta}\Bh$ in $L^2([0,T]\times\dom)$. Then, using the strong convergence of $\bht$ and $\tht$ and the bound~\eqref{eq:relax}, we have
	\begin{equation*}
		\lim_{\tau\to 0}\norm{\bmt-\chi_{\Theta}\Bh}_{L^2}\leq \lim_{\tau\to 0} \norm{\bmt -\chi_{\tht}\bht}_{L^2}+\lim_{\tau\to 0}\norm{\chi_{\tht} \bht - \chi_{\Theta}\Bh}_{L^2} =0,
	\end{equation*}
	hence also $\bmt$ converges strongly to $\Bm$ in $L^2([0,T]\times\dom)$. Now we can pass to the limit in all the terms of the weak formulation of~\eqref{eq:CHF}. To pass to the limit in the term
	\begin{equation*}
	\int_0^T\left(	\frac{\mu_0 \chi_{\tht}'}{\chi_{\tht}^2}|\bmt|^2, z\right)dt
	\end{equation*}
	we rewrite it as
	\begin{multline*}
			\int_0^T\left(	\frac{\mu_0 \chi_{\tht}'}{\chi_{\tht}^2}|\bmt|^2, z\right)dt = 	\underbrace{\int_0^T\left(	\frac{\mu_0 \chi_{\tht}'}{\chi_{\tht}^2}|\bmt-\chi_{\tht}\bht|^2, z\right)dt }_{(i)}- \underbrace{\int_0^T\left(	 {\mu_0 \chi_{\tht}'} | \bht|^2, z\right)dt}_{(ii)} 
			 + 	\underbrace{\int_0^T\left(	\frac{2\mu_0 \chi_{\tht}'}{\chi_{\tht}}\bmt\cdot\bht, z\right)dt }_{(iii)}.
	\end{multline*}
	For the first term $(i)$, we use Assumption~\eqref{eq:assonchi} and that from the energy balance
	\begin{equation}\label{eq:tauest}
		\int_0^T\int_{\dom} \frac{\left|\bmt-\chi_{\tht}\bht\right|^2 }{\chi_{\tht}}dx dt\leq C \tau,
	\end{equation}
	thus it goes to zero as $\tau\to 0$. For the second term $(ii)$, we use the boundedness of $\chi'_\theta$,~\eqref{eq:assonchi} and the strong convergence of $\bht$ to $\Bh$ in $L^2([0,T]\times\dom)$ to conclude that
	\begin{equation*}
		(ii)\stackrel{\tau\to 0}{\longrightarrow}\int_0^T\mu_0(\chi_\theta' |\Bh|^2,z)dt.
	\end{equation*}
	For the third term, we use again Assumption~\eqref{eq:assonchi} combined with the strong convergence of $\bmt$ and $\bht$ to conclude that
	\begin{equation*}
		(iii)\stackrel{\tau\to 0}{\longrightarrow}2\int_0^T\mu_0(\chi_\theta' |\Bh|^2,z)dt.
	\end{equation*}
	Thus, we have
	\begin{equation*}
			\int_0^T\left(	\frac{\mu_0 \chi_{\tht}'}{\chi_{\tht}^2}|\bmt|^2, z\right)dt \stackrel{\tau\to 0}{\longrightarrow}\int_0^T\mu_0(\chi_\theta' |\Bh|^2,z)dt,
	\end{equation*}
	as desired. To obtain strong convergence of $\Grad\tht$, we need to repeat the argument from~\eqref{eq:psilim}--\eqref{eq:convofnorms}. As in the previous section~\ref{sec:vanishing-mag}, the main challenge is to pass to the limit in the last term in the equation 
	\begin{equation}\label{eq:psilimapprox2}
		\int_0^T(\psi^{\tau},\tht)dt =\int_0^T\left(\lambda \eps(\Grad\tht,\Grad \tht) +\frac{\lambda}{\eps}(F'(\tht),\tht)-\left(\frac{\mu_0}{2\chi_{\tht}^2}\chi'_{\tht}|\bmt|^2,\tht\right)\right) dt
	\end{equation}
	(which corresponds to~\eqref{eq:psilimapprox} above.) The left hand side term converges to $\int_0^T(\Psi,\Theta) dt$ thanks to the weak convergence of $\psit$ in $L^2([0,T]\times\dom)$ and the strong convergence of $\tht$ in $L^2([0,T]\times\dom)$ and the second term on the right hand side converges to $\frac{\lambda}{\epsilon}\int_0^T (F'(\Theta),\Theta) dt$ thanks to the strong convergence of $\tht$ and the bounds on $F$,~\eqref{eq:Fass}. Therefore, we decompose the last term as follows:
	\begin{align*}
	\int_0^T	\left(\frac{\mu_0}{2\chi_{\tht}^2}\chi'_{\tht}|\bmt|^2,\tht\right)dt & = \underbrace{	\int_0^T	\left(\frac{\mu_0\chi'_{\tht}\tht}{2\chi_{\tht}},\frac{|\bmt-\chi_{\tht}\bht|^2}{\chi_{\tht}}\right) dt}_{(i)} \\
	&- \underbrace{	\int_0^T	\left(\frac{\mu_0}{2}\chi'_{\tht}\tht,| \bht|^2\right) dt}_{(ii)} + \underbrace{	\int_0^T	\left(\frac{\mu_0\chi'_{\tht}}{ \chi_{\tht}}\bmt\cdot\bht ,\tht\right) dt}_{(iii)} .
	\end{align*}
	Term (i) goes to zero thanks to the estimate~\eqref{eq:tauest} and the assumed bounds on $\chi_\theta$,~\eqref{eq:assonchi},
	\begin{equation*}
		|(i)|  \leq \frac{\mu_0}{2}\sup_{\tht\in \R}\left|\frac{ \chi'_{\tht}\tht}{ \chi_{\tht}}\right|	\int_0^T	 \int_{\dom}\frac{|\bmt-\chi_{\tht}\bht|^2}{\chi_{\tht}}dx dt \leq C \sqrt{\tau}.
	\end{equation*}
	For the second term, we use the bounds on $\chi_\theta$,~\eqref{eq:assonchi}, combined with the strong convergence of $\bht$ in $L^2([0,T]\times\dom)$ to apply Lemma~\ref{lem:kennethlemma} and conclude that
	\begin{equation*}
		(ii) \stackrel{\tau\to0}{\longrightarrow} \int_0^T \left(\frac{\mu_0}{2}\chi_{\Theta}' \Theta,|\Bh|^2\right) dt.
	\end{equation*}
	Finally, for (iii), we use again the bounds on $\chi_\theta$ and that $\bmt\cdot\bht\to \Bm\cdot\Bh = \chi_{\Theta} |\Bh|^2$ in $L^1([0,T]\times\dom)$ to apply Lemma~\ref{lem:kennethlemma} and conclude that
	\begin{equation*}
		(iii) \stackrel{\tau\to0}{\longrightarrow} \int_0^T \left(\mu_0\chi_{\Theta}' \Theta,|\Bh|^2\right) dt.
	\end{equation*}
	Thus, 
	\begin{equation*}
			\int_0^T	\left(\frac{\mu_0}{2\chi_{\tht}^2}\chi'_{\tht}|\bmt|^2,\tht\right)dt \stackrel{\tau\to 0}{\longrightarrow} \int_0^T \left(\frac{\mu_0}{2}\chi_{\Theta}' \Theta,|\Bh|^2\right) dt,
	\end{equation*}
	and we can proceed as in~\eqref{eq:psilim}--\eqref{eq:convofnorms} to conclude that $\Grad\tht\to \Grad\Theta$ in $L^2([0,T]\times\dom)$ up to a subsequence.	 
	 Thus the limit $(\Bu,\Theta,\Bm,\Bh,\Psi)$ is a weak solution of~\eqref{eq:CHFtau}.

\end{proof}

\appendix

\section{Derivation of the diffuse interface ferrofluids model}
\label{sec:derivation}
To derive the system~\eqref{eq:CHF}, we follow roughly~\cite{Nochetto2014} which is based on using Onsager's variational principle~\cite{Onsager1931,Onsager1931b}. We assume that the variables $\bu$, $\theta$ and $\bm$ satisfy the Navier-Stokes equations and advection equations of the form
\begin{align}
	\bu_t +(\bu\cdot\Grad)\bu +\Grad p &= \Div(\nu_\theta T(u))+\Fu,\label{eq:uansatz}\\
	\Div u & = 0,\\
	\bm_t +(\bu\cdot\Grad)\bm& = \Fm,\label{eq:mansatz}\\
	\theta_t+ \bu\cdot \Grad \theta & = -\Div \Jtheta,\label{eq:thetaansatz}
\end{align}
with forces and fluxes $\Fu, \Jtheta$ and $\Fm$ to be determined. This means that both the phase field variable $\theta$ and the magnetization $\bm$ are convected by the fluid. Furthermore, this also implies that we assume that the fluid is incompressible and both phases have the same density. As it was derived in previous works~\cite{Rosensweig1985,Rosensweig1985a,Odenbach2002}, the magnetic field $\bh$ satisfies a constitutive relation of the form
\begin{equation}\label{eq:bfield}
	\mathbf{b}= \bm + \bh,
\end{equation}
and the magnetostatics equations
\begin{equation}\label{eq:hfield}
	\Div \mathbf{b}=0,\quad \Curl\bh = 0.
\end{equation}
As specified above in~\eqref{eq:decomph}, the magnetic field is a sum of the demagnetizing field $\bh_d$ and the applied field $\bh_a$ which we assume is not perturbed by $\bh_d$ and $\bm$ (this is a simplifying assumption). We assume that $\bh_a$ is a given smooth harmonic field and that $\bh_d\approx 0$ outside $\dom$. We assume that the boundary is impermeable and that there is no slip, resulting in homogeneous Dirichlet boundary conditions for $\bu$, i.e., $\bu=0$ on $\partial\dom$ and no flux boundary conditions for $\Jtheta$, i.e., $\Jtheta\cdot \mathbf{n} = 0$ on $\partial\dom$. Under the assumptions above, the boundary conditions for $\bh$ become~\eqref{eq:bcelliptic}. We assume that the free energy of the system is given by
\begin{equation}
	\label{eq:freeenergy}
	\mathcal{E}(\sol) = \int_{\dom} \left(\frac12|\bu|^2 + \lambda\left(\frac{\varepsilon}{2}|\Grad\theta|^2 +\frac{1}{\varepsilon}F(\theta)\right)+\frac{\mu_0}{2\chi_\theta}|\bm|^2 +\frac{\mu_0}{2}|\bh|^2\right) dx.
\end{equation}
The first term is the kinetic energy of the fluid, the second two terms are the Cahn-Hilliard energy measuring the system's tendency to minimize energy by separating into different phases, where $\lambda$ is the surface tension between the phases, the third term is the magnetization energy density and the final term is the energy density of the magnetic field. We compute the time derivative of the energy and plug in the equations for the evolution of $\bu$, $\bm$, $\theta$ and $\bh$, \eqref{eq:uansatz}, \eqref{eq:mansatz}, \eqref{eq:thetaansatz}, \eqref{eq:hfield}, the boundary conditions, and use that $\bu$ is divergence free:
\begin{align*}
	\frac{d\mathcal{E}}{dt} & = \int_{\dom} \left(\bu\cdot \bu_t + \lambda \varepsilon \Grad\theta\Grad\theta_t + \frac{\lambda}{\varepsilon}F'(\theta)\theta_t +\frac{\mu_0}{\chi_\theta}\bm\cdot \bm_t -\frac{\mu_0 \chi_\theta'}{2\chi_\theta^2 }\theta_t |\bm|^2 +\mu_0 \bh \bh_t\right) dx \\
	& = \int_{\dom} \Bigg(\bu\cdot \left[-(\bu\cdot\Grad)\bu -\Grad p +\Div(\nu_\theta T(u) ) + \Fu\right] - \lambda \varepsilon \Delta\theta \theta_t + \frac{\lambda}{\varepsilon}F'(\theta)\theta_t +\frac{\mu_0}{\chi_\theta}\bm\cdot\left[-(\bu\cdot\Grad)\bm + \Fm\right]\\
	&\qquad \qquad -\frac{\mu_0 \chi_\theta'}{2\chi_\theta^2 }\theta_t |\bm|^2 -\mu_0 \varphi \Div \bh_t\Bigg) dx +\mu_0\int_{\partial\dom} \varphi (\bh_a-\bm)_t \cdot\mathbf{n} dS\\
	& = \int_{\dom} \Bigg(- \nu_\theta |T(u)|^2 + \bu\cdot \Fu - \lambda \varepsilon \Delta\theta \left[-\bu\cdot \Grad \theta -\Div\Jtheta\right] + \frac{\lambda}{\varepsilon}F'(\theta)\left[-\bu\cdot \Grad \theta -\Div\Jtheta\right]  \\
	&\qquad \qquad +\frac{\mu_0}{\chi_\theta}\bm\cdot\left[-(\bu\cdot\Grad)\bm + \Fm\right]-\frac{\mu_0 \chi_\theta'}{2\chi_\theta^2 }\left[-\bu\cdot \Grad \theta -\Div\Jtheta\right]  |\bm|^2 -\mu_0 \bh  \bm_t\Bigg) dx  \\
	&\qquad
	  + \mu_0\int_{\partial\dom} \varphi (\bh_a)_t \cdot\mathbf{n} dS\\
	& = \int_{\dom} \Bigg(- \nu_\theta |T(u)|^2 + \bu\cdot \Fu + \lambda \varepsilon \Delta\theta \Grad \theta\cdot\bu -\lambda \varepsilon \Delta\Grad\theta\cdot\Jtheta + \frac{\lambda}{\varepsilon}\Grad F'(\theta)   \cdot \Jtheta -\frac{\mu_0}{2\chi_\theta}  (\bu\cdot\Grad)|\bm|^2  \\
	&\qquad \qquad  + \frac{\mu_0}{\chi_\theta}\bm\cdot \Fm -\frac{\mu_0 \chi_\theta'}{2\chi_\theta^2 }\left[-\bu\cdot \Grad \theta -\Div\Jtheta\right]  |\bm|^2 -\mu_0 \bh  \left[-(\bu\cdot\Grad)\bm + \Fm\right]\Bigg) dx  \\
	&\qquad  + \mu_0\int_{\partial\dom} \varphi (\bh_a)_t \cdot\mathbf{n} dS\\
	& = \int_{\dom} \Bigg(- \nu_\theta |T(u)|^2 + \bu\cdot \left(\Fu +\lambda \varepsilon \Delta\theta \Grad \theta  -\mu_0 (\bm\cdot\Grad)\bh \right) +\Fm\cdot\left( \frac{\mu_0}{\chi_\theta}\bm -\mu_0\bh\right)    \\
	&\qquad \qquad +\Jtheta\cdot\left(\frac{\lambda}{\varepsilon}\Grad F'(\theta) -\lambda \varepsilon \Delta\Grad\theta-\Grad\left(\frac{\mu_0 \chi_\theta'}{2\chi_\theta^2 }|\bm|^2\right) \right)  \Bigg) dx    + \mu_0\int_{\partial\dom} \varphi (\bh_a)_t \cdot\mathbf{n} dS
\end{align*}
The power of the system, i.e., the time derivative of the work $\mathcal{W}$ of internal forces, is the collection of the terms that have a scalar product with $\bu$. We assume that the system is closed, meaning there are no external forces. This implies that $\frac{d\mathcal{W}}{dt}=0$. Thus, we must have
\begin{equation*}
	\int_{\dom}  \bu\cdot \left(\Fu +\lambda \varepsilon \Delta\theta \Grad \theta  -\mu_0 (\bm\cdot\Grad)\bh \right)   dx =0, 
\end{equation*}
which allows us to identify $\Fu$,
\begin{equation}\label{eq:Fu}
	\Fu = -\lambda \varepsilon \Delta\theta \Grad \theta  +\mu_0 (\bm\cdot\Grad)\bh.
\end{equation}
The first law of thermodynamics says that the time derivative of the energy is equal to the power of the system minus the temperature $\mathcal{T}$ times the time derivative of the entropy $\mathcal{S}$, the entropy production,
\begin{equation*}
	\frac{d\mathcal{E}}{dt} = \frac{d\mathcal{W}}{dt}-\mathcal{T}\frac{d\mathcal{S}}{dt}.
\end{equation*}
The second law of thermodynamics states that the entropy of a closed system must increase or stay constant. Thus we must have
\begin{equation*}
	-\mathcal{T}\frac{d\mathcal{S}}{dt} =  \int_{\dom} \Bigg(- \nu_\theta |T(u)|^2  +\Fm\cdot\left( \frac{\mu_0}{\chi_\theta}\bm -\mu_0\bh\right)    +\Jtheta\cdot\left(\frac{\lambda}{\varepsilon}\Grad F'(\theta) -\lambda \varepsilon \Delta\Grad\theta-\Grad\left(\frac{\mu_0 \chi_\theta'}{2\chi_\theta^2 }|\bm|^2\right) \right)  \Bigg) dx    \leq 0.
\end{equation*}
In order to achieve this, $\Fm$ must be proportional to 
$$-\left( \frac{\mu_0}{\chi_\theta}\bm -\mu_0\bh\right) $$
and $\Jtheta$ must be proportional to 
$$-\Grad\left(\frac{\lambda}{\varepsilon}  F'(\theta) -\lambda \varepsilon \Delta \theta-  \frac{\mu_0 \chi_\theta'}{2\chi_\theta^2 }|\bm|^2 \right).$$
This results in the choices
\begin{equation*}
	\Jtheta = -\kappa \Grad\left(\frac{\lambda}{\varepsilon}  F'(\theta) -\lambda \varepsilon \Delta \theta-  \frac{\mu_0 \chi_\theta'}{2\chi_\theta^2 }|\bm|^2 \right),
\end{equation*}
or 
\begin{equation*}
	\theta_t +\bu\cdot\Grad\theta = \kappa\Delta\psi,\quad \psi = \frac{\lambda}{\varepsilon}  F'(\theta) -\lambda \varepsilon \Delta \theta-  \frac{\mu_0 \chi_\theta'}{2\chi_\theta^2 }|\bm|^2, 
\end{equation*}
and 
\begin{equation*}
	\Fm = -\frac{1}{\tau}(\bm-\chi_\theta\bh).
\end{equation*}
We choose $\Fm$ inversely proportional to the relaxation time $\tau$ as this term drives $\bm$ to be parallel to $\bh$ in equilibrium and multiply by $\chi_\theta$ so that the magnetic field acts on the magnetization when in the ferrofluid phase. In summary, we have obtained the system~\eqref{eq:CHF}.

\section{Result from Real Analysis}
We believe the following result is standard, but state it here for convenience.
\begin{lemma}\label{lem:kennethlemma}
	Let $u_n, v_n:\dom\to\R$ be sequences of measurable functions, such that
	\begin{align*}
	&u_n\rightarrow u,\quad \text{a.e. in } \dom, \quad \norm{u_n}_{L^\infty(\dom)}\leq C\, \quad\forall \, n\\
	&v_n\weak v,\quad \text{in } L^1(\dom), 
	\end{align*} 
	for some functions $u\in L^\infty(\dom)$, $v\in L^1(\dom)$. Then
	\begin{equation*}
	u_n v_n\rightharpoonup u\, v,\quad \text{in } L^1(\dom),\quad \text{as } n\rightarrow\infty.
	\end{equation*}
\end{lemma}
The proof can be found for example in~\cite[Lemma A.1]{efield}.

\bibliographystyle{abbrv}
\bibliography{diffuse}

\end{document}